\documentclass[reqno]{amsart}

\usepackage[latin1]{inputenc}
\usepackage[T1]{fontenc}
\usepackage[english]{babel}
\usepackage{latexsym}
\usepackage{amssymb}
\usepackage{mathrsfs}
\usepackage{enumitem}
\usepackage{mathtools}
\mathtoolsset{centercolon}

\numberwithin{equation}{section} 
\renewcommand{\Re}{\operatorname{Re}}

\setenumerate{leftmargin=*}
\allowdisplaybreaks
\newtheorem{theorem}{Theorem}[section]
\newtheorem{lem}[theorem]{Lemma}
\newtheorem{cor}[theorem]{Corollary}
\newtheorem{prop}[theorem]{Proposition}
\newtheorem{definition}[theorem]{Definition}

\theoremstyle{remark}
\newtheorem{rem}[theorem]{Remark}

\newcommand{\R}{\mathbb{R}}

\newcommand{\N}{\mathbb{N}}

\newcommand{\Ltwoa}{L^2(\Omega)}

\newcommand{\Ltwoh}{L^2(\R^3_+)}
\newcommand{\Ltwohdom}{L^2(G)}
\newcommand{\Ltwohn}[1]{\|#1\|_{\Ltwoh}}

\newcommand{\Ltwoan}[1]{\|#1\|_{L^2_\gamma(\Omega)}}

\newcommand{\Ltwohnt}[1]{\Gnorm{0}{#1}}

\newcommand{\Ltwodomh}{L^2(G)}
\newcommand{\Ltwodomhn}[1]{\|#1\|_{\Ltwodomh}}

\newcommand{\Lpmtwohdom}{L^2(G_\pm)}

\newcommand{\Hh}[1]{H^{#1}(\R^3_+)}

\newcommand{\Ha}[1]{H^{#1}(\Omega)}

\newcommand{\Hagamma}[1]{H^{#1}_\gamma(\Omega)}

\newcommand{\Hhta}[1]{H_{\operatorname{ta}}^{#1}(\R^3_+)}

\newcommand{\Hata}[1]{H_{\operatorname{ta}}^{#1}(\Omega)}

\newcommand{\Hatan}[2]{\|#2\|_{H_{\operatorname{ta},\gamma}^{#1}(\Omega)}}

\newcommand{\Hhn}[2]{\|#2\|_{\Hh{#1}}}

\newcommand{\Hangamma}[2]{\|#2\|_{\Hagamma{#1}}}

\newcommand{\Hpmhdom}[1]{\mathcal{H}^{#1}(G)}

\newcommand{\Hpmadom}[1]{\mathcal{H}^{#1}(J \times G)}

\newcommand{\Hpmhndom}[2]{\|#2\|_{\Hpmhdom{#1}}}

\newcommand{\Hpmandom}[2]{\|#2\|_{\mathcal{H}^{#1}(J \times G)}}

\newcommand{\Hpmangammadom}[2]{\|#2\|_{\mathcal{H}^{#1}_\gamma(J \times G)}}

\newcommand{\G}[1]{G_{#1}(\Omega)}

\newcommand{\Gnorm}[2]{\|#2\|_{G_{#1, \gamma}(\Omega)}}

\newcommand{\Gpmdom}[1]{\mathcal{G}_{#1}(J \times G)}
\newcommand{\GpmdomP}[1]{\mathcal{G}_{#1}(J_\tau \times G)}

\newcommand{\Gpmdomvar}[1]{\tilde{\mathcal{G}}_{#1}(J \times G)}
\newcommand{\GpmdomvarP}[1]{\tilde{\mathcal{G}}_{#1}(J_\tau \times G)}

\newcommand{\Gpmdomnorm}[2]{\|#2\|_{\mathcal{G}_{#1, \gamma}(J \times G)}}
\newcommand{\Gpmdomnormwg}[2]{\|#2\|_{\mathcal{G}_{#1}(J \times G)}}

\newcommand{\GpmdomnormwgP}[2]{\|#2\|_{\mathcal{G}_{#1}(J_\tau \times G)}}

\newcommand{\F}[1]{F_{#1}(\Omega)}

\newcommand{\Fpdw}[2]{F_{#1,#2}(\Omega)}

\newcommand{\Fuwl}[2]{F_{#1}^{\operatorname{#2}}(\Omega)}

\newcommand{\Fupdwl}[3]{F_{#1,#3}^{\operatorname{#2}}(\Omega)}

\newcommand{\Fpdk}[3]{F_{#1,#3,#2}(\Omega)}

\newcommand{\Fuwlk}[3]{F_{#1,#3}^{\operatorname{#2}}(\Omega)}

\newcommand{\Fnorm}[2]{\|#2\|_{F_{#1}(\Omega)}}

\newcommand{\Fvarwkdom}[2]{F^0_{#1,#2}(\tilde{G})}
\newcommand{\Fcoeff}[2]{F_{#1, \operatorname{coeff}}^{\operatorname{#2}}(\R^3_+)}

\newcommand{\Fvarnorm}[2]{\|#2\|_{F^0_{#1}(\R^3_+)}}
\newcommand{\Fvarnormdom}[2]{\|#2\|_{F^0_{#1}(\tilde{G})}}

\newcommand{\Fpmuwl}[2]{\mathcal{F}_{#1}^{\operatorname{#2}}(\Omega)}

\newcommand{\Fpmupdwl}[3]{\mathcal{F}_{#1,#3}^{\operatorname{#2}}(\Omega)}

\newcommand{\Fpmdom}[1]{\mathcal{F}_{#1}(J \times G)}

\newcommand{\Fpmdompdw}[2]{\mathcal{F}_{#1,#2}(J \times G)}

\newcommand{\Fpmdomupdwl}[3]{\mathcal{F}_{#1,#3}^{\operatorname{#2}}(J \times G)}

\newcommand{\Fpmdomuwlk}[3]{\mathcal{F}_{#1,#3}^{\operatorname{#2}}(J \times G)}

\newcommand{\Fpmdomupdwlk}[4]{\mathcal{F}_{#1,#4,#3}^{\operatorname{#2}}(J \times G)}

\newcommand{\Fpmvardom}[1]{\mathcal{F}^0_{#1}(G)}

\newcommand{\Fpmvarwkdom}[2]{\mathcal{F}^0_{#1,#2}(G)}
\newcommand{\Fpmcoeff}[2]{\mathcal{F}_{#1, \operatorname{coeff}}^{\operatorname{#2}}(\R^3_+)}

\newcommand{\Fpmvarnormdom}[2]{\|#2\|_{\mathcal{F}^0_{#1}(G)}}
\newcommand{\Fpmnormdom}[2]{\|#2\|_{\mathcal{F}_{#1}(J \times G)}}

\newcommand{\E}[1]{E_{#1}(J \times \partial \R^3_+)}
\newcommand{\Edom}[1]{E_{#1}(J \times \Sigma)}

\newcommand{\Enorm}[2]{\|#2\|_{E_{#1, \gamma}(J \times \partial \R^3_+)}}
\newcommand{\Enormwgdom}[2]{\|#2\|_{E_{#1}(J \times \Sigma)}}
\newcommand{\Enormdom}[2]{\|#2\|_{E_{#1, \gamma}(J \times \Sigma)}}

\newcommand{\Epmdom}[1]{E_{#1}(J \times \Sigma)}

\newcommand{\Epmnormdom}[2]{\|#2\|_{E_{#1, \gamma}(J \times \Sigma)}}

\newcommand{\ml}[3]{\mathcal{ML}^{#1}(#2,#3)}
\newcommand{\mlwl}[4]{\mathcal{ML}^{#1,\operatorname{#4}}(#2,#3)}
\newcommand{\mlpd}[3]{\mathcal{ML}_{\operatorname{pd}}^{#1}(#2,#3)}

\newcommand{\mlpm}[3]{\mathcal{ML}^{#1}(#2,#3)}
\newcommand{\mlwlpm}[4]{\mathcal{ML}^{#1,\operatorname{#4}}(#2,#3)}
\newcommand{\mlpdpm}[3]{\mathcal{ML}_{\operatorname{pd}}^{#1}(#2,#3)}
\newcommand{\mlpdwlpm}[4]{\mathcal{ML}_{\operatorname{pd}}^{#1, \operatorname{#4}}(#2,#3)}
\newcommand{\curl}{\operatorname{curl}}
\newcommand{\Div}{\operatorname{Div}}
\renewcommand{\div}{\operatorname{div}}

\newcommand{\bs}{\boldsymbol}
\newcommand{\Tr}{\operatorname{Tr}}
\newcommand{\tr}{\operatorname{tr}}

\newcommand{\clJ}{\overline{J}}

\newcommand{\dist}{\operatorname{dist}}

\newcommand{\supp}{\operatorname{supp}}

\newcommand{\image}{\operatorname{im}}

\renewcommand{\epsilon}{\varepsilon}
\renewcommand{\phi}{\varphi}

\title[Quasilinear Maxwell interface problems]{Local wellposedness of quasilinear Maxwell equations with conservative interface conditions}
\author{Roland Schnaubelt and Martin Spitz}
\email{schnaubelt@kit.edu, martin.spitz@kit.edu}
\address{Department of Mathematics, Karlsruhe Institute of Technology, 76128 Karlsruhe, Germany.}
\subjclass[2010]{35L50, 35L60, 35Q61}
\keywords{Nonlinear Maxwell system, perfectly conducting boundary/interface conditions,
local wellposedness, blow-up criterion, continuous dependence, piecewise regular}
\thanks{We gratefully acknowledge financial support by the Deutsche Forschungsgemeinschaft (DFG)
through CRC 1173.}

\begin{document}
\begin{abstract}
We establish a comprehensive local wellposedness theory for the quasilinear Maxwell system with interfaces
in the space of piecewise $H^m$-functions for $m \geq 3$. The system is equipped with
instantaneous and piecewise regular material laws and perfectly conducting interfaces and boundaries. 
We also provide a blow-up criterion in the  Lipschitz norm and prove the continuous dependence on the data.
 The proof relies on precise a priori estimates and the regularity theory for the corresponding linear 
 problem also shown here.
\end{abstract}
 \maketitle
 
 \section{Introduction}
 \label{SectionIntroduction}
 The Maxwell equations are the basis of electro-magnetic theory and thus one of the fundamental partial 
 differential equations in physics. In the case of  instantaneous nonlinear material laws, 
 they form a symmetric quasilinear hyperbolic system under natural assumptions. 
 For such systems on  $\R^d$, in \cite{KatoQuasilinearHyperbolicSystems}
 Kato has established a satisfactory local wellposedness theory in $H^s(\R^d)$ for $s > 1+\frac{d}{2}$.
  However, on a domain $G \neq \R^3$, the Maxwell system with the boundary conditions of a perfect conductor 
 has a characteristic boundary and 
  does not belong to the classes of hyperbolic systems  for which one knows a wellposedness theory in $H^3$.
 The available results need much more regularity and exhibit a loss of derivatives in normal direction (encoded in 
 weighted function spaces),  see~\cite{Gues} or \cite{Secchi}. In the recent papers \cite{SpitzQuasilinearMaxwell}
 and \cite{SpitzMaxwellLinear} by one of the authors, a comprehensive local wellposedness theory in $H^m$ for $m\ge 3$ 
 has been established for the boundary conditions of a perfect conductor. The main effort in these works is
devoted to prove full regularity in normal direction at the boundary, heavily using the structure of the Maxwell system.

 However,  deriving boundary conditions for the Maxwell systems on a domain $G\subseteq \R^3$, 
 one starts from the interface conditions \eqref{EquationJumpConditionsOnInterface} at $\partial G$ and \emph{assumes}
 that one knows the trace of the fields outside $G$, see~Section~I.4.2.2 of \cite{DautrayLionsI} or Section~7.12 in \cite{FabrizioMorro}.
 Moreover, in applications one often deals with composite materials in which the constitutive relations
 are only piecewise regular in $x\in G$. Here one has to treat the jumps in the material as interfaces.
 It is thus necessary to investigate interface problems in electro-magnetism, and not only (pure) boundary value problems. 
 
 In this work, we treat a (possibly unbounded) domain $G\subseteq \R^3$ being the disjoint union of two subdomains $G_+$ and $G_{-}$ 
 and the interface $\Sigma=\partial G_-$, where  $\Sigma$ and $\partial G$ are smooth and have positive distance. 
 Our results immediately extend to domains consisting  of finitely many such components. 
 We establish a  comprehensive local wellposedness theory in $H^m$ with $m\ge 3$ for the  Maxwell system on $G$, given as 
\begin{align}
\label{EquationMaxwellSystem}
\begin{aligned}
 \partial_t \boldsymbol{D}_\pm &= \curl \boldsymbol{H}_\pm -  \boldsymbol{J}_\pm, \qquad  &&\text{for } x \in G_\pm, \quad  &&t \in J, \\
 \partial_t \boldsymbol{B}_\pm &= -\curl \boldsymbol{E}_\pm,  	 &&\text{for } x \in G_\pm, &&t \in J,  \\
 \div \boldsymbol{D}_\pm &= \rho_\pm,  \quad \div \boldsymbol{B}_\pm = 0, &&\text{for } x \in G_\pm, &&t \in J,  \\
 \boldsymbol{E}_+ \times \nu &= 0, \quad  \boldsymbol{B}_+ \cdot \nu = 0, &&\text{for } x \in \partial G, &&t \in J,  \\
 \boldsymbol{E}_{\pm}(t_0) &= \boldsymbol{E}_{0,\pm}, \quad \boldsymbol{H}_{\pm}(t_0) = \boldsymbol{H}_{0,\pm}, &&\text{for } x \in G_\pm,
\end{aligned}
 \end{align}
for an initial time $t_0 \in \R$, $J=(t_0,T)$, and the unit outward normal vector $\nu$ of $G_+$.
Here $\boldsymbol{E}_\pm(t,x),\boldsymbol{D}_\pm(t,x) \in \R^3$ are the electric and
$\boldsymbol{H}_{\pm}(t,x),\boldsymbol{B}_{\pm}(t,x) \in \R^3$ the magnetic fields on $G_\pm$. It is known that the 
divergence equations and the magnetic boundary condition $\boldsymbol{B}_+\cdot\nu = 0$ in \eqref{EquationMaxwellSystem}  remain valid if they are satisfied 
by the initial fields. Here, the charge densities $\rho_{\pm}(t,x)$ are given by the initial charge and the current densities
$\boldsymbol{J}_{\pm}(t,x) \in \R^3$  via
\begin{equation*}
 \rho_{\pm}(t) = \rho_{\pm}(t_0) - \int_{t_0}^t \div \boldsymbol{J}_{\pm}(s) ds
\end{equation*}
for all $t \geq t_0$ on $G_\pm$. (See Section~I.4.2.2 in \cite{DautrayLionsI}.)
In~\eqref{EquationMaxwellSystem} we have imposed the boundary conditions of a 
perfect conductor on $\partial G$. On $\Sigma$ the Maxwell equations imply the interface conditions
\begin{align}
\label{EquationJumpConditionsOnInterface}
 [\boldsymbol{D} \cdot \nu] &= -\rho_{\Sigma}, \quad \ [\boldsymbol{B} \cdot \nu] = 0, \quad \
 [\boldsymbol{E} \times \nu] = 0, \quad \ [\boldsymbol{H} \times \nu] = \boldsymbol{J}_{\Sigma}
 \end{align}
 for  $x \in \Sigma$ and $t \in (t_0,T),$  see Section~I.4.2.4 of \cite{DautrayLionsI}, where   
 $[\boldsymbol{D} \cdot \nu] = (\boldsymbol{D}_+ - \boldsymbol{D}_{-}) \cdot \nu$ etc. 
 In \eqref{EquationJumpConditionsOnInterface}  the charge density $\rho_\Sigma$ on the interface is determined by
 \begin{equation*}
  \rho_\Sigma(t) =  \rho_\Sigma(t_0) - \int_{t_0}^t (\div_{\Sigma} \boldsymbol{J}_{\Sigma}(s) - [\boldsymbol{J} \cdot \nu](s)) ds, \qquad t\in J,
 \end{equation*}
 and the equations for $\boldsymbol{D}$ and $ \boldsymbol{B} $  are  true if they are valid at $t=t_0$,
 see Lemma~\ref{LemmaConservationOfInterfaceConditions}.

  The system~\eqref{EquationMaxwellSystem} has to be complemented by constitutive relations 
between the electric and magnetic fields, where we 
choose $\boldsymbol{E}_{\pm}$ and $\boldsymbol{H}_{\pm}$ as state variables. There are various classes of such material laws.
In the so-called retarded ones the fields $\boldsymbol{D}_{\pm}$ and $\boldsymbol{B}_{\pm}$ depend also on 
the past of $\boldsymbol{E}_{\pm}$ and $\boldsymbol{H}_{\pm}$, 
see~\cite{BabinFigotin}, \cite{FabrizioMorro}, \cite{MN}, or~\cite{RoachStratisYannacopoulos}. 
In dynamical material laws the material response is modelled by additional evolution equations, 
see \cite{AmmariHamdache}, \cite{DumasSueur}, \cite{Jochmann}, \cite{JolyMetivierRauch}, or \cite{MN}.
We concentrate on instantaneous material laws, see~\cite{BuschEtAl} or~\cite{FabrizioMorro}, where the fields 
$\boldsymbol{D}_{\pm}$ and $\boldsymbol{B}_{\pm}$ are given by
\[\boldsymbol{D}_{\pm}(t,x) = \theta_{1,\pm}(x,\boldsymbol{E}_{\pm}(t,x), \boldsymbol{H}_{\pm}(t,x)), \qquad\!
\boldsymbol{B}_{\pm}(t,x) = \theta_{2,\pm}(x,\boldsymbol{E}_{\pm}(t,x), \boldsymbol{H}_{\pm}(t,x))\] 
for regular functions $\theta_\pm=(\theta_{1,\pm}, \theta_{2,\pm}) \colon G_{\pm} \times \R^6 \rightarrow \R^6$.
The most prominent example is the so called Kerr nonlinearity 
$\boldsymbol{D}_{\pm} =  \boldsymbol{E}_\pm + \vartheta_\pm |\boldsymbol{E}_{\pm}|^2 \boldsymbol{E}_{\pm}$ and $\boldsymbol{B}_{\pm} = \boldsymbol{H}_{\pm}$
with $\vartheta_\pm \colon G_\pm \rightarrow \R$. We further assume that the current density decomposes as 
\begin{equation}
\label{EquationOhmsLaw}
 \boldsymbol{J}_{\pm} = \boldsymbol{J}_{0,\pm} + \tilde{\sigma}_{\pm}(\boldsymbol{E}_{\pm}, \boldsymbol{H}_{\pm}) \boldsymbol{E}_{\pm},
\end{equation}
where $\boldsymbol{J}_{\pm,0}$ is a given external current density and $\tilde{\sigma}_{\pm}$ denotes the conductivity on $G_\pm$. 
If we insert these material laws into~\eqref{EquationMaxwellSystem} 
and formally differentiate, we derive
\begin{align*}
	(\partial_t \boldsymbol{D}_{\pm}, \partial_t \boldsymbol{B}_{\pm}) = \partial_{(\boldsymbol{E}_{\pm}, \boldsymbol{H}_{\pm})} \theta_{\pm}(x,\boldsymbol{E}_{\pm}, \boldsymbol{H}_{\pm}) \partial_t (\boldsymbol{E}_{\pm}, \boldsymbol{H}_{\pm}) 
	\! = \! (\curl \boldsymbol{H}_{\pm} \! - \! \boldsymbol{J}_{\pm}, - \! \curl \boldsymbol{E}_{\pm})
\end{align*}
from~\eqref{EquationMaxwellSystem}. Our main structural  assumption is that 
$\partial_{(\boldsymbol{E}_{\pm},\boldsymbol{H}_{\pm})} \theta_\pm$
is symmetric and positive definite, which is true for the Kerr law for small $\boldsymbol{E}_{\pm}$ (and globally  if $\vartheta_\pm \ge0$).
 Such assumptions are quite standard already for linear Maxwell equations.
 
 The resulting equations form a symmetric quasilinear hyperbolic system of  first order.
In order to transform~\eqref{EquationMaxwellSystem}  into a standard form, we introduce the matrices
\begin{align}
 J_1 &= \begin{pmatrix}
	  0 &0 &0 \\
	  0 &0 &-1 \\
	  0 &1 &0
       \end{pmatrix},
        \quad
 J_2 =  \begin{pmatrix}
         0 &0 &1 \\
         0 &0 &0 \\
         -1 &0 &0
        \end{pmatrix},
        \quad
 J_3 = \begin{pmatrix}
        0 &-1 &0 \\
        1 &0 &0 \\
        0 &0 &0
       \end{pmatrix}, \notag \\
\label{EquationDefinitionOfAj}
 A_j^{\operatorname{co}} &= \begin{pmatrix}
        0 & -J_j \\
        J_j &0 
       \end{pmatrix}, \qquad j\in\{1,2,3\}.
\end{align}
Note that $J_1 \partial_1 + J_2 \partial_2 +J_3 \partial_3 = \curl$. Writing $\chi_{\pm}=\partial_{(\boldsymbol{E}_{\pm}, \boldsymbol{H}_{\pm})} \theta_{\pm}$, 
$f_{\pm} = (- \boldsymbol{J}_{\pm,0},0)$, $\sigma_\pm = (\begin{smallmatrix} \tilde{\sigma}_{\pm} &0 \\ 0 &0 \end{smallmatrix}$), 
and  using $u_{\pm} = (\boldsymbol{E}_{\pm}, \boldsymbol{H}_{\pm})$ as a new variable, we  obtain the system
\begin{align}
\label{EquationMaxwellAsFirstOrderSystemEvolutionPart}
	\chi_\pm(u_\pm) \partial_t u_\pm + \sum_{j = 1}^3 A_j^{\operatorname{co}} \partial_j u_\pm +\sigma_\pm(u_\pm) u_\pm 
	   = f_\pm, \quad &&(t,x) \in J \times G_\pm.
\end{align}
To recast the electric boundary and interface conditions in \eqref{EquationMaxwellSystem} and \eqref{EquationJumpConditionsOnInterface}, we set
\begin{align}
\label{EquationDefinitionB}
	B_{\nu} = \begin{bmatrix}
      0 &\nu_3 &-\nu_2  \\
      -\nu_3 &0 &\nu_1  \\
      \nu_2 &-\nu_1 &0  
     \end{bmatrix}, \
     B_{\partial G} = \begin{bmatrix}
                       B_\nu &0
                      \end{bmatrix}, \
     B_{\Sigma} \!=\! \begin{bmatrix}
                   B_\nu &0 &-B_\nu &0 \\
                   0 &B_\nu &0 &-B_\nu
                  \end{bmatrix}
\end{align}
on $\partial G$ respectively $\Sigma$, and  put $g = (0, \boldsymbol{J}_{\Sigma})^T$. System~\eqref{EquationMaxwellSystem} 
is then equivalent to the symmetric quasilinear hyperbolic initial boundary value problem
 \begin{align}
 \label{EquationNonlinearIBVP} 
  \left\{\begin{aligned}
   \chi_\pm(u_\pm)\partial_t u_\pm + \sum\nolimits_{j=1}^3 A_j^{\operatorname{co}} \partial_j u_\pm + \tilde\sigma_\pm(u_\pm) u_\pm 
          &= f_\pm, \quad &&x \in G_\pm, \ &&t \in J; \\
   B_{\partial G} u_+ &= 0, &&x \in \partial G, &&t \in J; \\
   B_{\Sigma} (u_+, u_{-}) &= g, &&x \in \Sigma, &&t \in J; \\
   u(t_0) &= u_0, &&x \in  G.
 \end{aligned}\right.
\end{align}
On $\partial G$ we could also allow for inhomogeneous boundary values, see~\cite{SpitzQuasilinearMaxwell}.
As noted above, the magnetic boundary and interface conditions and the divergence relations in 
\eqref{EquationMaxwellSystem} and \eqref{EquationJumpConditionsOnInterface} are true if we impose corresponding
 conditions on $u_0$. (See Lemma~7.25 in \cite{SpitzDissertation} and Lemma~\ref{LemmaConservationOfInterfaceConditions}.)
We  look for solutions $u$ of \eqref{EquationNonlinearIBVP} in the spaces
\begin{align}
	\label{EquationDefinitionGm}
	\Gpmdom{m} &= \bigcap\nolimits_{j = 0}^m C^j(\clJ, \Hpmhdom{m-j}),\\
  \Hpmhdom{k} &= \{v \in \Ltwohdom \colon v_+ \in H^k(G_+), v_{-} \in H^k(G_{-})\},\notag
\end{align}
 cf.~\cite{BenzoniGavage,RauchMassey}, where $k,m\in \N_0$ and $v_\pm$
are the restrictions of $v$ to $G_\pm$. We assume that the coefficients and data  are appropriately smooth and 
 compatible  (in the sense of~\eqref{EquationNonlinearCompatibilityConditions}). Our main 
 Theorem~\ref{TheoremLocalWellposednessNonlinear} then shows that
\begin{enumerate}
	\item the system~\eqref{EquationNonlinearIBVP} has a unique maximal solution $u\in\Gpmdom{m}$  with $m \geq 3$,
	\item finite existence time can be characterized by blow-up in the Lipschitz-norm, 
	\item the solution depends continuously on the data.
\end{enumerate}

 These results are based on the detailed regularity theory in Theorem~\ref{TheoremExistenceAndUniquenessOnDomain} 
 for the corresponding nonautonomous linear system
 \begin{equation}
	\label{EquationIBVPIntroduction} 
	 \left\{\begin{aligned}
   A_{0,\pm} \partial_t u_\pm + \sum\nolimits_{j=1}^3 A_j^{\operatorname{co}} \partial_j u_\pm + D_\pm u_\pm &= f_\pm, \quad &&x \in G_\pm, \quad &&t \in J; \\
   B_{\partial G} u_+ &= 0, &&x \in \partial G, &&t \in J; \\
   B_{\Sigma} (u_+,u_{-}) &= g, &&x \in \Sigma, &&t \in J; \\
   u(t_0) &= u_0, &&x \in  G.
 \end{aligned}\right.
\end{equation}
 We follow the same strategy as for the 
pure initial boundary value problem in~\cite{SpitzQuasilinearMaxwell} and \cite{SpitzMaxwellLinear}. We freeze a map $\hat{u}$ 
in the nonlinearities of~\eqref{EquationNonlinearIBVP}. The resulting linear problem \eqref{EquationIBVPIntroduction} 
can be solved in $\Gpmdom{0}$ for Lipschitz coefficients using  \cite{Eller}. In a lengthy procedure one can first
show a priori estimates for solutions in $\Gpmdom{m}$ and then prove that the $\mathcal{G}_0$--solution  actually
belongs to $\Gpmdom{m}$, provided that data and coefficients are regular enough and compatible.
Here one has to inductively intertwine different results for the tangential, time, and normal directions. The
normal part is the most difficult one due to the characteristic interface and boundary (i.e.,  $A_1^{\operatorname{co}}\nu_1
+ A_2^{\operatorname{co}}\nu_2 + A_3^{\operatorname{co}}\nu_3$ is singular). Our treatment 
of the normal regularity heavily relies on the structure of the Maxwell system, see
Proposition~\ref{PropositionCentralEstimateInNormalDirection} and Lemma~\ref{LemmaRegularityInNormalDirection}.

For these arguments one has to localize the system. In this procedure one at first loses many of the zeros in the 
coefficient matrices of \eqref{EquationNonlinearIBVP}, which also become non-constant. 
However, using an additional  transformation described in \eqref{def:Gr}, \eqref{eq:A3-trafo} and \eqref{eq:B-trafo}, we obtain
localized systems with an unchanged space-independent matrix $A_3^{\operatorname{co}}$ and space-independent 
boundary matrices $B_{\Sigma}$ and $B_{\partial G}$. This fact allows us to partly separate the treatment of the normal 
directions from the others. This achievement is crucial for our analysis.

The nonlinear problem is then solved by a contraction argument in Theorem~\ref{TheoremLocalExistenceNonlinear}, 
which is basically standard though one has to be very
careful setting up the constants.  Here one uses the precise form of the a priori estimate in Theorem~\ref{TheoremExistenceAndUniquenessOnDomain}.
In the derivation of the blow-up criterion and the continuous dependence of the data, one has to use the localized problems
and the structure of the system once more.

Fortunately, the methods developed in \cite{SpitzQuasilinearMaxwell} and \cite{SpitzMaxwellLinear} for the pure boundary value problem
work quite well in the present situation. Many arguments can be adapted with straightforward changes. These are omitted 
below. However, at several points the structure of the problem changes significantly because of the interface condition.
In the first step one has to apply the basic linear $L^2$ results of \cite{Eller} to the localized interface problem on $\R^3$.
To this aim, one rewrites the Maxwell system as a $12\times 12$ initial boundary value system on the positive half-space by reflecting 
the coefficients from the negative one. In this procedure extra signs arise due to the reflection and spoil the structure of the pure 
Maxwell system appearing in~\cite{SpitzMaxwellLinear}, see e.g.\ \eqref{def:cAj} and \eqref{EquationDefinitionOfTildeMu}. 
However, the core parts of the proof concerning normal regularity
heavily depend on cancellation properties of  the arising (linear) Maxwell system. Similarly the structure  of the new $12 \times 12$ Maxwell system
is crucial in order to obtain constant coefficients $A_3^{\operatorname{co}}$ and  $B_{\Sigma}$ in the localization procedure.
 These and several other arguments are closely tied to the structure of the interface problem.
They are thus  worked out in detail, though they lead to lengthy and intricate calculations.

In the next section we introduce our basic notation and some auxiliary results. The localization  procedure is discussed
in Section~\ref{SectionLocalization}.  The core  a priori estimates and regularity results for the linear problem are shown 
in Sections~\ref{SectionAPrioriEstimates} and~\ref{SectionRegularity}, respectively. The basic fixed point argument is included in 
Section~\ref{SectionLocalExistence}, and the main local wellposedness theorem in Section~\ref{SectionMain}.

\section{Function spaces and linear compatibility conditions}
\label{SectionNotation}
\textbf{Standing notation:}
Let $m\in \N_0$ and  set $\tilde{m}=\max\{m,3\}$. We work with domains $G$, $G_+$, and  $G_-$  in $\R^3$ 
such that $G$ is the disjoint union of $G_+$, $G_-$, and $\Sigma:= \partial G_-$. Moreover it is assumed that
 $\Sigma$ and $\partial G $ have a positive distance and are \emph{tame uniform} $C^{\tilde{m} + 2}$--boundaries,
 see Definitions~2.24 and 5.4 of \cite{SpitzDissertation}. This means that they are uniform $C^{\tilde{m} + 2}$-boundaries (see e.g.~\cite{Adams})
 and that there exist a smooth partition  of unity 
 $(\theta_i)_{i \in \N_0}$ of $G_{-}$ respectively $G$ subordinate to the locally 
 finite covering $(U_i)_{i \in \N_0}$ (where $U_0 = G_{-}$ respectively $U_0 = G$), as well as test functions $\sigma_i$ with $\sigma_i=1$ on  $\supp \theta_i$
 and  $\omega_i$ with $\omega_i=1$ on  $\varphi_i(\supp \sigma_i)$, which are all uniformly bounded in $C^{\tilde{m} + 2}$.
 Of course, compact boundaries of class $C^{\tilde{m} + 2}$ or  halfspaces satisfy these assumptions. 

Our solutions take values in domains  $\mathcal{U}_+$ and $\mathcal{U}_{-}$ in $\R^6$.
We further write $\mathcal{L}(\mathcal{A}_{0}, \ldots, \mathcal{A}_{3}, \mathcal{D})$ or
$\mathcal{L}(\mathcal{A}_{j},  \mathcal{D})$ for 
 the differential operator $\sum_{j = 0}^3 \mathcal{A}_{j} \partial_j + \mathcal{D}$
 with the coefficients $\mathcal{A}_{j}$ and  $\mathcal{D}$, where $\partial_0=\partial_t$.
 By $J$ we mean an open time interval and we set $\Omega = J \times \R^3_+$. The image of 
 a function $v$ is designated by $\image v$. For a function $w$ in $\Hpmhdom{1}$, we denote by $\partial_j w$ the $L^2(G)$-function 
 whose restriction to $G_\pm$ coincides with $\partial_j w_\pm$.  In the localization procedure we employ the matrices
 \begin{equation}
  \label{EquationDefinitionOfLargeAj}
  \mathcal{A}_j^{\mathrm{co}} = \begin{pmatrix}
                                 A_j^{\mathrm{co}} &0 \\
                                 0 &A_j^{\mathrm{co}}
                                \end{pmatrix}
                                \quad\text{for \  } j \in \{1,2,3\}\qquad \text{and}\qquad
   \tilde{\mathcal{A}}_3^{\mathrm{co}} = \begin{pmatrix}
                                          A_3^{\mathrm{co}} &0 \\
                                          0 &-A_3^{\mathrm{co}}
                                         \end{pmatrix}.
 \end{equation}

 To introduce the necessary trace operators, take coefficients $A_j \in \mathcal{W}^{1,\infty}(J \times G)$, i.e., 
 the restrictions $A_{j,\pm}$ belong to $W^{1,\infty}(J \times G_\pm)$. Let $v_+$ be an element of  $L^2(J \times G_+)$ 
 such that $\sum_{j = 0}^3 A_{j,+} \partial_j v_+$ is contained in $L^2(J \times G_+)$.
Then the product  $A_+(\nu) v_+ = (\sum_{j=0}^3 A_{j,+} \nu_j) v_+$ has a trace on $J \times \partial G_+$  
belonging to $H^{-1/2}(J \times \partial G_+)$, cf.\ \cite{SpitzDissertation, SpitzMaxwellLinear}, for instance. 
 Here $\nu$  denotes the unit outer normal of $J \times G_+$. We may  restrict this trace to $J \times \Sigma$ and 
 to $J \times \partial  G$, respectively. Moreover, the corresponding trace operators  $\Tr_{J \times \Sigma,+}$ and 
 $\Tr_{J \times \partial G}$ are given by the standard ones $\tr_{\Sigma,+}$ and   $\tr_{\partial G,+}$, respectively,
 if $v_+$ takes values in in $H^1(G_+)$. Here we can replace the subscript $+$ by $-$. We further set 
 \begin{equation*}
 \Tr_{J \times \Sigma,\pm}(A(\nu) u) = (\Tr_{J \times \Sigma,+} (A_+(\nu) u_+), \Tr_{J \times \Sigma,-} (A_{-}(\nu) u_{-}))
 \end{equation*}
 if $u \in L^2(J \times G)$ satisfies $\sum_{j = 0}^3 A_{j,\pm} \partial_j u_\pm \in L^2(J \times G_\pm)$, 
 respectively
 \begin{equation*} 
  \tr_{\Sigma,\pm} u = (\tr_{\Sigma,+} u_+, \tr_{\Sigma,-} u_{-})
 \end{equation*}
 if $u \in \mathcal{H}^1(G)$. We define the trace
 $\Tr_{J \times \Sigma,+}(M A(\nu) u)$ by $M \Tr_{J \times \Sigma,+}(A(\nu) u)$
 for matrix-functions $M \in \mathcal{W}^{1,\infty}(J \times G)$, and correspondingly for 
 the other trace operators. Finally, $\tr_{\Sigma}$ is the usual trace at $\Sigma$ 
for functions in $H^1(G)$ or $C(G)$.
 On $\R^3_+=\{x\in \R^3: x_3>0\}$ we  use the trace operator $\Tr_{J \times \partial \R^3_+}$ as introduced 
 in~\cite{SpitzMaxwellLinear}.
 
 We will employ the same function spaces as in~\cite{SpitzMaxwellLinear}, but we have to add variants allowing 
 discontinuities across the interface.  For reasons of clarity, we introduce all the spaces here.  Take a subdomain 
 $\tilde{G}$ of $\R^3$. We have already encountered the spaces $\Gpmdom{m}$ and $\Hpmhdom{m}$ in \eqref{EquationDefinitionGm}. 
 Their  norms are given by
 \begin{align*}
  \Gpmdomnormwg{m}{v} &=  \max_{j\in\{0,\dots,m\}}  \|\partial_t^j v\|_{L^\infty(J,\Hpmhdom{m-j})},\\
  \Hpmhndom{m}{v}^2 &= \|v_+\|_{H^m(G_+)}^2 + \|v_{-}\|_{H^m(G_{-})}^2.
 \end{align*}
 We also need the simpler version 
 \begin{equation*}
  G_m(J \times \tilde{G}) = \bigcap\nolimits_{j=0}^m C^j(J, H^{m-j}(\tilde{G})).
 \end{equation*}
 Set $e_{-\gamma}(t)=e^{-\gamma t}$ for  $\gamma \ge0$ and $t\in\R$. 
 We use the time-weighted norms
 \begin{equation*}
  \|v\|_{G_{m,\gamma}(J \times \tilde{G})} = \max_{j\in \{0,\dots,m\}} \|e_{-\gamma} \partial_t^j v\|_{L^\infty(J,H^{m-j}(\tilde{G}))}
 \end{equation*}
 for all $\gamma \geq 0$. If $\gamma = 0$, we also write $\|\cdot\|_{G_{m}(J \times \tilde{G})}$ 
 instead of $\|\cdot\|_{G_{m,0}(J \times \tilde{G})}$. Other function spaces on  $J \times \tilde{G}$  or $J\times G$
 are treated analogously.  We further set
 \begin{equation*}
  \tilde{G}_m(J \times \tilde{G}) = \{v \!\in\! L^\infty(J, L^2(\tilde{G})) \colon \partial^\alpha v\! \in\! L^\infty(J, L^2(\tilde{G})) 
  \text{ for all } \alpha \!\in\! \N_0^4 \text{ with } |\alpha| \leq m\},
 \end{equation*}
and define $\tilde{\mathcal{G}}_m(J \times \tilde{G})$ in a similar way. These spaces are endowed with the same norms as 
 $G_m(J \times \tilde{G})$ respectively $\Gpmdom{m}$.
 
 The coefficients of the linear problem will be contained in
 \begin{align*}
  F_{m,k}(J \times \tilde{G}) &= \{A \in W^{1,\infty}(J \times \tilde{G})^{k \times k} \colon  
    \partial^\alpha A \in L^\infty(J, L^2(\tilde{G})) \text{ for all } \alpha \in \N_0^4 \\
    &\hspace{12em} \text{with } 1 \leq |\alpha| \leq m\}, \\
  \|A\|_{F_m(J \times \tilde{G})} &= \max\{\|A\|_{W^{1,\infty}(J \times \tilde{G})}, \max_{1 \leq |\alpha| \leq m}\|\partial^\alpha A\|_{L^\infty(J, L^2(\tilde{G}))}\}; \\  
  \Fpmdompdw{m}{k} &= \{A \in \mathcal{W}^{1,\infty}(J \times G) \colon A_+ \in F_{m,k}(J \times G_+), A_{-} \in F_{m,k}(J \times G_{-})\}, \\
  \|A\|_{\mathcal{F}_{m}(J \times G)} &= \max\{\|A_+\|_{F_m(J \times G_+)}, \|A_{-}\|_{F_m(J \times G_{-})}\}.
 \end{align*}
 The regularity of time-evaluations is measured in the spaces
 \begin{align*}
  \Fvarwkdom{m}{k} &= \{A \in L^\infty(\tilde{G})^{k \times k} \colon \partial^\alpha A \in L^2(\tilde{G})^{k \times k} \text{ for all } \alpha \in \N_0^3  \text{ with }  1\leq |\alpha| \leq m\}, \\
  \Fvarnormdom{m}{A} &= \max\{\|A\|_{L^\infty(\tilde{G})}, \max_{1 \leq |\alpha| \leq m}\|\partial^\alpha A\|_{L^2(\tilde{G})}\}; \\
  \Fpmvarwkdom{m}{k} &= \{A \in L^\infty(G)^{k \times k} \colon A_+ \in F_{m,k}^0(G_+), A_{-} \in F_{m,k}^0(G_{-})\}, \\
  \Fpmvarnormdom{m}{A} &= \max\{\|A_+\|_{F_m^0(G_+)}, \|A_{-}\|_{F_m^0(G_{-})}\}.
 \end{align*}
 The subscript $\eta$ always designates the  subspace of  matrix-valued maps $A$  with 
$A^T = A \geq \eta>0$. By $\Fpmdomupdwl{m}{cp}{k}$  we mean those $A\in \Fpmdompdw{m}{k}$ which 
are constant outside of a compact subset of $\overline{J \times G}$, and by $\Fpmdomupdwl{m}{cv}{k}$
those  which have a limit as $|(t,x)| \rightarrow \infty$. The variants for $F$ instead of $\mathcal{F}$ are defined
analogously. We will only use the parameters $k\in\{ 1, 6, 12\}$. As it will be clear from the 
context which parameter we consider, we usually drop it from our notation. 

After the localization procedure below, 
the coefficients in front of the spatial derivatives  belong to the space
\begin{align} \label{EquationDefinitionFmcoeff}
 \Fcoeff{m}{cp} = \{&\mathcal{A} \in \Fuwlk{m}{cp}{12} \colon \exists\, \mu_1, \mu_2, \mu_3 \in \Fuwlk{m}{cp}{1} 
 \text{ independent of time, } \nonumber\\
 &\text{ such that } \mathcal{A} = \sum\nolimits_{j = 1}^3 \mathcal{A}_j^{\operatorname{co}} \mu_j\}.
\end{align}
Finally, we introduce the space for the data on the interface, namely
\begin{equation*}
 \Edom{m} = \bigcap\nolimits_{j = 0}^m H^j(J, H^{m+\frac{1}{2}-j}(\Sigma)).
\end{equation*}

We next state several bilinear estimates, which will be ubiquitous in the following.
One proves this result by applying Lemma~2.1 from~\cite{SpitzMaxwellLinear} on $G_{-}$ and on $G_+$.
 \begin{lem}
 \label{LemmaRegularityForA0}
  Take  $m_1, m_2 \in \N$ with $m_1 \geq m_2$ and $m_1 \geq 2$ and a parameter $\gamma \geq 0$.
 \begin{enumerate}
  \item \label{ItemProductInG0} Let $k \in \{0,\ldots,m_1\}$, $f \in  \Gpmdomvar{m_1 - k}$, and
      $g \in \Gpmdomvar{k}$. Then 
      \begin{align*}
    f  g \in \Gpmdomvar{0} \quad  \text{ and }\quad \Gpmdomnorm{0}{ f g} \leq C \Gpmdomnormwg{m_1 - k}{f} \Gpmdomnorm{k}{g}.
      \end{align*}

  \item \label{ItemProductHkappa}  Let $f \in \Gpmdomvar{m_1}$ and 
      $g \in \Gpmdomvar{m_2}$. Then $f g \in \Gpmdomvar{m_2}$ and 
      \begin{align*}
       \Gpmdomnorm{m_2}{f g} \leq C \min\{&\Gpmdomnormwg{m_1}{f} \Gpmdomnorm{m_2}{g}, \\
       &\Gpmdomnorm{m_1}{f} \Gpmdomnormwg{m_2}{g}\}.
      \end{align*}
      The result remains true if we replace $\Gpmdomvar{m_1}$ by $\Fpmdom{m_1}$ and if 
      we replace both $\Gpmdomvar{m_1}$ and $\Gpmdomvar{m_2}$ 
      by $\Fpmdom{m_1}$ and $\Fpmdom{m_2}$.
    
  \item 
  Let $k \in \{0,\ldots,m_1\}$, $f \in  \Hpmhdom{m_1 -k}$,
      and $g \in \Hpmhdom{k}$. Then $f  g \in \Ltwodomh$ and
      \begin{align*}
       \Ltwodomhn{ f g} \leq C \Hpmhndom{m_1 - k}{f} \Hpmhndom{k}{g}.
      \end{align*}
      
  \item 
  Let $f \in \Hpmhdom{m_1}$ and $g \in \Hpmhdom{m_2}$. 
	 Then $f g \in \Hpmhdom{m_2}$ and
	 \begin{align*}
	  \Hpmhndom{m_2}{f g} \leq C \Hpmhndom{m_1}{f}\Hpmhndom{m_2}{g}.
	 \end{align*}
	 The result is also valid with $\Hpmhdom{m_1}$ replaced by $\Fpmvardom{m_1}$.
 \end{enumerate}
 In assertions~\ref{ItemProductInG0} and~\ref{ItemProductHkappa} one  can also remove the tildes.
\end{lem}

In Section~\ref{SectionRegularity} we develop a regularization procedure which needs the next
 approximation result for the coefficients, taken from Lemma~2.2 of \cite{SpitzMaxwellLinear}.
 (There it  is stated for $k\in\{1,6\}$, but the proof works componentwise and thus for all $k\in\N$,
 cf.~\cite[Lemma~2.21]{SpitzDissertation}.)
\begin{lem}
 \label{LemmaApproximationOfCoefficients}
 Let $m \in \N$. Choose $A \in \F{m}$.
 Then there exists a family $\{A_\epsilon\}_{\epsilon > 0}$ 
 in $C^\infty(\overline{\Omega})$ satisfying
 \begin{enumerate}
  \item 
	$\partial^\alpha A_\epsilon \in \F{m}$ for all $\alpha \in \N_0^4$ and $\epsilon > 0$,
  \item 
	$\|A_\epsilon\|_{W^{1,\infty}(\Omega)} \leq C \|A\|_{W^{1,\infty}(\Omega)}$ and 
	$\|\partial^\alpha A_\epsilon\|_{L^\infty(J, \Ltwoh)} \leq C \Fnorm{m}{A}$ 
	for all multiindices $1 \leq |\alpha| \leq m$ and $\epsilon > 0$,
  \item 
	$A_\epsilon \rightarrow A$ in $L^\infty(\Omega)$ as $\epsilon \rightarrow 0$,
  \item 
	$A_\epsilon(0) \rightarrow A(0)$ in $L^\infty(\R^3_+)$, and $\partial^\alpha  A$ and $\partial^\alpha A_\epsilon$
	have a 	representative in the space $C(\overline{J},\Ltwoh)$ with
	$\partial^\alpha  A_\epsilon(0) \rightarrow \partial^\alpha  A(0)$ in $\Ltwoh$ as 
	$\epsilon \rightarrow 0$ for all $\alpha \in \N_0^4$ with $0 < |\alpha| \leq m-1$.
 \end{enumerate}
 If $A$ is independent of time, the same is true for $A_\epsilon$ for all $\epsilon > 0$. If $A$ 
 additionally belongs to $\Fuwl{m}{cp}$, $\Fuwl{m}{cv}$,  $\Fpdw{m}{\eta}$ for a number
 $\eta > 0$, or the 
 intersection of two of these spaces, then the same is true for $A_\epsilon$ for all $\epsilon > 0$.
\end{lem}

In order to discuss the compatibility conditions both for the linear Maxwell system~\eqref{EquationIBVPIntroduction}
and its localized variants, we look at \eqref{EquationIBVPIntroduction} with 
variable, time-indepen\-dent  coefficients ${A}_1, {A}_2, {A}_3 \in \Fpmdom{m}$  for a moment. 
We further  fix coefficients ${A}_0 \in \Fpmdom{m,\eta}$ and ${D} \in \Fpmdom{m}$,  as well as data $f \in \Hpmadom{m}$, 
$g \in \Edom{m}$, and $u_0 \in \Hpmhdom{m}$. Given a solution $u$ in  $\Gpmdom{m}$ of  \eqref{EquationIBVPIntroduction}, 
we can differentiate the differential equation in~\eqref{EquationIBVPIntroduction}
up to $(m-1)$-times in time by means of Lemma~\ref{LemmaRegularityForA0}, obtaining the identity
\begin{equation}
\label{EquationTimeDerivativesOfSolution}
 \partial_t^p u(t) = S_{G,m,p}(t,{A}_0, {A}_1,{A}_2,{A}_3,{D},f,u(t)),
\end{equation}
for all $t \in \clJ$ and $p \in \{0, \ldots, m-1\}$. Here we inductively define the maps $S_{G,m,p} =  S_{G,m,p}(t_0,A_j, D, f,u_0 )
=S_{G,m,p}(t_0,{A}_0, {A}_1,{A}_2,{A}_3,{D},f,u_0)$ by
\begin{align}
\label{EquationDefinitionSmp}
 &S_{G,m,0,\pm} = u_{0,\pm}, \nonumber \\
 &S_{G,m,p,\pm}= {A}_{0,\pm}(t_0)^{-1} \Big(\partial_t^{p-1} f_{\pm}(t_0) 
  - \sum_{j = 1}^3 {A}_{j,\pm} \partial_j S_{G,m, p-1,\pm}   \\
  &\hspace{1em} - \sum_{l=1}^{p-1} \binom{p-1}{l} \partial_t^l {A}_{0,\pm}(t_0) S_{G,m, p-l,\pm} - \sum_{l=0}^{p-1} \binom{p-1}{l} \partial_t^l {D}_{\pm}(t_0) S_{G,m,p-1-l,\pm}\Big), \nonumber 
\end{align}
for $1 \leq p \leq m$. On the other hand, we can differentiate the boundary condition 
in~\eqref{EquationIBVPIntroduction} up to $(m-1)$-times in time and insert $t$. It follows the equation
\begin{equation}
\label{EquationDifferentiatedBoundaryCondition}
 B_\Sigma \tr_{\Sigma,\pm}(\partial_t^p u(t)) = \partial_t^p g(t)
\end{equation}
on $\Sigma$ for all $0 \leq p \leq m-1$ and $t \in \clJ$. We proceed on $\partial G$ in the same way.  For $t=t_0$ 
equations~\eqref{EquationTimeDerivativesOfSolution}  and~\eqref{EquationDifferentiatedBoundaryCondition} 
yield the compatibility  conditions of order $m$ 
 \begin{align}
 \label{EquationCompatibilityConditionPrecised}
 B_\Sigma \tr_{\Sigma,\pm} S_{G,m,p}(t_0, {A}_0, \ldots, {A}_3, {D}, f, u_0) &= \partial_t^p g(t_0) 
     \qquad \text{on } \Sigma \text{ \ for } 0 \leq p \leq m-1, \nonumber \\
 B_{\partial G} \tr_{\partial G} S_{G,m,p}(t_0, {A}_0, \ldots, {A}_3, {D}, f, u_0)&= 0 \qquad \text{on } \partial G \text{ \ for } 0 \leq p \leq m-1
\end{align}
for the coefficients and data. These conditions are thus necessary for the existence of a 
solution in $\Gpmdom{m}$. In Section~\ref{SectionRegularity} their sufficiency will be shown.
We will also need them to treat the half-space problem arising from the localization procedure, where $G=\R_+^3$, $k=12$, and $A_j$, $D$, and $B_\Sigma$ are replaced by $\mathcal{A}_j$, $\mathcal{D}$, and $B$.
We often suppress $G$ in the notation.

As the maps $S_{G,m,p}$ appear frequently,  the following estimates are indispensable.
They follow from  Lemma~2.3 of \cite{SpitzMaxwellLinear} applied on $G_+$ and on $G_-$.
\begin{lem}
 \label{LemmaEstimatesForHigherOrderInitialValues}
 Let $\eta > 0$, $m \in \N$, and $\tilde{m} = \max\{m,3\}$.  Pick $r_0 > 0$. Choose
 ${A}_0 \!\in\!  \mathcal{F}_{\tilde{m},\eta}(J\!\times\! G)$, time-independent ${A}_1,\! {A}_2,\! {A}_3\!\in\! \mathcal{F}_{\tilde{m}}(J\!\times\! G)$, 
 and ${D} \!\in\! \mathcal{F}_{\tilde{m}}(J\!\times\! G)$ with
 \begin{align*}
  &\Fpmvarnormdom{\tilde{m}-1}{{A}_i(t_0)} \leq r_0, \quad \Fpmvarnormdom{\tilde{m}-1}{{D}(t_0)} \leq r_0, \\
  &\max_{1 \leq j \leq m-1}\Hpmhndom{\tilde{m}-1-j}{\partial_t^j {A}_0(t_0)} \leq r_0, \quad
  \max_{1 \leq j \leq m-1}{\Hpmhndom{\tilde{m}-1-j}{\partial_t^j {D}(t_0)}} \leq r_0
 \end{align*}
 for all $i \in \{0, \ldots, 3\}$.  Take $f \in \Hpmadom{m}$ and $u_0 \in \Hpmhdom{m}$. Let $0 \leq p \leq m$.
 Then the function $S_{G,m,p}(t_0,A_0,\ldots,A_3, D, f,u_0)$ is contained in $\Hpmhdom{m-p}$.
 Moreover, there exist constants  $C_{m,p} = C_{m,p}(\eta,r_0) > 0$ such that
 \begin{align*}
  \Hpmhndom{m-p}{S_{G,m,p}} \leq C_{m,p} \Big(\sum_{j = 0}^{p-1} \Hpmhndom{m-1-j}{\partial_t^j f(t_0)} + \Hpmhndom{m}{u_0} \Big).
 \end{align*}
\end{lem}

 \section{Localization}
 \label{SectionLocalization}

 We first discuss the localization procedure. In fact, in the logical order of our reasoning 
 this section should be placed after the linear part as  in~\cite{SpitzDissertation}, but we decided to start 
 with it as it determines the linear problems we have to study. 
 The next theorem thus assumes that we can solve the arising linear problems on the half space, 
 which will be shown in Sections~\ref{SectionAPrioriEstimates} and~\ref{SectionRegularity}.
 \begin{theorem}
   \label{TheoremExistenceAndUniquenessOnDomain}
  Let $\eta > 0$, $m \in \N_0$, and $\tilde{m} = \max\{m,3\}$. Fix $r \geq  r_0 >0$. Take a domain $G$ 
  as described at the beginning of Section~\ref{SectionNotation}.
  Choose $t_0 \in \R$, $T' > 0$, $T \in (0, T')$, and set 
  $J = (t_0, t_0 + T)$. Take coefficients $A_0 \in \Fpmdomupdwlk{\tilde{m}}{cv}{\eta}{6}$ and  
  $D \in \Fpmdomuwlk{\tilde{m}}{cv}{6}$ satisfying
  \begin{align*}
    &\Fpmnormdom{\tilde{m}}{A_0} \leq r, \quad \Fpmnormdom{\tilde{m}}{D} \leq r, \\
    &\max \{\Fpmvarnormdom{\tilde{m}-1}{A_0(t_0)},\max_{1 \leq j \leq \tilde{m}-1} \Hpmhndom{\tilde{m}-j-1}{\partial_t^j A_0(t_0)}\} \leq r_0, \\
    &\max \{\Fpmvarnormdom{\tilde{m}-1}{D(t_0)},\max_{1 \leq j \leq \tilde{m}-1} \Hpmhndom{\tilde{m}-j-1}{\partial_t^j D(t_0)}\} \leq r_0.
  \end{align*}
  Choose data $f \in \Hpmadom{m}$, $g \in \Epmdom{m}$, and $u_0 \in \Hpmhdom{m}$ such that the tuples 
  $(t_0, A_0, A_1^{\operatorname{co}}, A_2^{\operatorname{co}}, A_3^{\operatorname{co}}, D, B_{\Gamma}, f, g,u_0)$
  fulfills the compatibility conditions~\eqref{EquationCompatibilityConditionPrecised} of 
  order $m$ on $\Gamma=\Sigma$ and  on $\Gamma=\partial G$. 
   
  Then the linear initial boundary value problem~\eqref{EquationIBVPIntroduction} has a unique solution $u$ in 
  $\Gpmdom{m}$. Moreover,  there is a number $\gamma_m = \gamma_m(\eta, r, T') \geq 1$ such that
 \begin{align}
 \label{EquationAPrioriEstimatesOnADomain}
  &\Gpmdomnorm{m}{u}^2  \leq (C_{m,0} + T C_m) e^{m C_1 T} \Big(  \sum_{j = 0}^{m-1} \Hpmhndom{m-1-j}{\partial_t^j f(t_0)}^2 + \Epmnormdom{m}{g}^2  \nonumber\\
      &\hspace{16em} + \Hpmhndom{m}{u_0}^2 \Big) + \frac{C_m}{\gamma}  \Hpmangammadom{m}{f}^2    
 \end{align}
 for all $\gamma \geq \gamma_m$, where $C_i = C_i(\eta,  r,T') \geq 1$ and $C_{i,0} = C_{i,0}(\eta,r_0) \geq 1$ 
 for $i \in \{1,m\}$.
 \end{theorem}
 
 \begin{proof}
Set $\N_{-1}=\{-1,0\} \cup \N$. Fix a covering $(U_i)_{i \in \N_{-1}}$ of $\overline{G}$, 
a sequence of sets $(V_i)_{i \in \N_{-1}}$, and sequences of functions 
$(\phi_i)_{i \in \N_{-1}}$, $(\theta_i)_{i \in \N_{-1}}$, $(\sigma_i)_{i \in \N_{-1}}$, and $(\omega_i)_{i \in \N_{-1}}$ 
as in Definition~5.4 in~\cite{SpitzDissertation} for the tame uniform $C^{\tilde{m} + 2}$-boundary  $\Sigma$ of $G_{-}$ (complemented by a domain $U_{-1}$ covering 
$\overline{G} \setminus \overline{G_{-}}$ and corresponding functions). We further take $\phi_i = \operatorname{id}$ for $i \in \{-1,0\}$.
Here, $\phi_i:U_i\to V_i$ is a chart,  $(U_i)_{i \in \N}$ is a cover of $\Sigma$ with positive distance to $\partial G$, the set $U_0$ covers 
$G_{-} \setminus \bigcup_{i = 1}^\infty U_i$, while  $\overline{G_+} \setminus \bigcup_{i = 1}^\infty U_{i}$ is contained in $U_{-1}$.
In particular, $(\theta_i)_{i \in \N_{-1}}$ is a smooth partition of unity on $G$. We recall that 
the maps $\omega_i$ equal $1$ on the sets $K_i = \phi_i(\supp \sigma_i)$ and that
$\sigma_i=1$  on $\supp \theta_i$ for all $i \in \N_{-1}$. Moreover,
$\phi_i(U_i \cap G_+) = \{y \in V_i \colon y_3 > 0\}$ and $\phi_i(U_i \cap G_{-}) = \{y \in V_i \colon y_3 < 0\}$ for $i \in \N$.
We use the same symbol for a function and its zero extensions.

\smallskip

 I) In the first step we determine the coefficients of the localized problem on $\R_+^3$.
 To this aim, we write $\psi_i=\phi_i^{-1}  \colon V_i \rightarrow U_i$,
 and  define the composition operators
 \begin{align*} 
  \Phi_i &\colon L^2(U_i) \rightarrow L^2(V_i), \quad v \mapsto v \circ \psi_i; \qquad
  \Phi_i^{-1} \colon L^2(V_i) \rightarrow L^2(U_i), \quad v \mapsto v \circ \phi_i;
 \end{align*}
 for all $i \in \N_{-1}$. Observe that $\phi_i$, and thus $\Phi_i$, are  the identity for $i \in \{-1,0\}$.
The operators $\Phi_i$ and $\Phi_i^{-1}$ act  componentwise on vector-valued functions.
 With a slight abuse of notation we also denote the composition with $\psi_i$ on $L^2(J \times V_i)$ 
 and $H^{-1}(J \times V_i)$ by $\Phi_i$, and analogously for $\Phi_i^{-1}$.

 For $v \in L^2(J \times V_i)$ we  introduce the differential operator 
 \begin{align}
  \label{EquationDerivationTransportedDifferentialOperator}
  \mathfrak{A}^i_\pm v_\pm &:= \Phi_i \Big( A_{0,\pm} \partial_t + \sum_{j=1}^3 A_j^{\operatorname{co}} \partial_j + D_\pm \Big) \Phi_i^{-1} v_\pm \nonumber \\
  &= \Phi_i A_{0,\pm} \, \partial_t v_\pm +  \sum_{l = 1}^3 \Big(\sum_{j = 1}^3 A_j^{\operatorname{co}} \Phi_i \partial_j \phi_{i,l}  \Big) \partial_l v_\pm + \Phi_i D_\pm\, v_\pm, 
 \end{align}
 where $\phi_{i,l}$ is the $l$-th component of $\phi_i$ for all $i \in \N$. Throughout, 
for a function $v$ defined on $V_i$ respectively $\R^3$ we write $v_\pm$ for the restrictions to $V_i\cap \R_\pm^3$
respectively to $\R^3_\pm$, where $\R^3_{-}=\{x\in \R^3: x_3<0\}$. We define
 \begin{align}\label{def:Aj-tilde}
  \tilde{A}_0^i = \Phi_i A_0 , \qquad
  \tilde{A}_l^i = \Phi_i \Big(\sum\nolimits_{j = 1}^3 A_j^{\operatorname{co}} \partial_j \phi_{i,l}\Big), \qquad
  \tilde{D}^i =   \Phi_i D 
 \end{align}
 on $V_i$ for all $i \in \N$ and $l \in \{1,2,3\}$, as well as $\tilde{A}_0^0 = \Phi_0 A_0 = A_0$ and 
 $\tilde{D}^0 = \Phi_0 D = D$ on $U_0$, and $\tilde{A}_0^{-1} = \Phi_{-1} A_0 = A_0$ and $\tilde{D}^{-1} 
 = \Phi_{-1} D = D$ on $U_{-1}$. 
 (This notation is only used if confusion with a matrix inverse is not possible.)
 
   Lemma~5.1 in~\cite{SpitzDissertation} yields numbers $z(i) \in \{1,2,3\}$ and $\tau\in (0,1)$ with
 \begin{equation}
 \label{EquationDefinitionofzi}
    |\partial_{z(i)} \phi_{i,3}| \geq \tau \qquad \text{on \ } U_i
 \end{equation}
 for all $i \in \N$. We pick a point $y_i \in V_i$ for each $i \in \N$ and set
 \begin{align}
  A_0^i &= \omega_i  \tilde{A}_0^i + (1 - \omega_i) \eta \qquad \text{for \ } i \in \N_{-1}, \notag  \\
  A_j^i &= \omega_i  \tilde{A}_j^i + (1 - \omega_i) \frac{\partial_{z(i)} \phi_{i,3}}{|\partial_{z(i)} \phi_{i,3}|}(\psi_i(y_i)) 
   A_{z(i)}^{\operatorname{co}} \qquad\text{for \ } i \in \N, \ \ j \in \{1,2,3\}, \label{def:Aj-ext} \\
  D^i &= \omega_i \tilde{D}^i  \qquad \text{for \ }i \in \N_{-1}. \notag 
 \end{align}
 These coefficients will only be multiplied with  functions supported in the set where $\omega_i=1$, but we need the above extensions
in our reasoning.
 The differential operator $\mathfrak{A}^i$ can thus be  extended to a differential operator  on $\R^3$ by setting
 \begin{equation*}
  \mathfrak{A}^i_\pm v_\pm = A_{0,\pm}^i \partial_t v_\pm + \sum\nolimits_{j = 1}^3 A_{j,\pm}^i \partial_j v_\pm + D_\pm^i  v_\pm
 \end{equation*}
 for all $v \in L^2(J \times \R^3)$ and $i \in \N$. To rewrite the interface problem on $\R^3$
 as an boundary value problem on $\R^3_+$, we set
 \begin{align*}
  \breve{A}_{j,-}^{i}(\cdot, x_3) = A_{j,-}^{i}(\cdot, -x_3), \quad \breve{A}_{3,-}^{i}(\cdot, x_3) = - A_{3,-}^{i}(\cdot,-x_3), 
  \;\;\; \breve{D}_{-}^{i}(\cdot,x_3) = D_{-}^{i}(\cdot, -x_3)
 \end{align*}
 for $j \in \{0,1,2\}$, and introduce the $(12 \times 12)$-matrices
 \begin{align}\label{def:cAj}
  {\mathcal{A}}_j^i = \begin{pmatrix}
                   A_{j,+}^{i} &0 \\
                   0 &\breve{A}_{j,-}^{i}
                  \end{pmatrix}
                  \qquad\text{and}\qquad
  {\mathcal{D}}^i = \begin{pmatrix}
                 D_+^{i} &0 \\
                 0   &\breve{D}_{-}^{i}
                \end{pmatrix}
 \end{align}
 for all $j \in \{0, \ldots, 3\}$ on $J \times \R^3_+$. Here the part of the equation on $\R^3_{-}$ is reflected to $\R^3_+$ 
 and written in the new 6 lines.  The minus in front of $A_{3,-}^{i}$ is needed to compensate the inner derivative when applying $\partial_3$.

 We turn our attention to the interface condition. By Remark~5.2 in~\cite{SpitzDissertation}, the vector field 
 $\nabla \phi_{i,3}$ is normal to $\Sigma$, and hence there is a number $\kappa_i(x) \in \R$ with
 \begin{equation*}
  \nabla \phi_{i,3}(x) = \kappa_i(x) \nu(x)
 \end{equation*}
 for all $x \in \Sigma \cap U_i$ and $i \in \N$. In particular, $\kappa_i = \nabla \phi_{i,3} \cdot \nu$ belongs to
 $C^{m+1}(\Sigma \cap U_i, \R)$ for all $i \in \N$. Moreover, we can extend the product $\kappa_i \nu$ 
 smoothly from $U_i \cap \Sigma$ to $U_i$ by $\nabla \phi_{i,3}$. Let $i \in \N$. We now introduce the interface matrices
 \begin{equation}\label{def:B-hat-i}
  \hat{B}^i = \omega_i \Phi_i (\kappa_i B_\Sigma)  + (1 - \omega_i) \frac{\partial_{z(i)} \phi_{i,3}}{|\partial_{z(i)} \phi_{i,3}|}(\psi_i(y_i)) 
  B_{z(i)}^{\operatorname{co}}, \qquad  B_j^{\operatorname{co}} := B_{\Sigma}(e_j),
 \end{equation}
 on $\R^3$ for $j \in \{1,2,3\}$, where $e_j$ denotes the $j$th unit vector in $\R^3$ and $B_\Sigma(e_j)$ 
 is given by the second line in \eqref{EquationDefinitionB} with $\nu = e_j$.
 Define the function $b_{z(i)}\colon \R^3 \rightarrow \R$ by
 \begin{equation*}
  b_{z(i)} = \omega_i \Phi_i \partial_{z(i)} \phi_{i,3} + (1- \omega_i) \frac{\partial_{z(i)} \phi_{i,3}}{|\partial_{z(i)} \phi_{i,3}|}(\psi_i(y_i)).
 \end{equation*}
  Since 
 $\partial_{z(i)} \phi_{i,3}$ does not change signs on $U_i$, estimate~\eqref{EquationDefinitionofzi} implies the 
 lower bound
 \begin{align*}
  |b_{z(i)}|   &= \omega_i |\Phi_i \partial_{z(i)} \phi_{i,3}| + (1 - \omega_i) \geq \tau \omega_i + 1 - \omega_i 
  = 1 - (1 - \tau) \omega_i \geq \tau
 \end{align*}
 on $\R^3$ as $\tau \in (0,1)$. Consequently, the functions $b_{z(i)}$ and $b_{z(i)}^{-1}$ belong 
 to $C^{m+1}(\R^3)$ and their restrictions  to $\partial \R^3_+$ are elements of $C^{m+1}(\partial \R^3_+)$.

We next want to transform the coefficients  $\mathcal{A}_{3}^i$ and $\hat{B}^i$ to constant coefficients 
similar to those in the original Maxwell system \eqref{EquationIBVPIntroduction} on $G$.
Here we only consider the case $z(i) = 3$ with  $b_3 \geq \tau$ on $\R^3$. The other ones are treated analogously,
 cf.\ Section~5 of \cite{SpitzDissertation}. To rewrite $\mathcal{A}_{3}^i$, we use the matrices
 \begin{align*}
  \hat{A}_3^i = \begin{pmatrix}
                 0 &b_3^i &- \omega_i \Phi_i \partial_2 \phi_{i,3} \\
                 - b_3^i &0 &\omega_i \Phi_i \partial_1 \phi_{i,3} \\
                 \omega_i \Phi_i \partial_2 \phi_{i,3} & -\omega_i \Phi_i \partial_1 \phi_{i,3} &0
                \end{pmatrix}
 \end{align*}
 on $\R^3$. Let $Q$ be the reflection operator defined by $Q v(\cdot, x_3) = v(\cdot, -x_3)$ for any 
 $v \in L^2_{\operatorname{loc}}(J \times \R^3)$. The coefficient ${\mathcal{A}}_3^i$ can now be  written as
 \begin{align*}
  {\mathcal{A}}_3^i = \begin{pmatrix}
                             A_{3,+}^{i} &0 \\
                             0 &-Q A_{3,-}^{i} 
                            \end{pmatrix}
                          = \begin{pmatrix}
                             0 &\hat{A}_3^i &0 &0 \\
                             -\hat{A}_3^i &0 &0 &0 \\
                             0 &0 &0 &- Q \hat{A}_3^i \\
                             0 &0 &Q \hat{A}_3^i &0
                            \end{pmatrix}.
 \end{align*}
 Our main tool are the matrix-valued functions
 \begin{align}\label{def:Gr}
  \hat{G}_r^i = b_3^{i,-1/2} \begin{pmatrix}
			      1 &0 &\omega_i \Phi_i \partial_1 \phi_{i,3} \\
			      0 &1 &\omega_i \Phi_i \partial_2 \phi_{i,3} \\
			      0 &0 &b_3^i
                           \end{pmatrix}, \qquad 
   \mathcal{G}_r^i    = \begin{pmatrix}
		      \hat{G}_r^i &0 &0 &0 \\
		      0 &\hat{G}_r^i &0 &0 \\
		      0 &0 &Q \hat{G}_r^i &0 \\
		      0 &0 &0 &Q \hat{G}_r^i  
		      \end{pmatrix}
 \end{align}
on $\R^3$. Equation \eqref{EquationDefinitionOfLargeAj} then yields
the first desired transformation
 \begin{align}\label{eq:A3-trafo}
  (\mathcal{G}_r^i)^T {\mathcal{A}}_3^i \mathcal{G}_r^i = \begin{pmatrix}
                                                                 A_3^{\operatorname{co}} &0 \\
                                                                 0 &-A_3^{\operatorname{co}}
                                                                \end{pmatrix}
                                                                = \tilde{\mathcal{A}}_3^{\operatorname{co}}.
 \end{align}
 For the boundary condition,  we note that
 \begin{align*}
  \hat{B}^i = \begin{pmatrix}
               \hat{B}_{3,\operatorname{bl}}^i &0 &-\hat{B}_{3,\operatorname{bl}}^i &0 \\
               0 &\hat{B}_{3,\operatorname{bl}}^i &0 &-\hat{B}_{3,\operatorname{bl}}^i
              \end{pmatrix}
              \qquad \text{with} \quad \hat{B}_{3,\operatorname{bl}}^i := \hat{A}_3^i.
 \end{align*}
 Setting $\hat{R}^i_3= (\hat{G}_r^i)^T$, we calculate
 \begin{equation*}
  \hat{R}_3^i \hat{B}_{3,\operatorname{bl}}^i = b_3^{i,1/2}\begin{pmatrix}
                     0 &1 &-\omega_i \Phi_i (\partial_2 \phi_{i,3} ) b_3^{i,-1}  \\
                     -1 & 0 & \omega_i \Phi_i (\partial_1 \phi_{i,3}) b_3^{i,-1}  \\
                     0 &0 &0 
                    \end{pmatrix}
                    =: \tilde{B}^i_{\operatorname{bl},3}
 \end{equation*}
 on $\partial \R^3_+$. Consequently,
 \begin{align}\label{eq:RB-hat}
   R_3^i \hat{B}^i := \begin{pmatrix}
			\hat{R}_3^i &0 \\
			0 &\hat{R}_3^i
		      \end{pmatrix}
		      \cdot \hat{B}^i
  = \begin{pmatrix}
     \tilde{B}_{\operatorname{bl},3}^i &0 &-\tilde{B}_{\operatorname{bl},3}^i &0 \\
     0 &\tilde{B}_{\operatorname{bl},3}^i &0 &-\tilde{B}_{\operatorname{bl},3}^i
    \end{pmatrix}.
 \end{align}
 Delete in $\tilde{B}^i_{\operatorname{bl},3}$ the line of zeros and call the resulting matrix
  $B^i_{\operatorname{bl},3}$. We then introduce the boundary matrices
 \begin{align}  
	B_3^i = \begin{pmatrix} 
           B_{\operatorname{bl},3}^i &0 &-B_{\operatorname{bl},3}^i &0 \\
           0 &B_{\operatorname{bl},3}^i &0 &-B_{\operatorname{bl},3}^i
          \end{pmatrix}.\label{eq:Bi}
 \end{align}
 We next infer that
 \begin{align*}
  B_{\operatorname{bl},3}^i \hat{G}_r^i &= b_3^{i,1/2} \begin{pmatrix}
                     0 &1 &-\omega_i \Phi_i (\partial_2 \phi_{i,3} )b_3^{i,-1}  \\
                     -1 & 0 & \omega_i \Phi_i (\partial_1 \phi_{i,3}) b_3^{i,-1}
                    \end{pmatrix}
                    b_3^{i,-1/2}
                    \begin{pmatrix}
			      1 &0 &\omega_i \Phi_i \partial_1 \phi_{i,3} \\
			      0 &1 &\omega_i \Phi_i \partial_2 \phi_{i,3} \\
			      0 &0 &b_3^i
                           \end{pmatrix} \\
                         &= 
                         \begin{pmatrix}
                          0 &1 &0 \\
                          -1 &0 &0
                         \end{pmatrix}
                         =: B_{\operatorname{bl}}.
 \end{align*}
 On the boundary $\partial \R^3_+$ we thus obtain the second crucial identity 
 \begin{align}\label{eq:B-trafo}
  B_3^i \cdot \mathcal{G}_r^i = \begin{pmatrix}
                                 B_{\operatorname{bl}} &0 &-B_{\operatorname{bl}} &0 \\
                                 0 &B_{\operatorname{bl}} &0 &-B_{\operatorname{bl}}
                                \end{pmatrix}
  =: \mathcal{B}^{\operatorname{co}}.
 \end{align}
Finally, we define the matrices
 \begin{align*} 
  C_{\operatorname{bl}} &= \begin{pmatrix}
         1  &0  &0  \\
         0  &1  &0  
        \end{pmatrix}, \qquad
  \mathcal{C}^{\operatorname{co}} = \begin{pmatrix}
            0 &-C_{\operatorname{bl}} &0 &-C_{\operatorname{bl}} \\
            C_{\operatorname{bl}} &0 &C_{\operatorname{bl}} &0 
           \end{pmatrix}=:  \mathcal{M}^{\operatorname{co}}.
 \end{align*}
 Using \eqref{EquationDefinitionOfAj}, we then compute
 \begin{align*}
  C_{\operatorname{bl}}^T \cdot B_{\operatorname{bl}} & =\begin{pmatrix}
                                                          0 &1 &0 \\
                                                          -1 &0 &0 \\
                                                          0 &0 &0
                                                         \end{pmatrix} = -J_3, \qquad
    B_{\operatorname{bl}}^T C_{\operatorname{bl}}  = (-J_3)^T=J_3,  \\
(\mathcal{C}^{\operatorname{co}})^T \mathcal{B}^{\operatorname{co}} &= \begin{pmatrix}
                                                   0 &C_{\operatorname{bl}}^T \\
                                                   -C_{\operatorname{bl}}^T &0 \\
                                                   0 &C_{\operatorname{bl}}^T \\
                                                   -C_{\operatorname{bl}}^T &0
                                                  \end{pmatrix}
  \cdot \begin{pmatrix}
         B_{\operatorname{bl}} &0 &-B_{\operatorname{bl}} &0 \\
         0 &B_{\operatorname{bl}} &0 &-B_{\operatorname{bl}}
        \end{pmatrix} \\
  &= \begin{pmatrix}
     0 &C_{\operatorname{bl}}^T B_{\operatorname{bl}} &0 &-C_{\operatorname{bl}}^T B_{\operatorname{bl}} \\
     -C_{\operatorname{bl}}^T B_{\operatorname{bl}} &0 &C_{\operatorname{bl}}^T B_{\operatorname{bl}} &0 \\
     0 &C_{\operatorname{bl}}^T B_{\operatorname{bl}} &0 &-C_{\operatorname{bl}}^T B_{\operatorname{bl}} \\
     -C_{\operatorname{bl}}^T B_{\operatorname{bl}} &0 &C_{\operatorname{bl}}^T B_{\operatorname{bl}} &0
    \end{pmatrix}, \\
  (\mathcal{B}^{\operatorname{co}})^T \mathcal{C}^{\operatorname{co}} 
   &= \begin{pmatrix}
     0 &-B_{\operatorname{bl}}^T C_{\operatorname{bl}} &0 &-B_{\operatorname{bl}}^T C_{\operatorname{bl}} \\
     B_{\operatorname{bl}}^T C_{\operatorname{bl}} &0 &B_{\operatorname{bl}}^T C_{\operatorname{bl}} &0 \\
     0 &B_{\operatorname{bl}}^T C_{\operatorname{bl}} &0 &B_{\operatorname{bl}}^T C_{\operatorname{bl}} \\
     -B_{\operatorname{bl}}^T C_{\operatorname{bl}} &0 &-B_{\operatorname{bl}}^T C_{\operatorname{bl}} &0
    \end{pmatrix}.
    \end{align*}
We can now check certain algebraic conditions needed to apply \cite{Eller}, namely
    \begin{align}
     \label{EquationDecompositionOfA3co}
        \Re\big((\mathcal{C}^{\operatorname{co}})^T \mathcal{B}^{\operatorname{co}}\big) 
        &= \frac{1}{2} \Big((\mathcal{C}^{\operatorname{co}})^T \mathcal{B}^{\operatorname{co}} 
          + (\mathcal{B}^{\operatorname{co}})^T \mathcal{C}^{\operatorname{co}} \Big) 
    = \begin{pmatrix}
     0 &-J_3 &0 &0 \\
     J_3 &0 &0 &0 \\
     0 &0 &0 &J_3 \\
     0 &0 &-J_3 &0
    \end{pmatrix} \nonumber\\
    &= \begin{pmatrix}
        A_3^{\operatorname{co}} &0 \\
        0 &-A_3^{\operatorname{co}}
       \end{pmatrix} = \tilde{\mathcal{A}}_3^{\operatorname{co}},\\
  \mathcal{M}^{\operatorname{co}} \tilde{\mathcal{A}}_3^{\operatorname{co}} &= \mathcal{B}^{\operatorname{co}}.\notag
 \end{align}

To simplify the notation, we write $B^i$ and $R^i$ instead of $B^i_{z(i)}$ and $R^i_{z(i)}$ in the following.
Observe that the restrictions of $B^i$ and $R^i$ to $\R^3_+$ belong 
 to $C^{\tilde{m}+1}(\overline{\R^3_+})$.  The rank of $\mathcal{B}^{\operatorname{co}}$
 and $\mathcal{C}^{\operatorname{co}}$ is  $4$ and $R^i(x)$ is invertible for all $x \in \overline{\R^3_+}$. 
 The inverse of  $R^i$ is as regular as $R^i$ itself. Moreover, the transformed coefficients satisfy
 \begin{align}\label{def:cAj-tilde}
  \tilde{\mathcal{A}}_0^i &:= (\mathcal{G}_r^i)^T \begin{pmatrix}
                                                 A_{0,+}^{i} &0 \\
                                                 0 &Q A_{0,-}^{i}
                                                \end{pmatrix}
			      \mathcal{G}_r^i \ \in \ \Fpmupdwl{\tilde{m}}{cp}{\eta},\notag \\
  \tilde{\mathcal{A}}_j^i &:= (\mathcal{G}_r^i)^T {\mathcal{A}}_j^i
			      \mathcal{G}_r^i \ \in \ \Fpmcoeff{\tilde{m}}{cp} \qquad  \text{for \ } j \in \{1,2\},\\
  \tilde{\mathcal{D}}^i &:= (\mathcal{G}_r^i)^T {\mathcal{D}}^i \mathcal{G}_r^i - \sum\nolimits_{j = 1}^3 (\mathcal{G}_r^i)^T 
    {\mathcal{A}}_j^i \mathcal{G}_r^i \partial_j( \mathcal{G}_r^i)^{-1} \mathcal{G}_r^i \ \in \ \Fpmuwl{\tilde{m}}{cp},\notag
 \end{align}
 where we reduced the size of $\eta$ independently of $i$ if necessary.
 
 We next fix a constant $M_1$ as in Lemma~5.1 of~\cite{SpitzDissertation} and constants $M_2$, $M_3$, 
 and $M_4$ as in Definition~5.4 in~\cite{SpitzDissertation}
 for the tame uniform $C^{\tilde{m}+2}$-boundary $\Sigma$ of $G_{-}$. We put $M = \max_{i = 1,\ldots, 4} M_i$
 The construction of our extended coefficients then shows 
 \begin{align}\label{est:cAj} 
  &\Fnorm{m}{{\mathcal{A}}_0^i} \leq C(M_1, M_4) \Fpmnormdom{m}{{A}_0}\le R, \nonumber\\
  &\max \{\Fvarnorm{\tilde{m}-1}{{\mathcal{A}}_0^i(0)},\max_{1 \leq j \leq \tilde{m}-1} \Hhn{\tilde{m}-j-1}{\partial_t^j {\mathcal{A}}_0^i(0)}\} \nonumber\\
  &\quad \leq C(M_1, M_4) \max \{\Fpmvarnormdom{\tilde{m}-1}{{A}_0(0)},\max_{1 \leq j \leq \tilde{m}-1} \Hpmhndom{\tilde{m}-j-1}{\partial_t^j {A}_0(0)}\}\le R_0, \nonumber\\
  &\Fnorm{\tilde{m}}{{\mathcal{A}}_j^i} \leq C(M_1, M_4)\le R, \\
  &\Fnorm{\tilde{m}}{{\mathcal{D}}^i} \leq C(M_1, M_4) \Fpmnormdom{m}{{D}}\le R, \nonumber\\
  &\max \{\Fvarnorm{\tilde{m}-1}{{\mathcal{D}}^i(0)},\max_{1 \leq j \leq \tilde{m}-1} \Hhn{\tilde{m}-j-1}{\partial_t^j {\mathcal{D}}^i(0)}\}  \nonumber\\
  &\quad \leq C(M_1, M_4) \max \{\Fpmvarnormdom{\tilde{m}-1}{{D}(0)},\max_{1 \leq j \leq \tilde{m}-1} \Hpmhndom{\tilde{m}-j-1}{\partial_t^j {D}(0)}\}\le R, \notag
   \end{align}
 for all $i \in \N$ and $j \in \{1,2,3\}$, and for constants $R = R(M, r)$ and $R_0 = R_0(M, r_0)$.

 \smallskip

II) After introducing some notation, we relate the compatibility conditions of the localized problem to the given ones. 
Using the reflection operator $Q$ from step~I), we define the maps
\begin{align*}
 &\mathcal{R}_{6} \colon L_{\mathrm{loc}}^2(\R^3, \R^6) \rightarrow L_{\mathrm{loc}}^2(\R^3_+, \R^{12}), \ \  v \mapsto (v_+, Q v_{-}), \\
 &\mathcal{R}_{6 \times 6} \colon L_{\mathrm{loc}}^2(\R^3, \R^{6 \times 6}) \rightarrow L_{\mathrm{loc}}^2(\R^3_+, \R^{12 \times 12}), \ \ 
   A \mapsto \begin{pmatrix}
     A_+ &0 \\
   0 &Q A_{-}
    \end{pmatrix}, \\
 &\hat{\mathcal{R}}_{6 \times 6}\colon L_{\mathrm{loc}}^2(\R^3, \R^{6 \times 6}) \rightarrow L_{\mathrm{loc}}^2(\R^3_+, \R^{12 \times 12}), \ \ 
   A \mapsto \begin{pmatrix}
                          A_+ &0 \\
                         0 &-Q A_{-}
            \end{pmatrix}.
\end{align*}
As it will be clear from the context which operator we consider, we drop the index, and we put
 $\mathcal{R}_i = \operatorname{id}$ for $i \in \{-1,0\}$ and $\mathcal{R}_i = \mathcal{R}$ for $i \in \N$.

 In step~IV) we  determine  the initial (boundary) value problem  solved by the functions
 $\mathcal{R}_i \Phi_i (\theta_i u)$ on $J \times G$, $J \times \R^3$, respectively $J \times \R^3_+$.
For given functions $v \in \Gpmdom{m}$ and  $ h\in \Hpmadom{m}$, then the transformed data
\begin{align}
\label{EquationDefinitionOfLocalizedData}
 f^i(h,v) &= \mathcal{R}_i \Phi_i (\theta_i h) + \mathcal{R}_i \Phi_i \Big(\sum_{j=1}^3 {A}_j^{\operatorname{co}} \partial_j \theta_i v \Big)\ \in \ \Ha{m}, \nonumber\\
 g^i&= \big((\tr_{\partial \R^3_+} R^i) \tilde{\Phi}_i (\tr_{\Sigma} (\theta_i) \kappa_i g)\big)_{\alpha(i)}\ \in \ \E{m}, 
 \nonumber\\
 u_0^i &= \mathcal{R}_i \Phi_i (\theta_i u_0)\ \in \  \Hh{m},
\end{align}
arise for  $i \in \N_{-1}$ respectively $i \in \N$. Here $\alpha(i)$ denotes the $4$-tuple obtained by removing $z(i)$ and
$z(i) + 3$ from $(1,\ldots,6)$ and $\tilde{\Phi}_i$ the composition operator with the restriction of $\psi_i$ to 
$U_i \cap \Sigma$.

Let $v \in \Gpmdom{m}$ be a map with $\partial_t^p v(0) = S_{G,m,p}(0, {A}_0, {A}_1^{\operatorname{co}}, {A}_2^{\operatorname{co}}, 
{A}_3^{\operatorname{co}}, {D}, f, u_0 )$ for all $p \in \{0, \ldots, m-1\}$, 
with the operators $S_{G,m,p}$ from~\eqref{EquationDefinitionSmp}.  We abbreviate
\begin{align}
\label{EquationDefinitionSmpi}
 S_{m,p}^i &= S_{\R^3_+,m,p}(0,{\mathcal{A}}_0^i,{\mathcal{A}}_1^i, {\mathcal{A}}_2^i, {\mathcal{A}}_3^i, {\mathcal{D}}^i, f^i(f,v), u_0^i), \\
 S_{m,p} &= S_{G,m,p}(0, {A}_0, {A}_1^{\operatorname{co}}, {A}_2^{\operatorname{co}}, {A}_3^{\operatorname{co}}, {D}, f, u_0 ) \nonumber
\end{align}
for all $p \in \{0, \ldots, m\}$ and $i \in \N$. The maps $S_{m,p}^i$ and $S_{m,p}$ are well-defined 
due to the regularity of the coefficients and the data. Fix an index $i \in \N$. We claim that 
\begin{equation}
 \label{EquationClaimTransferOfCompatibilityConditions}
 S_{m,p}^i = \mathcal{R} \Phi_i (\theta_i S_{m,p}) \qquad \text{for all \ }p \in \{0, \ldots, m\}.
\end{equation}

To show this assertion, we first note that 
\begin{equation*}
 S_{m,0}^i = u_0^i = \mathcal{R} \Phi_i (\theta_i u_0) = \mathcal{R} \Phi_i (\theta_i S_{m,0}).
\end{equation*}
Next, let the claim~\eqref{EquationClaimTransferOfCompatibilityConditions} be true  for all $l \in \{0, \ldots, p-1\}$ 
and some $p \in \{1, \ldots, m\}$. The definition of the operators $S_{\R^3_+,m,p}$ then yields
\begin{align}
\label{EquationRecursionForSmpi}
 S_{m,p}^i = {\mathcal{A}}_0^i(0)^{-1} \Big[&\partial_t^{p-1} f^i(f,v)(0) - \sum_{j = 1}^3 {\mathcal{A}}_j^i \partial_j S_{m,p-1}^i 
      - \sum_{l=1}^{p-1} \binom{p-1}{l} \partial_t^l {\mathcal{A}}_0^i(0) S_{m,p-l}^{i} \nonumber\\
      &- \sum_{l=0}^{p-1} \binom{p-1}{l} \partial_t^l {\mathcal{D}}^i(0) S_{m, p-1-l}^i \Big].
\end{align}
The induction hypothesis implies that 
\[\supp S_{m,p-l}^i = \supp \Phi_i (\theta_i S_{m,p}) \subseteq \supp \Phi_i \theta_i \subseteq K_i\] 
for all $l \in \{1, \ldots, p\}$. Together with \eqref{def:Aj-ext} and \eqref{def:cAj}, we thus obtain
\begin{align*}
 {\mathcal{A}}_j^i \partial_j S_{m, p-1}^i = {\mathcal{R}} ({A}_j^i) \partial_j S_{m, p-1}^i
 = {\mathcal{R}} (\tilde{{A}}_j^i) \partial_j \mathcal{R} \Phi_i(\theta_i S_{m,p-1}) 
 = \mathcal{R}(\tilde{{A}}_j^i \partial_j \Phi_i(\theta_i S_{m,p-1}) ) 
\end{align*}
for  $j \in \{1,2\}$, as $\omega_i = 1$ on $K_i$. Similarly it follows
\begin{align*}
 {\mathcal{A}}_3^i \partial_3 S_{m, p-1}^i = \hat{\mathcal{R}} ({{A}}_3^i) \partial_3 \mathcal{R} \Phi_i(\theta_i S_{m,p-1}) 
 = \mathcal{R}(\tilde{{A}}_3^i \partial_3 \Phi_i(\theta_i S_{m,p-1})).
\end{align*}
Using also \eqref{def:Aj-tilde}, we next compute
\begin{align*}
 \partial_j (\Phi_i (\theta_i S_{m,p-1})) &= (\nabla (\theta_i S_{m, p-1})) \circ \psi_i \, \partial_j \psi_i 
 = \sum_{l = 1}^3 \Phi_i (\partial_l (\theta_i S_{m,p-1}))\,\partial_j \psi_{i,l},\\
 \mathcal{R}(\tilde{{A}}_j^i \partial_j \Phi_i(\theta_i S_{m,p-1})) &= \mathcal{R} \Big(\sum_{k = 1}^3 {A}_k^{\operatorname{co}} \Phi_i \partial_k \phi_{i,j} \sum_{l=1}^3 \Phi_i \partial_l (\theta_i S_{m, p-1}) \partial_j \psi_{i,l} \Big) \\
 &= \mathcal{R}\Big(\sum_{k,l = 1}^3 {A}_k^{\operatorname{co}} \Phi_i \partial_l (\theta_i S_{m,p-1}) \Phi_i \partial_k \phi_{i,j}\, \partial_j \psi_{i,l} \Big)
\end{align*}
for all $j \in \{1, 2, 3\}$. Applying $\Phi_i$ to the identity
\begin{align*}
 \delta_{lk} & = (\nabla \operatorname{id}_{U_i})_{lk} = (\nabla (\psi_i \circ \phi_i))_{lk} 
 = \sum_{j = 1}^3 \Phi_i^{-1} \partial_j \psi_{i,l} \, \partial_k \phi_{i,j}
\end{align*}
on $U_i$ for all $k,l \in \{1,2,3\}$, we conclude
\begin{align*} 
 &\sum_{j = 1}^3 {\mathcal{A}}_j^i \partial_j S_{m,p-1}^i = \mathcal{R} \Big(\sum_{j,k,l = 1}^3 {A}_k^{\operatorname{co}} \Phi_i \partial_l (\theta_i S_{m,p-1}) \Phi_i \partial_k \phi_{i,j}\, \partial_j \psi_{i,l} \Big) \nonumber\\
 &= \mathcal{R} \Big(\sum_{k,l = 1}^3 {A}_k^{\operatorname{co}} \Phi_i \partial_l (\theta_i S_{m,p-1}) \delta_{lk} \Big)
 = \mathcal{R} \Big( \sum_{k = 1}^3 {A}_k^{\operatorname{co}} \Phi_i \partial_k (\theta_i S_{m,p-1})\Big).
\end{align*}
Note that the support of every term in the brackets on the right hand side of~\eqref{EquationRecursionForSmpi}  is contained in $K_i$ 
and $\omega_i = 1$ on $K_i$. Proceeding as above, the induction hypothesis then yields that $S_{m,p}^i$ is equal to 
\begin{align*}
 & \mathcal{R} \Phi_i {A}_0^i(0)^{-1} \Big[\mathcal{R}\Phi_i (\theta_i \partial_t^{p-1} f(0)) 
  + \mathcal{R} \Phi_i \Big[\sum_{j = 1}^3 \! {A}_j^{\operatorname{co}} \partial_j \theta_i \partial_t^{p-1} v(0) 
    - \sum_{j = 1}^3 \! {A}_j^{\operatorname{co}} \partial_j (\theta_i S_{m,p-1}) \Big] \\
    &\quad \ \, - \sum_{l=1}^{p-1} \! c_{p,l} \mathcal{R} \Phi_i (\partial_t^l  {A}_0^i(0)) \mathcal{R} \Phi_i (\theta_i S_{m,p-l})
      - \sum_{l=0}^{p-1} \! c_{p,l} \mathcal{R} \Phi_i (\partial_t^l {D}^i(0)) \mathcal{R} \Phi_i (\theta_i S_{m, p-1-l}) \Big ] \\
      &\;\; = \mathcal{R} \Phi_i \Big[\theta_i {A}_0(0)^{-1}  \Big(\partial_t^{p-1} f(0) - \sum_{j = 1}^3 {A}_j^{\operatorname{co}} \partial_j S_{m,p-1} 
      - \sum_{l=1}^{p-1}  c_{p,l} \partial_t^l {A}_0(0) S_{m,p-l} \\
      &\hspace{8em} -  \sum_{l=0}^{p-1} c_{p,l} \partial_t^l {D}(0)  S_{m, p-1-l}\Big) \Big], \\
      &\;\; = \mathcal{R} \Phi_i (\theta_i S_{m,p}),
\end{align*}
where $c_{p,l}= \binom{p-1}{l}$ and we also employed that $\partial_t^{p-1}v(0) = S_{m,p-1}$. 
So \eqref{EquationClaimTransferOfCompatibilityConditions} is true.

\smallskip

III) In this step we show that the tuple $(0, {\mathcal{A}}_0^i, \ldots, {\mathcal{A}}_3^i, {\mathcal{D}}^i, B^i, f^i(f,v), g^i, u_0^i)$ 
fulfills the linear compatibility conditions~\eqref{EquationCompatibilityConditionPrecised} on $G=\R^3_+$ of order $m$, where $v$ 
is any function in $\Gpmdom{m}$ with  $\partial_t^p v(0) = S_{m,p}$ for all $p \in \{0, \ldots, m-1\}$. 

To that purpose, we exploit 
our assumption~\eqref{EquationCompatibilityConditionPrecised}, i.e., $ B_{\Sigma} \tr_{\Sigma,\pm} S_{m,p}  = \partial_t^p g(0)$
for all $p \in \{0, \ldots, m-1\}$. Fix  a number $p \in \{0, \ldots, m-1\}$. The trace operator commutes with multiplication by
test functions and the composition with diffeomorphisms, so that \eqref{EquationCompatibilityConditionPrecised} and \eqref{def:B-hat-i}
imply the identities
\begin{align*}
 &\partial_t^p (\tilde{\Phi}_i (\tr_{\Sigma}(\theta_i) \kappa_i g))(0) =\tilde{\Phi}_i (\tr_{\Sigma}(\theta_i) \kappa_i \partial_t^p g(0)) 
 =\tilde{\Phi}_i(\kappa_i B_\Sigma \tr_{\Sigma}(\theta_i) \tr_{\Sigma,\pm} S_{m,p}) \nonumber \\
 &= \tr_{\partial \R^3_+} \hat{B}^i \tilde{\Phi}_i  \tr_{\Sigma,\pm} (\theta_i S_{m,p}) 
 = \tr_{\partial \R^3_+} \hat{B}^i \tr_{\partial \R^3_+,\pm} (\Phi_i (\theta_i S_{m,p})) \\
 &= \tr_{\partial \R^3_+} \hat{B}^i \tr_{\partial \R^3_+} (\mathcal{R}\Phi_i (\theta_i S_{m,p}))
 =  \tr_{\partial \R^3_+} (\hat{B}^i S_{m,p}^i).
\end{align*}
Multiplying this equation with the trace of $R^i$, we arrive at
\begin{equation}
\label{EquationOneStepBeforeRemovingziComponent}
 \tr_{\partial \R^3_+} (R^i)  \tr_{\partial \R^3_+} (\hat{B}^i S_{m,p}^i) = \partial_t^p (\tr_{\partial \R^3_+} (R^i) \tilde{\Phi}_i (\tr_{\Sigma}(\theta_i) \kappa_i g))(0).
\end{equation}
The $z(i)$-th and the $(z(i)+3)$-th components on the left-hand side are zero by \eqref{eq:RB-hat}, so that the same is true for the right-hand side. 
In view of formulas \eqref{eq:RB-hat}, \eqref{eq:Bi} and \eqref{EquationDefinitionOfLocalizedData}, 
equation~\eqref{EquationOneStepBeforeRemovingziComponent} thus yields the desired compatibility conditions 
\begin{equation*}
 \tr_{\partial \R^3_+} (B^i S_{m,p}^i) = \partial_t^p (\tr_{\partial \R^3_+} (R^i) \tilde{\Phi}_i (\tr_{\Sigma}(\theta_i) \kappa_i g))_{\alpha(i)}(0) = \partial_t^p g^i(0).
\end{equation*}

IV) Let $u$ be a solution in $\Gpmdom{m}$ of~\eqref{EquationIBVPIntroduction} with data $f$,  $g$, and  $u_0$.
In this step we derive a priori estimates for $u$ by applying  a priori estimates on $G_+$ from~\cite{SpitzMaxwellLinear}, 
on $\R^3$ from~\cite{SpitzDissertation}, respectively on $\R^3_+$ from 
Theorem~\ref{TheoremAPrioriEstimatesAndRegularityOfSolutionForGeneralCoefficients} below to $\theta_{-1}u$, $\theta_0 u$, 
respectively $\Phi_i(\theta_i u)$ for $i \in \N$. To that purpose, we first note that the properties of the functions 
$\phi_i$, $\psi_i$, and $\theta_i$ imply the equivalences
\begin{align}\label{EquationEstimateForgiInTermsOfg}
 u \in \Gpmdom{m} &\Longleftrightarrow \theta_{-1} u \in G_m(J \times G), \theta_0 u \in G_m(J \times \R^3) \notag\\
 &\hspace{2em} \text{ and \  } \mathcal{R} \Phi_i (\theta_i u) \in G_m(J \times \R^3_+) \text{ \ for all } i \in \N, \notag\\
 f \in \Hpmadom{m} &\Longleftrightarrow \theta_{-1} u \in H^m(J \times G), \theta_0 f \in H^m(J \times \R^3) \\
 &\hspace{2em} \text{ and \ } \mathcal{R} \Phi_i (\theta_i u) \in H^m(J \times \R^3_+) \text{ \ for all } i \in \N, \notag\\
 g \in \Edom{m} &\Longleftrightarrow g^i \in E_m(J \times \partial \R^3_+) \text{ \ for all } i \in \N,\notag
\end{align}
with corresponding bounds.

Fix an index $i \in \N$. Since  $\supp \Phi_i (\theta_i u) \subseteq \supp \Phi_i \theta_i \subseteq K_i$, 
the definition of the extended coefficients in \eqref{def:cAj} as well as
formulas~\eqref{EquationDerivationTransportedDifferentialOperator} and \eqref{EquationDefinitionOfLocalizedData} yield
\begin{align*} 
 {\mathcal{A}}_0^i \partial_t(\mathcal{R} &\Phi_i(\theta_i u)) + \sum_{j = 1}^3 {\mathcal{A}}_j^i \partial_j (\mathcal{R} \Phi_i(\theta_i u)) + {\mathcal{D}}^i \mathcal{R} \Phi_i(\theta_i u) \\
 &= \mathcal{R} \Phi_i \Big({A}_{0,\pm} \partial_t (\theta_i u_\pm) + \sum_{j = 1}^3 {A}_j^{\operatorname{co}} \partial_j (\theta_i u_\pm) + {D}_\pm (\theta_i u_\pm) \Big) \nonumber \\
 &= \mathcal{R} \Phi_i (\theta_i f) + \mathcal{R} \Phi_i \Big(\sum_{j=1}^3 {A}_j^{\operatorname{co}} \partial_j \theta_i u \Big)= f^i(f,u) \nonumber
\end{align*}
on $J \times \R^3_+$. Since $\Tr_{J \times \Sigma} (B_\Sigma (u_+,u_{-})) = g$ on $J \times \Sigma$,
a similar computation  as in step~III) shows that
\begin{align*}
 \Tr_{J \times \partial \R^3_+}[\hat{B}^i \mathcal{R} \Phi_i(\theta_i u)]
   &= \Tr_{J \times \partial \R^3_+}[ \Phi_i (\theta_i \kappa_i B_\Sigma (u_+,\!u_{-}))] 
 = \tilde{\Phi}_i \! \Tr_{J \times \Sigma}  [\theta_i \kappa_i B_\Sigma (u_+,\!u_{-})] \\
 &= \tilde{\Phi}_i (\tr_{\Sigma} (\theta_i) \kappa_i \, \Tr_{J \times \Sigma} [B_\Sigma (u_+,u_{-})]) 
 = \tilde{\Phi}_i (\tr_{\Sigma} (\theta_i) \kappa_i g).
\end{align*}
Multiplying  this equation with the trace of $R^i$ and removing the $z(i)$-th and $z(i)+3$-th component of the
result, we obtain
\begin{align*}
 \Tr_{J \times \partial \R^3_+}(B^i \mathcal{R} \Phi_i(\theta_i u)) 
 &= \Tr_{J \times \partial \R^3_+}(R^i \hat{B}^i \mathcal{R} \Phi_i(\theta_i u))_{\alpha(i)} \\
& = (\tr_{\partial \R^3_+} (R^i) \tilde{\Phi}_i (\tr_{\Sigma} (\theta_i) \kappa_i g))_{\alpha(i)} = g^i,
\end{align*}
cf.\ \eqref{eq:RB-hat}, \eqref{eq:Bi} and \eqref{EquationDefinitionOfLocalizedData}. 
We conclude that the function $\mathcal{R} \Phi_i (\theta_i u)$ is a $G_m(J \times \R^3_+)$-solution of the initial 
boundary value problem
\begin{align}
  \label{EquationIBVPForPhiiThetaiu}
   {\mathcal{A}}_0^i \partial_t v + \sum\nolimits_{j=1}^3 {\mathcal{A}}_j^i \partial_j v + {\mathcal{D}}^i v  &= f^i(f,u), \quad &&x \in \R^3_+, \quad &t \in J; \notag\\
   B^i v &= g^i, \quad &&x \in \partial \R^3_+, &t \in J; \\
   v(0) &= u_0^i, \quad &&x \in \R^3_+.\notag
\end{align}
In the following we abbreviate $U_i \cap G$ by $G_i$ for all $i \in \N_{-1}$. 
The spaces $\mathcal{H}^m(G_i)$, $\mathcal{H}^m(J \times G_i)$ and $\mathcal{G}_m(J \times G_i)$ 
are defined as their analogues on $G$.

To apply Theorem~\ref{TheoremAPrioriEstimatesAndRegularityOfSolutionForGeneralCoefficients}, 
we have to work with a constant 
boundary matrix $\mathcal{A}_3$ and a constant matrix $B$. As shown in step~I), 
this is achieved via the multiplication with the matrices $\mathcal{G}_r^i$. We therefore recall, 
respectively define, the maps 
\begin{align}\label{def:tilde-voll}
 &\tilde{\mathcal{A}}_j^i = (\mathcal{G}_r^i)^T {\mathcal{A}}_j^i \mathcal{G}_r^i, \ \
 \tilde{\mathcal{B}}^i = B^i \mathcal{G}_r^i = \mathcal{B}^{\operatorname{co}}, \ \
 \tilde{\mathcal{D}}^i = (\mathcal{G}_r^i)^T {\mathcal{D}}^i \mathcal{G}_r^i \!
- \!\sum_{j = 1}^3 (\mathcal{G}_r^i)^T\!{\mathcal{A}}_j^i \mathcal{G}_r^i \partial_j (\mathcal{G}_r^i)^{-1} \mathcal{G}_r^i, \notag\\
 &\tilde{f}^i = (\mathcal{G}_r^i)^T f^i, \quad \tilde{g}^i = g^i, \quad \tilde{u}_0^i = (\mathcal{G}_r^i)^{-1} u_0^i
\end{align}
for all $j \in \{0, \ldots, 3\}$. Recall that $\tilde{\mathcal{A}}_3^i = \tilde{\mathcal{A}}_3^{\operatorname{co}}$ by \eqref{eq:A3-trafo}.
We claim that a function $u^i$  belongs to $\G{m}$ and solves~\eqref{EquationIBVPForPhiiThetaiu} if and only if the function 
$\tilde{u}^i = \mathcal{G}_r^{i,-1} u^i$ belongs to $\G{m}$ and solves the initial boundary value problem
\begin{align}
  \label{EquationIBVPForGriinverseui}
   \tilde{\mathcal{L}}v:=\tilde{\mathcal{A}}_0^i \partial_t v + \sum\nolimits_{j=1}^3 \tilde{\mathcal{A}}_j^i \partial_j v + \tilde{\mathcal{D}}^i v  
     &= \tilde{f}^i, \quad &&x \in \R^3_+, \quad &t \in J; \notag \\
   \mathcal{B}^{\operatorname{co}} v &= \tilde{g}^i, \quad &&x \in \partial \R^3_+, &t \in J; \\
   v(0) &= \tilde{u}_0^i, \quad &&x \in \R^3_+.\notag
\end{align}
To see this claim, we assume that $u^i$ is a solution of~\eqref{EquationIBVPForPhiiThetaiu}. We then compute
\begin{align*}
\tilde{\mathcal{L}}\tilde{u}^i &= (\mathcal{G}_r^i)^T \Big[{\mathcal{A}}_0^i \partial_t u^i 
   + \sum_{j = 1}^3 {\mathcal{A}}_j^i \mathcal{G}_r^i \partial_j ((\mathcal{G}_r^i)^{-1} u^i) 
   + {\mathcal{D}}^i u^i - \sum_{j = 1}^3 {\mathcal{A}}_j^i \mathcal{G}_r^i \partial_j (\mathcal{G}_r^i)^{-1} u^i \Big] \\
  &= (\mathcal{G}_r^i)^T \Big[ {\mathcal{A}}_0^i \partial_t u^i + \sum_{j = 1}^3 {\mathcal{A}}_j^i  \partial_j  u^i
   + {\mathcal{D}}^i u^i  \Big] = (\mathcal{G}_r^i)^T f^i = \tilde{f}^i,\\
 \mathcal{B}^{\operatorname{co}} \tilde{u}^i &= B^i u^i = g^i = \tilde{g}^i, \\
 \tilde{u}^i(0)& = (\mathcal{G}_r^i)^{-1} u^i(0) = (\mathcal{G}_r^i)^{-1}  u_0^i = \tilde{u}_0^i.
\end{align*}
Analogously, one shows the other direction. We further note that the tuple
$(0, {\mathcal{A}}_j^i,{\mathcal{D}}^i, B^i, f^i, g^i, u_0^i)$ 
fulfills the compatibility conditions of order $m$ on $\partial \R^3_+$ if and only if the tuple
$(0, \tilde{\mathcal{A}}_j^i, \tilde{\mathcal{D}}^i, \mathcal{B}^{\operatorname{co}}, \tilde{f}^i, \tilde{g}^i, \tilde{u}_0^i)$ 
fulfills the compatibility conditions of order $m$ on $\partial \R^3_+$. To that purpose it is enough to show  that 
\begin{equation}
\label{EquationLinearCompatibilityConditionsTransformWithGr}
 \tilde{S}^i_{m,p} = (\mathcal{G}_r^i)^{-1} S^i_{m,p},
\end{equation}
for all $0 \leq p \leq m$, where we use \eqref{def:tilde-voll} and  set, respectively recall, 
\begin{align*}
 \tilde{S}^i_{m,p} &= S_{\R^3_+,m,p}(0, \tilde{\mathcal{A}}_j^i, \tilde{\mathcal{D}}^i, \tilde{f}^i, \tilde{u}_0^i), \qquad
 S_{m,p}^i =  S_{\R^3_+, m,p}(0, {\mathcal{A}}_j^i, {\mathcal{D}}^i, f^i, u_0^i).
\end{align*}
For $p = 0$ we have $ \tilde{S}^i_{m,0} = \tilde{u}_0^i = (\mathcal{G}_r^i)^{-1} u_0^i = (\mathcal{G}_r^i)^{-1} S^i_{m,0}.$
Next, let~\eqref{EquationLinearCompatibilityConditionsTransformWithGr} be true for all 
$0 \leq l \leq p-1$. Inserting  \eqref{def:tilde-voll}, we compute
\begin{align*}
 \tilde{S}_{m,p}^i &= \tilde{\mathcal{A}}_0^{i,-1} \Big(\partial_t^{p-1} \tilde{f}^i(0) 
      - \sum_{j = 1}^3 \tilde{\mathcal{A}}_j^i \partial_j \tilde{S}_{m,p-1}^i 
      - \sum_{l = 1}^{p-1} \binom{p-1}{l} \partial_t^l \tilde{\mathcal{A}}_0^i(0) \tilde{S}_{m,p-l}^i \\
      &\hspace{4em} - \sum_{l=0}^{p-1} \binom{p-1}{l} \partial_t^l \tilde{\mathcal{D}}^i(0) \tilde{S}_{m,p-1-l}^i\Big) \\
      &= \mathcal{G}_r^{i,-1} {\mathcal{A}}_0^{i,-1} \mathcal{G}_r^{i,-T} \Big(\mathcal{G}_r^{i,T} \partial_t^p f^i(0) 
	  - \sum_{j = 1}^3 \mathcal{G}_r^{i,T} {\mathcal{A}}^i_j \mathcal{G}_r^i \partial_j (\mathcal{G}_r^{i,-1} S_{m,p-1}^i) \\
	  &\quad - \sum_{l = 1}^{p-1} \binom{p-1}{l} \partial_t^l (\mathcal{G}_r^{i,T} {\mathcal{A}}_0^i \mathcal{G}_r^i)(0) \mathcal{G}_r^{i,-1} S_{m,p-l}^i \\
	&\quad - \sum_{l=0}^{p-1} \binom{p-1}{l} \partial_t^l \Big(\mathcal{G}_r^{i,T} {\mathcal{D}}^i \mathcal{G}_r^i 
	  - \sum_{j = 1}^3 \mathcal{G}_r^{i,T} {\mathcal{A}}_j^i \mathcal{G}_r^i \partial_j \mathcal{G}_r^{i,-1} \mathcal{G}_r^i\Big)(0) \mathcal{G}_r^{i,-1} S_{m,p-1-l}^i\Big) \\
      &= \mathcal{G}_r^{i,-1}{\mathcal{A}}_0^{i,-1} \Big(\partial_t^{p-1} f^i(0) 
      - \sum_{j = 1}^3 {\mathcal{A}}^i_j \partial_j  S_{m,p-1}^i - \sum_{l = 1}^{p-1} \binom{p-1}{l} \partial_t^l {\mathcal{A}}_0^i(0) S_{m,p-l}^i \\
      &\quad - \sum_{l=0}^{p-1} \binom{p-1}{l} \partial_t^l {\mathcal{D}}^i(0) S_{m,p-1-l}^i\Big) \\
&= (\mathcal{G}_r^i)^{-1} S_{m,p}^i,
\end{align*}
omitting some parentheses. The claim \eqref{EquationLinearCompatibilityConditionsTransformWithGr} is thus valid
for all $0 \leq p \leq m$.

Consequently, we can apply 
Theorem~\ref{TheoremAPrioriEstimatesAndRegularityOfSolutionForGeneralCoefficients} to this transformed problem 
and then obtain a solution of the same regularity of the original problem via the inverse transform. 
Also the a priori estimates carry over to the original problem with an additional constant $C(M_1)$. 
In order to simplify the notation, we suppress this transform in the following but assume 
that the matrices ${\mathcal{A}}_3^{i}$ and $B^i$ are constant.
Theorem~\ref{TheoremAPrioriEstimatesAndRegularityOfSolutionForGeneralCoefficients} 
in combination with~\eqref{EquationDefinitionOfLocalizedData} and \eqref{EquationEstimateForgiInTermsOfg} then yield
\begin{align}
	\label{EquationAPrioriEstimatesAppliedToPhiiThetaiu}
	&\Gnorm{m}{\mathcal{R} \Phi_i(\theta_i u)}^2  \nonumber\\
	&\leq (C_{\ref{TheoremAPrioriEstimatesAndRegularityOfSolutionForGeneralCoefficients}, m, 0} + T C_{\ref{TheoremAPrioriEstimatesAndRegularityOfSolutionForGeneralCoefficients},m}) e^{m C_{\ref{TheoremAPrioriEstimatesAndRegularityOfSolutionForGeneralCoefficients}, 1} T} 
	\Big(\sum_{j = 0}^{m-1} \Hhn{m-1-j}{\partial_t^j f^i(f,u)(0)}^2 \notag \\ 
	&\hspace{2em}  + \Enorm{m}{g^i}^2 
	+ \Hhn{m}{u_0^i}^2\Big)  + C_{\ref{TheoremAPrioriEstimatesAndRegularityOfSolutionForGeneralCoefficients},m} e^{m C_{\ref{TheoremAPrioriEstimatesAndRegularityOfSolutionForGeneralCoefficients}, 1} T} \frac{1}{\gamma} \Hangamma{m}{f^i(f,u)}^2 	\nonumber \\
	&\leq C(M_1)(C_{\ref{TheoremAPrioriEstimatesAndRegularityOfSolutionForGeneralCoefficients}, m, 0} + T C_{\ref{TheoremAPrioriEstimatesAndRegularityOfSolutionForGeneralCoefficients},m}) e^{m C_{\ref{TheoremAPrioriEstimatesAndRegularityOfSolutionForGeneralCoefficients}, 1} T} \Big[\sum_{j = 0}^{m-1} \|\theta_i \partial_t^j f(0)\|_{\mathcal{H}^{m-1-j}(G_i)}^2 \nonumber \\
	&\quad + \sum_{j = 0}^{m-1} \sum_{k=1}^3 \|\partial_k \theta_i S_{m,j}\|_{\mathcal{H}^{m-1-j}(G_i)}^2 +  \Enormdom{m}{\tr_{\Sigma}(\theta_i) \, g}^2 +  \|\theta_i u_0\|_{\mathcal{H}^{m}(G_i)}^2 \Big] \nonumber \\
	&\quad + C(M_1) \frac{ C_{\ref{TheoremAPrioriEstimatesAndRegularityOfSolutionForGeneralCoefficients},m}}{\gamma} e^{m C_{\ref{TheoremAPrioriEstimatesAndRegularityOfSolutionForGeneralCoefficients}, 1} T} \Big(\|\theta_i f\|_{\mathcal{H}^m_\gamma(J \times G_i)}^2 + \sum_{k = 1}^3 \|\partial_k \theta_i u\|_{\mathcal{H}^m_\gamma(J \times G_i)}^2  \Big)
\end{align}
for all $\gamma \geq \gamma_{\ref{TheoremAPrioriEstimatesAndRegularityOfSolutionForGeneralCoefficients},m}$. 
Here we exploited that $\partial_t^j u(0) = S_{m,j}$ 
for all $j \in \{0, \ldots, m-1\}$, and
$C_{\ref{TheoremAPrioriEstimatesAndRegularityOfSolutionForGeneralCoefficients},m} = C_{\ref{TheoremAPrioriEstimatesAndRegularityOfSolutionForGeneralCoefficients},m}(\eta,R,T')$, 
$C_{\ref{TheoremAPrioriEstimatesAndRegularityOfSolutionForGeneralCoefficients},m,0} = C_{\ref{TheoremAPrioriEstimatesAndRegularityOfSolutionForGeneralCoefficients},m,0}(\eta,R_0)$, and 
$\gamma_{\ref{TheoremAPrioriEstimatesAndRegularityOfSolutionForGeneralCoefficients},m} = \gamma_{\ref{TheoremAPrioriEstimatesAndRegularityOfSolutionForGeneralCoefficients},m}(\eta,R,T')$ 
are constants from Theorem~\ref{TheoremAPrioriEstimatesAndRegularityOfSolutionForGeneralCoefficients}.
The estimates for $i\in\{-1,0\}$ follow in the same way from Theorem~1.1 in~\cite{SpitzMaxwellLinear} and Theorem~5.3 in~\cite{SpitzDissertation}
with corresponding constants $\tilde{C}_{m, 0}$ and $\tilde{C}_{m}$.

By Definition~2.24 of \cite{SpitzDissertation} at most $N$ of the sets $U_i$ intersect at a given point, and 
we use the constants $M_1$ and $M_2$ introduced there and Definition~5.4 of \cite{SpitzDissertation}.
The monotone convergence theorem thus implies  that
\begin{align}	
	&\sum_{i = -1}^\infty \|\theta_i u_0\|_{\mathcal{H}^{m}(G_i)}^2 = \sum_{i=-1}^\infty \Big[\int_{G_+} \sum_{|\alpha| \leq m} |\partial^\alpha (\theta_i u_{0,+})|^2 dx +\! \int_{G_{-}} \sum_{|\alpha| \leq m} |\partial^\alpha (\theta_i u_{0,-})|^2 dx \Big] \nonumber\\
	&\leq C(m, M_2) \sum_{|\alpha| \leq m}  \Big[ \int_{G_+} \sum_{i=-1}^\infty \chi_{U_i}|\partial^\alpha u_{0,+}|^2 dx 
	    + \int_{G_{-}} \sum_{i=-1}^\infty \chi_{U_i}|\partial^\alpha u_{0,-}|^2 dx \Big] \nonumber\\
	&\leq C(m, M_2, N) \Hpmhndom{m}{u_0}^2.\label{EquationRecomposeOtherData}
\end{align}
Analogously, we treat the other terms on the right-hand side of \eqref{EquationAPrioriEstimatesAppliedToPhiiThetaiu}.
We set $C_m' = \max\{\tilde{C}_{m}, C_{\ref{TheoremAPrioriEstimatesAndRegularityOfSolutionForGeneralCoefficients},m}\}$
and $C_{m,0}' = \max\{\tilde{C}_{m,0}, C_{\ref{TheoremAPrioriEstimatesAndRegularityOfSolutionForGeneralCoefficients},m,0}\}$. 
Equation~\eqref{EquationAPrioriEstimatesAppliedToPhiiThetaiu}  then yields the inequality
\begin{align*}
 \Gpmdomnorm{m}{u}^2 
 &\leq C(N) \sum_{i = -1}^\infty \| \theta_i u \|_{\mathcal{G}_{m,\gamma}(J \times G_i)}^2  \leq C(N, M_1) \sum_{i=-1}^\infty \Gnorm{m}{\mathcal{R}_i \Phi_i (\theta_i u)}^2 \nonumber \\
 &\leq C(m, N, M_1, M_2,\tau) (C_{m, 0}' + T C_m') e^{m C_{1}' T} \Big(\sum_{j = 0}^{m-1} \|\partial_t^j f(0)\|_{\mathcal{H}^{m-1-j}(G)}^2 \nonumber \\
	&\hspace{2em} +  \sum_{j = 0}^{m-1}\|S_{m,j}\|_{\mathcal{H}^{m-1-j}(G)}^2 + \Enormdom{m}{ g}^2 +  \|u_0\|_{\mathcal{H}^{m}(G)}^2 \Big) \\
	&\quad + C(m, N,  M_1,M_2) \frac{ C_{m}'}{\gamma} e^{m C_{1}' T} \Big( \|f\|_{\mathcal{H}^m_\gamma(J \times  G)}^2 +  \| u\|_{\mathcal{H}^m_\gamma(J \times G))}^2 \Big)\nonumber 
\end{align*}
for all $\gamma \geq \max\{\tilde{\gamma}_{m}, \gamma_{\ref{TheoremAPrioriEstimatesAndRegularityOfSolutionForGeneralCoefficients},m}\}$.
Choosing $\gamma_m = \gamma_m (\eta,\tau,N, M_1, M_2, r, T')$ large enough and using Lemma~\ref{LemmaEstimatesForHigherOrderInitialValues} 
we thus arrive at 
\begin{align*}
 \Gpmdomnorm{m}{u}^2 
 &\leq (C_{m,0} + T C_m) e^{m C_1 T} \Big(\sum_{j = 0}^{m-1} \Hpmhndom{m-1-j}{\partial_t^j f(0)}^2 + \Enormdom{m}{g}^2 \nonumber \\
 &\hspace{12em} + \Hpmhndom{m}{u_0}^2 \Big) + C_m e^{m C_1 T} \frac{1}{\gamma} \Hpmangammadom{m}{f}^2 
\end{align*}
for all $\gamma \geq \gamma_m$. Employing that $R = R(M,r)$ and $R_0 = R_0(M,r_0)$, we also 
deduce that the constants $C_{m,0}$ and $C_m$ are of the claimed form (where we drop the dependence 
on $M$ as $G$ is fixed). We have thus 
shown the a priori estimates~\eqref{EquationAPrioriEstimatesOnADomain}, which imply uniqueness of the 
$\Gpmdom{m}$-solution of~\eqref{EquationIBVPIntroduction}.

\smallskip

V) To solve \eqref{EquationIBVPIntroduction}, we introduce the spaces
\begin{align*}
 \mathcal{G}_{m, \operatorname{iv}}(J \times G) &= \{v \in \Gpmdom{m} \colon \partial_t^j v(0) = S_{m,j}, \, j \in \{0, \ldots, m-1\}\}, \\
 \mathcal{H}^m_{\operatorname{iv}, f}(J \times G) &= \{\tilde{f} \in \Hpmadom{m} \colon \partial_t^j \tilde{f}(0) = \partial_t^j f(0), \, j \in \{0, \ldots, m-1\}\}.
\end{align*}
We point out that $\mathcal{G}_{m, \operatorname{iv}}(J \times G)$ is nonempty by Lemma~2.34 from~\cite{SpitzDissertation} and 
$\mathcal{H}^m_{\operatorname{iv}, f}(J \times G)$ is nonempty as $f \in \mathcal{H}^m_{\operatorname{iv}, f}(J \times G)$. Because 
the time derivatives up to order $m-1$ in $0$ of functions from $\mathcal{H}^m_{\operatorname{iv}, f}(J \times G)$ respectively 
$\mathcal{G}_{m, \operatorname{iv}}(J \times G)$ coincide, we obtain
\begin{align}
\label{EquationReplacingfbytildefIfInitialValuesfit}
 S_{\R^3_+,m,p}(0, {\mathcal{A}}_j^i, {\mathcal{D}}^i, f^i(\tilde{f}, \tilde{v}), u_0^i) 
 &= S_{\R^3_+,m,p}(0, {\mathcal{A}}_j^i, {\mathcal{D}}^i, f^i(f, v), u_0^i) 
 = S_{m,p}^i
\end{align}
for all $\tilde{f} \in \mathcal{H}^m_{\operatorname{iv}, f}(J \times G)$, $v, \tilde{v} \in \mathcal{G}_{m, \operatorname{iv}}(J \times G)$, 
$p \in \{0, \ldots, m\}$, and $i \in \N$, cf.~\eqref{EquationDefinitionSmpi}. The analogous equations for $i\in\{-1,0\}$ are 
also true. Step~III) thus implies that the tuple $(0, {\mathcal{A}}_j^i, {\mathcal{D}}^i , B^i , f^i(\tilde{f}, v), g^i, u_0^i)$ 
fulfills the compatibility conditions of order $m$ for all $\tilde{f} \in \mathcal{H}^m_{\operatorname{iv}, f}(J \times G)$, $v \in \mathcal{G}_{m, \operatorname{iv}}(J \times G)$, 
and $i \in \N$. 
As explained in step~IV), we can now apply Theorem~\ref{TheoremAPrioriEstimatesAndRegularityOfSolutionForGeneralCoefficients} which shows that the problem
\begin{align}
  \label{EquationIBVPForExistence}
  {\mathcal{A}}_0^i \partial_t w + \sum\nolimits_{j=1}^3 {\mathcal{A}}_j^i \partial_j w +{\mathcal{D}}^i w  &= f^i(\tilde{f},v), \quad &&x \in \R^3_+, \quad &t \in J; \notag\\
   {B}^i w &= g^i, \quad &&x \in \partial \R^3_+, &t \in J; \\
   w(0) &= u_0^i, \quad &&x \in \R^3_+;\notag 
\end{align}
has a unique solution $\mathcal{U}^i(\tilde{f}, v)$ in $\G{m}^{12}$ for all $\tilde{f} \in \mathcal{H}^m_{\operatorname{iv}, f}(J \times G)$, 
$v \in \mathcal{G}_{m, \operatorname{iv}}(J \times G)$, and $i \in \N$. Moreover, Theorem~5.3 from~\cite{SpitzDissertation}
gives a function $\mathcal{U}^0(\tilde{f}, v)$ in $G_m(J \times \R^3)^6$ solving the initial value problem
\begin{align}
  \label{EquationIVPForExistence}
   A_0^0 \partial_t w + \sum\nolimits_{j=1}^3 A_j^{\operatorname{co}} \partial_j w + D^0 w  &= f^0(\tilde{f},v), \quad &&x \in \R^3, \quad &t \in J; \\
   w(0) &= u_0^0, \quad &&x \in \R^3;\notag
\end{align}
for all such $\tilde{f}$ and $v$. Finally, Theorem~1.1 and Remark~1.2 in~\cite{SpitzMaxwellLinear} yield a solution 
$\mathcal{U}^{-1}(\tilde{f},v)$ in $G_m(J \times G)^6$ of the initial boundary value problem
\begin{align}
  \label{EquationIBVPPartialGForExistence}
  {{A}}_0^{-1} \partial_t w + \sum\nolimits_{j=1}^3 {{A}}_j^{\operatorname{co}} \partial_j w +{{D}}^{-1} w  &= f^{-1}(\tilde{f},v), \quad &&x \in G, \quad &t \in J; \notag\\
   {B}_{\partial G} w &= 0, \quad &&x \in \partial G, &t \in J; \\
   w(0) &= u^{-1}_0, \quad &&x \in G;\notag
\end{align}
for all such $\tilde{f}$ and $v$.
We claim that there is a map $f^* = f^*(v)$ in $\mathcal{H}^m_{\operatorname{iv}, f}(J \times G)$ with
\begin{equation}
 \label{EquationDefiningEquationForfstar}
 f^* +  \sum_{i = -1}^\infty \sum_{j = 1}^3  {A}_j^{\operatorname{co}} \partial_j \sigma_i \Phi_i^{-1} \mathcal{R}_i^{-1} \mathcal{U}^i(f^*, v) = f
\end{equation}
for all $v \in \mathcal{G}_{m, \operatorname{iv}}(J \times G)$.
To prove this claim, we define the operator
\begin{align*}
 \Psi_v \colon \mathcal{H}^m_{\operatorname{iv}, f}(J \times G) \rightarrow \mathcal{H}^m_{\operatorname{iv}, f}(J \times G), \quad \ \
	    \tilde{f} \longmapsto f  - \sum_{i = -1}^\infty \sum_{j = 1}^3  {A}_j^{\operatorname{co}} \partial_j \sigma_i \Phi_i^{-1} \mathcal{R}_i^{-1} \mathcal{U}^i(\tilde{f}, v)
\end{align*}
for each $v \in \mathcal{G}_{m, \operatorname{iv}}(J \times G)$. We fix such a function $v$. The operator $\Psi_v$ indeed takes values 
in $\Hpmadom{m}$ since $\Phi_i^{-1} \mathcal{R}^{-1}$ maps the $\Ha{m}$-function 
$\mathcal{U}^i(\tilde{f},v)$ into $\mathcal{H}^m(J \times U_{i})$ for $i \in \N$, $\partial_j \sigma_i$ has compact support in $U_i$, 
and the covering $(U_i)_{i \in \N}$ is locally finite. We further compute
\begin{align*}
 \partial_t^p \Psi_v(\tilde{f})(0) &= \partial_t^p f(0) -  \sum_{i = -1}^\infty \sum_{j = 1}^3 {A}_j^{\operatorname{co}} \partial_j \sigma_i \Phi_i^{-1} \mathcal{R}_i^{-1} \partial_t^p \mathcal{U}^i(\tilde{f},v)(0) \\
    &= \partial_t^p f(0)  - \sum_{i = -1}^\infty \sum_{j = 1}^3 \mathcal{A}_j^{\operatorname{co}} \partial_j \sigma_i \Phi_i^{-1} \mathcal{R}_i^{-1} \mathcal{R}_i \Phi_i (\theta_i S_{m,p}) \\
    &= \partial_t^p f(0) - \sum_{i = -1}^\infty \sum_{j = 1}^3 \mathcal{A}_j^{\operatorname{co}} \partial_j \sigma_i \theta_i S_{m,p} = \partial_t^p f(0)
\end{align*}
for all $p \in \{0, \ldots, m-1\}$ and $\tilde{f} \in \mathcal{H}^m_{\operatorname{iv}, f}(J \times G)$, where we 
used~\eqref{EquationTimeDerivativesOfSolution}, \eqref{EquationReplacingfbytildefIfInitialValuesfit}, 
\eqref{EquationClaimTransferOfCompatibilityConditions}, and that $\sigma_i$ equals $1$ on the support of $\theta_i$ for all $i \in \N_{-1}$. 
Therefore $\Psi_v$ indeed maps $\mathcal{H}^m_{\operatorname{iv}, f}(J \times G)$ into itself.

We observe that the difference $\mathcal{U}^i(f_1,v) - \mathcal{U}^i(f_2,v)$ solves a problem with zero initial and boundary data.
Moreover, formula \eqref{EquationDefinitionOfLocalizedData} and the initial conditions in the spaces
$\mathcal{H}^m_{\operatorname{iv}, f}(J \times G)$ and $\mathcal{G}_{m, \operatorname{iv}}(J \times G)$ imply that
the time derivatives of the inhomogeneities $f^i(f_k,v)$ coincide at $t=0$. (Such facts are also used below without further notice.)
Theorems~1.1 in~\cite{SpitzMaxwellLinear}, 5.3~in~\cite{SpitzDissertation}, and~\ref{TheoremAPrioriEstimatesAndRegularityOfSolutionForGeneralCoefficients} then imply
\begin{align}
\label{EquationContractionPsi}
 &\Hpmangammadom{m}{\Psi_v(f_1) - \Psi_v(f_2)}^2\\ 
 &\leq C(N, M_1,M_3) \Big( \| \mathcal{U}^{-1}(f_1,v) - \mathcal{U}^{-1}(f_2,v) \|_{H^m_\gamma(J \times G)}^2 \nonumber\\
 &\quad + \| \mathcal{U}^0(f_1,v) - \mathcal{U}^0(f_2,v) \|_{H^m_\gamma(J \times \R^3)}^2+  \sum_{i = 1}^\infty \Hangamma{m}{\mathcal{U}^i(f_1,v) - \mathcal{U}^i(f_2,v)}^2 \Big) \nonumber\\
 &\leq  \frac{C}{\gamma} \sum_{i=-1}^\infty \|\theta_i (f_1-f_2)\|_{\mathcal{H}_\gamma^{m}(J \times G_i)}^2  \le \frac{C}{\gamma}  \Hpmangammadom{m}{f_1 - f_2}^2\notag
\end{align}
for all $\gamma \geq \max\{\gamma_{1.1,m}, \gamma_{5.3,m}, \gamma_{\ref{TheoremAPrioriEstimatesAndRegularityOfSolutionForGeneralCoefficients},m}\}$, 
proceeding as in~\eqref{EquationRecomposeOtherData} in the last step and putting $C=C(m,\eta, \tau,  N, M, r, T')$.
We set
\begin{equation*}
 \gamma^* = \max\{\gamma_{1.1,m},\gamma_{5.3,m}, \gamma_{\ref{TheoremAPrioriEstimatesAndRegularityOfSolutionForGeneralCoefficients},m}, 4 C_{\ref{EquationContractionPsi}}\}, 
\end{equation*}
where $C_{\ref{EquationContractionPsi}}$ denotes the constant on the right-hand side of~\eqref{EquationContractionPsi}. This estimate 
then leads to the bound
\begin{equation}
 \label{EquationEstimateContractionPsi}
 \Hpmangammadom{m}{\Psi_v(f_1) - \Psi_v(f_2)} \leq \frac{1}{2} \Hpmangammadom{m}{f_1 - f_2}
\end{equation}
for all $\gamma \geq \gamma^*$.
We conclude that $\Psi_v$ is a strict contraction on $\mathcal{H}^m_{\operatorname{iv}, f}(J \times G)$, and 
there thus exists a unique function $f^* = f^*(v)$ in $\mathcal{H}^m_{\operatorname{iv}, f}(J \times G)$
satisfying equation~\eqref{EquationDefiningEquationForfstar}.

We next define the operator
\begin{align*}
 \mathcal{S} \colon \mathcal{G}_{m, \operatorname{iv}}(J \times G) \rightarrow \mathcal{G}_{m, \operatorname{iv}}(J \times G), \quad
	v \longmapsto \sum_{i = -1}^\infty \sigma_i \Phi_i^{-1} \mathcal{R}_{i}^{-1}\mathcal{U}^i(f^*(v), v).
\end{align*}
Let $v \in \mathcal{G}_{m, \operatorname{iv}}(J \times G)$.
We first check that $\mathcal{S}(v)$ indeed belongs to $\mathcal{G}_{m, \operatorname{iv}}(J \times G)$. Since $\mathcal{U}^i(f^*(v),v)$ is an element of $\G{m}$,
the function $\Phi_i^{-1} \mathcal{R}^{-1} \mathcal{U}^i(f^*(v),v)$ belongs to 
$\mathcal{G}_m(J \times G_i)$ for $i \in \N$. Moreover, $\mathcal{U}^{-1}(f^*(v),v)$ is contained $G_m(J \times G)$ and 
$\mathcal{U}^0(f^*(v),v)$ in $G_m(J \times \R^3)$.
Exploiting that $\sigma_i$ has compact support in $U_i$, the a priori estimates for $\mathcal{U}^i$, and~\eqref{EquationRecomposeOtherData}, 
we infer that $\mathcal{S}(v)$ belongs to $\Gpmdom{m}$. 
As $f^*(v) \in \mathcal{H}^m_{\operatorname{iv}, f}(J \times G)$, we now combine 
formula~\eqref{EquationReplacingfbytildefIfInitialValuesfit} with~\eqref{EquationClaimTransferOfCompatibilityConditions} 
as well as $\sigma_i = 1$ on $\supp \theta_i$ for all $i \in \N_{-1}$, and compute
\begin{align*}
 &\partial_t^p \mathcal{S}(v)(0) = \sum_{i = -1}^\infty \sigma_i \Phi_i^{-1} \mathcal{R}_i^{-1} \partial_t^p \mathcal{U}^i(f^*(v), v)(0) \\
    &= \sigma_{-1} \theta_{-1} S_{m,p} + \sigma_0 \theta_0 S_{m,p} + \sum_{i = 1}^\infty \sigma_i \Phi_i^{-1} \mathcal{R}^{-1} \mathcal{R} \Phi_i (\theta_i S_{m,p}) 
    = \sum_{i = -1}^\infty \theta_i S_{m,p} = S_{m,p}
\end{align*}
for all $p \in \{0, \ldots, m\}$ and $v \in \mathcal{G}_{m, \operatorname{iv}}(J \times G)$. Hence,
$\mathcal{S}$ maps into $\mathcal{G}_{m, \operatorname{iv}}(J \times G)$.

To show that $\mathcal{S}$ is a strict contraction, we take $v_1, v_2 \in \mathcal{G}_{m, \operatorname{iv}}(J \times G)$. 
 Estimate~\eqref{EquationEstimateContractionPsi} further yields 
\begin{align}
 \label{EquationEstimatingDifferenceOffstars}
 &\Hpmangammadom{m}{f^*(v_1) - f^*(v_2)} = \Hpmangammadom{m}{\Psi_{v_1}(f^*(v_1)) - \Psi_{v_2}(f^*(v_2))} \nonumber\\
    &\leq \Hpmangammadom{m}{\Psi_{v_1}(f^*(v_1)) - \Psi_{v_1}(f^*(v_2))} + \Hpmangammadom{m}{\Psi_{v_1}(f^*(v_2)) - \Psi_{v_2}(f^*(v_2))}  \nonumber \\
    &\leq \frac{1}{2} \Hpmangammadom{m}{f^*(v_1) - f^*(v_2)} + \Hpmangammadom{m}{\Psi_{v_1}(f^*(v_2)) - \Psi_{v_2}(f^*(v_2))}
\end{align}
for all $\gamma \geq \gamma^*$. The definition of the operator $\Psi_{v}$, Theorems~1.1 in~\cite{SpitzMaxwellLinear}, 5.3~in~\cite{SpitzDissertation}, 
and~\ref{TheoremAPrioriEstimatesAndRegularityOfSolutionForGeneralCoefficients}, 
formula~\eqref{EquationDefinitionOfLocalizedData} and a variant of~\eqref{EquationRecomposeOtherData}  imply 
\begin{align}
 \label{EquationEstimatingSecondDifferenceOfPsi}
 &\Hpmangammadom{m}{\Psi_{v_1}(f^*(v_2)) - \Psi_{v_2}(f^*(v_2))}^2 \nonumber\\
 &\leq C(N,M_3) \sum_{i = -1}^\infty \|\Phi_i^{-1} \mathcal{R}_i^{-1} \mathcal{U}^i(f^*(v_2), v_1) - \Phi_i^{-1} \mathcal{R}_i^{-1} \mathcal{U}^i(f^*(v_2),v_2)\|_{\mathcal{H}^m_\gamma(J \times G_i)}^2 \nonumber \\
 &\leq C(m,\eta, \tau, N,M, r, T') \frac{1}{\gamma} \sum_{i = -1}^\infty \Big\|\sum_{j=1}^3 {A}_j^{\operatorname{co}} \partial_j \theta_i (v_1 - v_2) \Big\|_{\mathcal{H}^m_\gamma(J \times G)}^2 \nonumber\\
 &\leq C(m,\eta, \tau, N,M, r, T') \frac{1}{\gamma} \Hpmangammadom{m}{v_1 - v_2}^2
\end{align}
for all $\gamma \geq \gamma^*$. We set $\gamma^{**} = \max\{\gamma^*, 16 C_{\ref{EquationEstimatingSecondDifferenceOfPsi}}\}$ 
and insert~\eqref{EquationEstimatingSecondDifferenceOfPsi} into~\eqref{EquationEstimatingDifferenceOffstars}, 
where $C_{\ref{EquationEstimatingSecondDifferenceOfPsi}}$ denotes the constant on the right-hand 
side of~\eqref{EquationEstimatingSecondDifferenceOfPsi}. We then arrive at
\begin{equation*}
 \Hpmangammadom{m}{f^*(v_1) - f^*(v_2)} \leq \frac{1}{2} \Hpmangammadom{m}{v_1 - v_2}
 \qquad \text{for all \ } \gamma \geq \gamma^{**}.
\end{equation*}

After these preparations, we can now estimate the difference of $\mathcal{S}(v_1)$ and $\mathcal{S}(v_2)$. 
Applying the a priori estimates from Theorem~1.1 in~\cite{SpitzMaxwellLinear}, Theorem~5.3 in~\cite{SpitzDissertation}, 
respectively Theorem~\ref{TheoremAPrioriEstimatesAndRegularityOfSolutionForGeneralCoefficients} 
once more and recalling that $v_1$ and $v_2 $ belong to $\mathcal{G}_{m, \operatorname{iv}}(J \times G)$, we infer as above
\begin{align}  \label{EquationEstimateDifferenceS}
 &\Gpmdomnorm{m}{\mathcal{S}(v_1) - \mathcal{S}(v_2)}^2 \nonumber\\
 &\le C(N,M_1,M_3)\sum_{i = -1}^\infty \Gpmdomnorm{m}{\Phi_i^{-1} \mathcal{R}_i^{-1}\big(\mathcal{U}^i(f^*(v_1),v_1) - \mathcal{U}^i(f^*(v_2), v_2)\big)}^2 \nonumber\\
 &\leq C(m, \eta, \tau,  N, M_, r, T') \frac{1}{\gamma} \Big(\Hpmangammadom{m}{f^*(v_1) - f^*(v_2)}^2 + \Hpmangammadom{m}{v_1 - v_2}^2 \Big)\nonumber \\
 &\leq C(m, \eta, \tau,  N, M, r, T') \frac{1}{\gamma} \cdot \frac{5}{4} \Gpmdomnorm{m}{v_1 - v_2}^2
\end{align}
for all $\gamma \geq \gamma^{**}$. 
We finally set $\gamma_{\mathcal{S}} = \max\{\gamma^{**}, 5 C_{\ref{EquationEstimateDifferenceS}}\}$, 
for  the constant $C_{\ref{EquationEstimateDifferenceS}}$ on the right-hand side of~\eqref{EquationEstimateDifferenceS}. It
follows
\begin{equation*}
 \Gpmdomnorm{m}{\mathcal{S}(v_1) - \mathcal{S}(v_2)} \leq \frac{1}{2} \Gpmdomnorm{m}{v_1 - v_2}
\end{equation*}
for all $\gamma \geq \gamma_{\mathcal{S}}$. There thus exists a unique fixed point $u\in\mathcal{G}_{m, \operatorname{iv}}(J \times G)$ 
of $\mathcal{S}$.

\smallskip

VI) We claim that the fixed point $u$ of $\mathcal{S}$ is a solution of~\eqref{EquationIBVPIntroduction}. To verify this 
assertion, we first compute for $u_\pm =\mathcal{S}(u)_\pm$ 
\begin{align*}
 &\mathcal{L}_\pm u_\pm 
 := {A}_{0,\pm} \partial_t u_\pm + \sum_{j = 1}^3 {A}_{j}^{\operatorname{co}} \partial_j u_\pm + {D}_\pm u_\pm \\
  & \ \ = \sum_{i = -1}^\infty \sigma_{i,\pm} \Big({A}_{0,\pm} \partial_t (\Phi_i^{-1} \mathcal{R}_i^{-1}\mathcal{U}^i(f^*(u),u))_\pm + \sum_{j = 1}^3 {A}_{j}^{\operatorname{co}} \partial_j (\Phi_i^{-1} \mathcal{R}_i^{-1} \mathcal{U}^i(f^*(u),u))_{\pm} \\
&\qquad  + {D}_{\pm} (\Phi_i^{-1} \mathcal{R}_i^{-1} \mathcal{U}^i(f^*(u),u))_\pm \Big) + \sum_{i = -1}^\infty \sum_{j = 1}^3 {A}_j^{\operatorname{co}} \partial_j \sigma_{i,\pm} (\Phi_i^{-1} \mathcal{R}_i^{-1} \mathcal{U}^i(f^*(u),u))_\pm
\end{align*}
on $J \times G_\pm$. Recalling \eqref{def:Aj-tilde}, \eqref{def:Aj-ext}, \eqref{def:cAj}, 
and that $\omega_i = 1$ on $\phi_i(\supp \sigma_i)$, on $G_+ \cap \supp \sigma_i$ we have
\begin{align*}
 &\sum_{j = 1}^3 {A}_j^{\operatorname{co}} \partial_j (\Phi_i^{-1} \mathcal{R}^{-1} v)
 = \sum_{j = 1}^3 {A}_j^{\operatorname{co}} \partial_j v_{(1, \ldots, 6)}(\phi_i(x))\\
 &= \sum_{j,l = 1}^3 {A}_j^{\operatorname{co}} \partial_l v_{(1, \ldots, 6)}(\phi_i(x)) \partial_j \phi_{i,l}(x)  
 = \sum_{l = 1}^3 \Phi_i^{-1} ({A}_l^i)(x) \partial_l v_{(1, \ldots, 6)}(\phi_i(x)) \\
 &= \Phi_i^{-1} \mathcal{R}^{-1} \Big(\sum_{l = 1}^3 {\mathcal{A}}_l^i \partial_l v \Big),
\end{align*}
whereas on $G_{-} \cap \supp \sigma_i$, we deduce
\begin{align*}
	&\sum_{j = 1}^3 {A}_j^{\operatorname{co}} \partial_j (\Phi_i^{-1} \mathcal{R}^{-1} v)
 = \sum_{j = 1}^3 {A}_j^{\operatorname{co}} \partial_j  v_{(7, \ldots, 12)}(\phi_{i,1}(x), \phi_{i,2}(x), - \phi_{i,3}(x)) \\
 &= \sum_{j = 1}^3 {A}_j^{\operatorname{co}} \nabla  v_{(7, \ldots, 12)}(\phi_{i,1}(x), \phi_{i,2}(x), - \phi_{i,3}(x)) \cdot (\partial_j \phi_{i,1}(x), \partial_j \phi_{i,2}(x), - \partial_j \phi_{i,3}(x)) \\
 &= \sum_{j,l = 1}^3 {A}_j^{\operatorname{co}}  \partial_l v_{(7, \ldots, 12)}(\phi_{i,1}(x), \phi_{i,2}(x), - \phi_{i,3}(x)) \partial_j \phi_{i,l}(x) (-1)^{\delta_{3l}} \\
 &=  \sum_{l = 1}^3 \Phi_i^{-1} ({A}_{l,-}^i (-1)^{\delta_{3l}} Q \partial_l v_{(7, \ldots, 12)}) =  \sum_{l = 1}^3 \Phi_i^{-1} Q(\breve{A}_{l,-}^i \partial_l v_{(7, \ldots, 12)}) \\
& = \sum_{l = 1}^3 \Phi_i^{-1} Q (\mathcal{A}_l^i \partial_l v)_{(7, \ldots, 12)} = \Phi_i^{-1} \mathcal{R}^{-1} \Big(\sum_{l = 1}^3 {\mathcal{A}}_l^i \partial_l v \Big)
\end{align*}
for all $v \in L^2(V_i \cap \R^3_+)^{12}$. Since also ${A}_{0,\pm} = (\Phi_i^{-1} \mathcal{R}^{-1} {\mathcal{A}}_0^i)_\pm$ 
and ${D}^i_\pm = (\Phi_i^{-1} \mathcal{R}^{-1} {\mathcal{D}}^i)_{\pm}$ (where we put $\mathcal{A}_0^i = A_0$ and $\mathcal{D}^i = D$ 
for $i \in \{-1,0\}$)
on $\supp \sigma_i$ for all $i \in \N_{-1}$, the definition of the maps $\mathcal{U}^i(f^*(u),u)$ and~\eqref{EquationDefinitionOfLocalizedData}
imply the equality
\begin{align*}
 \mathcal{L}_\pm u_\pm &= \sum_{i = -1}^\infty \sigma_{i,\pm} \Big(\Phi_i^{-1} \mathcal{R}_i^{-1} \Big({\mathcal{A}}_0^i \partial_t \mathcal{U}^i(f^*(u),u)
 + \sum_{j = 1}^3 {\mathcal{A}}_j^i \partial_j \mathcal{U}^i(f^*(u),u) \\
  &\quad  + {\mathcal{D}}^i \mathcal{U}^i(f^*(u),u) \Big)\Big)_\pm 
    + \sum_{i = -1}^\infty \sum_{j = 1}^3 {A}_j^{\operatorname{co}} \partial_j \sigma_{i,\pm} (\Phi_i^{-1} \mathcal{R}_i^{-1} \mathcal{U}^i(f^*(u),u))_{\pm} \\
     &= \!\sum_{i = -1}^\infty\! \Big[\sigma_{i,\pm} (\Phi_i^{-1} \mathcal{R}_i^{-1} f^i(f^*(u),u))_\pm 
    +  \!\sum_{j = 1}^3\! {A}_j^{\operatorname{co}} \partial_j \sigma_{i,\pm} (\Phi_i^{-1} \mathcal{R}^{-1} \mathcal{U}^i(f^*(u),u))_{\pm} \Big] \\
     &= \sum_{i = -1}^\infty\! \Big[\sigma_{i,\pm} \theta_{i,\pm} f^*(u)_\pm + \sum_{j = 1}^3 {A}_j^{\operatorname{co}}
     \Big[ \sigma_{i,\pm}   \partial_j \theta_{i,\pm} u_\pm +\partial_j \sigma_{i,\pm} (\Phi_i^{-1} \mathcal{R}_i^{-1} w^i)_\pm\Big]\Big].
\end{align*}
where $w^i:=\mathcal{U}^i(f^*(u),u))$.
Employing that $\sigma_i=1$ on the support of $\theta_i$, that $(\theta_i)_{i \in \N_{-1}}$ is a partition of unity, and the defining property 
of $f^*(u)$, i.e.~\eqref{EquationDefiningEquationForfstar},
we deduce
\begin{align*}
 \mathcal{L}_\pm u_\pm &=\! \sum_{i = -1}^\infty\! \!\Big[\theta_{i,\pm} f^*(u)_\pm \!+ \!\sum_{j = 1}^3\! {A}_j^{\operatorname{co}} \partial_j \theta_{i,\pm} u_\pm \!+\! \sum_{j = 1}^3 \!{A}_j^{\operatorname{co}} \partial_j \sigma_{i,\pm} (\Phi_i^{-1} \mathcal{R}_i^{-1} \mathcal{U}^i(f^*(u),u))_\pm \Big]\\
     &= f^*(u)_\pm + \sum_{i = -1}^\infty \sum_{j = 1}^3 {A}_j^{\operatorname{co}} \partial_j \sigma_{i,\pm} (\Phi_i^{-1} \mathcal{R}_i^{-1} \mathcal{U}^i(f^*(u),u))_\pm = f_\pm.
\end{align*}
Since the covering $(U_i)_{i \in \N_{-1}}$ is locally finite, we can compute
\begin{align*}
 \Tr_{J \times \Sigma,\pm} (B_\Sigma u) &= \Tr_{J \times \Sigma} (B_\Sigma \cdot (\mathcal{S}(u)_+, \mathcal{S}(u)_{-})) = \Tr_{J \times \Sigma} \Big[B_\Sigma \!\sum_{i = 1}^\infty \!\sigma_i \Phi_i^{-1} \mathcal{U}^i(f^*(u),u) \Big] \\
 &= \sum_{i = 1}^\infty \tr_{\Sigma} \sigma_i \Tr_{J \times \Sigma} (B_\Sigma \Phi_i^{-1}  \mathcal{U}^i(f^*(u),u)) \\
 &= \sum_{i = 1}^\infty \tr_{\Sigma} (\sigma_i) \kappa_i^{-1} \Tr_{J \times \Sigma} \Big(\Phi_i^{-1}\big(\omega_i \Phi_i(\kappa_i B_\Sigma) \,  \mathcal{U}^i(f^*(u),u) \big) \Big), 
\end{align*}
using $\Phi_i^{-1} \omega_i = 1$ on $\supp \sigma_i$. The identity $\hat{B}^i = \omega_i \Phi_i(\kappa_i B)$ on $\supp \sigma_i$
then yields
\begin{align*}
 \Tr_{J \times \Sigma} (B_\Sigma u) &= \sum_{i = 1}^\infty \tr_{\Sigma} (\sigma_i) \kappa_i^{-1} \Tr_{J \times \Sigma} \Big(\Phi_i^{-1}\big(\hat{B}^i \,  \mathcal{U}^i(f^*(u),u) \big) \Big) \\
 &= \sum_{i = 1}^\infty \tr_{\Sigma} (\sigma_i) \kappa_i^{-1} \tilde{\Phi}_i^{-1} \Tr_{J \times \partial \R^3_+} ((R^i)^{-1} R^i \hat{B}^i \, \mathcal{U}^i(f^*(u),u) ). 
\end{align*}
Because $\mathcal{U}^i(f^*(u),u)$ solves the initial boundary value problem~\eqref{EquationIBVPForExistence} with 
the boundary value $g^i$ defined in~\eqref{EquationDefinitionOfLocalizedData} for every $i \in \N$, we arrive at
\begin{align*}
 \Tr_{J \times \Sigma} (B u) &= \sum_{i = 1}^\infty \tr_{\Sigma} (\sigma_i) \kappa_i^{-1} \tilde{\Phi}_i^{-1} \Tr_{J \times \partial \R^3_+} \Big((R^i)^{-1} R^i \hat{B}^i \, \mathcal{U}^i(f^*(u),u) \Big) \\
  &= \sum_{i = 1}^\infty \tr_{\Sigma} (\sigma_i) \kappa_i^{-1} \tilde{\Phi}_i^{-1}\Big( \tr_{\partial \R^3_+} ((R^i)^{-1}) \, g^i_{z(i) \rightarrow 0}\Big) \\
  &= \sum_{i = 1}^\infty \tr_{\Sigma} (\sigma_i) \kappa_i^{-1} \tilde{\Phi}_i^{-1}\Big( \tr_{\partial \R^3_+} ((R^i)^{-1}) \, \tr_{\partial \R^3_+} (R^i) \tilde{\Phi}_i(\tr_{\Sigma}(\theta_i) \kappa_i g) \Big) \\
  &= \sum_{i = 1}^\infty \tr_{\Sigma} (\sigma_i \theta_i) g = \sum_{i = 1}^\infty \tr_{\Sigma} ( \theta_i) g = g,
\end{align*}
where $g^i_{z(i) \rightarrow 0}$ denotes the vector we get by adding a zero in the $z(i)$-th and $z(i)+3$-th component of $g^i$. 
Moreover, we get
\begin{align*}
	\Tr_{J \times \partial G}(B_{\partial G} u) = \Tr_{J \times \partial G}(B_{\partial G} \mathcal{S}(u))
	= \Tr_{J \times \partial G}(B_{\partial G} \mathcal{U}^{-1}(f^*(u),u)) = 0
\end{align*}
as $\mathcal{U}^{-1}(f^*(u),u)$ solves the problem~\eqref{EquationIBVPPartialGForExistence}. Similarly it follows
\begin{align*}
 u(0) &= \mathcal{S}(u)(0) = \sum_{i = -1}^\infty \sigma_i \Phi_i^{-1} \mathcal{R}_i^{-1} \mathcal{U}^i(f^*(u),u)(0) 
      = \sum_{i = -1}^\infty \sigma_i \Phi_i^{-1} \mathcal{R}_i^{-1} u_0^i  \\
      &= \sum_{i = -1}^\infty \sigma_i \Phi_i^{-1} \mathcal{R}_i^{-1} \mathcal{R}_i \Phi_i(\theta_i u_0) = \sum_{i = -1}^\infty \sigma_i \theta_i u_0  = \sum_{i = -1}^\infty  \theta_i u_0 = u_0. 
\end{align*}
We conclude that $u$ is a solution of~\eqref{EquationIBVPIntroduction} in $\Gpmdom{m}$.
\end{proof}


\section{A priori estimates for the linear problem}
\label{SectionAPrioriEstimates}
In the previous section we have reduced~\eqref{EquationIBVPIntroduction} to the system
\begin{equation}
  \label{IBVP}
\left\{\begin{aligned}
   \mathcal{A}_0 \partial_t u + \sum_{j=1}^3 \mathcal{A}_j \partial_j u + \mathcal{D} u  &= f, \quad &&x \in \R^3_+, \quad &t \in J; \\
   B u &= g, \quad &&x \in \partial \R^3_+, &t \in J; \\
   u(0) &= u_0, \quad &&x \in \R^3_+;
\end{aligned}\right.
\end{equation}
on $\R^3_+$ with $\mathcal{A}_3 = \tilde{\mathcal{A}}_3^{\operatorname{co}}$, $B = \mathcal{B}^{\operatorname{co}}$, 
and $\mathcal{A}_1, \mathcal{A}_2 \in \Fcoeff{m}{cp}$, cf.\  \eqref{EquationDefinitionOfLargeAj}
and \eqref{eq:B-trafo}. Here we fix  $T' > 0$ and assume that $J = (0,T)$ for a time $T \in (0,T')$.

In this section we derive a priori estimates for $\G{m}$-solutions of~\eqref{IBVP}.
A (weak) \emph{solution} of~\eqref{IBVP} is a function $u \in C(\clJ, \Ltwoh)$ 
with $\mathcal{L}(\mathcal{A}_0, \ldots, \mathcal{A}_3, \mathcal{D}) u = f$ in the weak sense, $\Tr_{J \times \partial \R^3_+}(B u) = g$ on 
$J \times \partial \R^3_+$, and $u(0) = u_0$.

We first state the basic wellposedness result on $L^2$-level which directly follows from Proposition~5.1
in~\cite{Eller} because of the formulas \eqref{EquationDecompositionOfA3co}. The precise form of the constants 
is a consequence of  the proof in~\cite{Eller}.
\begin{lem}
 \label{LemmaEllerResult}
 Let $\eta > 0$ and $r \geq r_0 > 0$. Take $\mathcal{A}_0 \in \Fupdwl{0}{\operatorname{cp}}{\eta}$, 
 $\mathcal{A}_1, \mathcal{A}_2 \in \Fpmcoeff{0}{\operatorname{cp}}$ with $\|\mathcal{A}_i\|_{W^{1,\infty}(\Omega)} \leq r$
 and $\|\mathcal{A}_i(0)\|_{L^\infty(\R^3_+)} \leq r_0$ for all $i \in \{0, 1, 2\}$, and 
 $\mathcal{A}_3 = \tilde{\mathcal{A}}_3^{\operatorname{co}}$. Let $\mathcal{D} \in L^\infty(\Omega)$ 
 with  $\|\mathcal{D}\|_{L^\infty(\Omega)} \leq r$ and $B = \mathcal{B}^{\operatorname{co}}$.
 Choose data $f \in \Ltwoa$, $g \in L^2(J, H^{1/2}(\partial \R^3_+))$, and $u_0 \in \Ltwoh$. Then~\eqref{IBVP} has a unique 
 solution $u$ in $C(\clJ, \Ltwoh)$, 
 and there exists a number $\gamma_0 = \gamma_0(\eta,r) \geq 1$ such that we obtain 
 \begin{align}
 \label{EquationEllerEstimateInL2}
 \sup_{t \in J}&\Ltwohn{e^{-\gamma t}u(t)}^2 + \gamma \Ltwoan{u}^2 \nonumber\\
 &\leq C_{0,0}\Ltwohn{u_0}^2 + C_{0,0} \|g\|_{L^2_\gamma(J, H^{1/2}(\partial \R^3_+))}^2 +  \frac{C_0}{\gamma} \Ltwoan{f}^2
\end{align}
for all $\gamma \geq \gamma_0$, where $C_0 = C_0(\eta, r)$ 
and $C_{0,0} = C_{0,0}(\eta, r_0)$.
\end{lem}

The a priori estimates for the $\alpha$th tangential and time derivatives of a regular solution of \eqref{IBVP} now follow in a standard way: 
These derivatives  satisfy \eqref{IBVP} with new data $f_\alpha$, $g_\alpha$ and $u_{0,\alpha}$, where $f_\alpha$ also
contains commutator terms involving $ \mathcal{A}_0$,  $\mathcal{A}_1$,  $\mathcal{A}_2$, and $\mathcal{D}$. On the resulting problem 
one can apply the $L^2$-estimate \eqref{EquationEllerEstimateInL2}. The differentiated system has the same structure as 
the corresponding problem (3.4) in \cite{SpitzMaxwellLinear}, and hence the proof of the next result is analogous to that given there.
 It is thus omitted. We use the space $\Hata{m}$ of
 those maps $v \in \Ltwoa$ with $\partial^\alpha v \in \Ltwoa$ for all 
$\alpha \in \N_0^4$ with $|\alpha| \leq m$ and $\alpha_3 = 0$. It is equipped with its natural norm.
\begin{lem}
\label{LemmaCentralEstimateInTangentialDirections}
Let $\eta > 0$, $r \geq r_0 > 0$, $m \in \N$, and $\tilde{m}=\max\{m,3\}$.
Take $\mathcal{A}_0 \in \Fupdwl{\tilde{m}}{\operatorname{cp}}{\eta}$, $\mathcal{A}_1, \mathcal{A}_2 \in \Fcoeff{\tilde{m}}{\operatorname{cp}}$, 
$\mathcal{A}_3 = \tilde{\mathcal{A}}_3^{\operatorname{co}}$, $\mathcal{D} \in \Fuwl{\tilde{m}}{\operatorname{cp}}$, and 
$B = \mathcal{B}^{\operatorname{co}}$ with
\begin{align*}
	&\Fnorm{\tilde{m}}{\mathcal{A}_i} \leq r, \quad \Fnorm{\tilde{m}}{\mathcal{D}} \leq r, \\ 
	&\max\{\Fvarnorm{\tilde{m}-1}{\mathcal{A}_i(0)}, \max_{1 \leq j \leq m-1} \Hhn{\tilde{m}-1-j}{\partial_t^j \mathcal{A}_0(0)}\} \leq r_0, \\
	&\max\{\Fvarnorm{\tilde{m}-1}{\mathcal{D}(0)}, \max_{1 \leq j \leq m-1} \Hhn{\tilde{m}-1-j}{\partial_t^j \mathcal{D}(0)}\} \leq r_0,
\end{align*}
for all $i \in \{0, 1, 2\}$.
Choose data $f \in \Hata{m}$, $g \in \E{m}$, and $u_0 \in \Hh{m}$. 
Assume that the solution $u$ of~\eqref{IBVP} belongs to $\G{m}$. 
Then there exists a parameter $\gamma_m = \gamma_m(\eta, r) \geq 1$ such that $u$ satisfies
\begin{align*} 
  \sum_{\substack{|\alpha| \leq m \\ \alpha_3 = 0}}\! \! \Gnorm{0}{\partial^\alpha u}^2 \! + \!\gamma \Hatan{m}{u}^2  
  &\leq  C_{m,0} \Big[\sum_{j = 0}^{m-1}\!\! \Hhn{m-1-j}{\partial_t^j f(0)}^2 \!+ \!\Enorm{m}{g}^2\\
   &\ \ + \Hhn{m}{u_0}^2 \Big]  \! +   \frac{C_m}{\gamma} \Big[ \Hatan{m}{f}^2 \!+\! \Gnorm{m}{u}^2 \Big], 
\end{align*}
for all $\gamma \geq \gamma_0$, where $C_m = C_m(\eta, r, T')$,  
and $C_{m,0} = C_{m,0}(\eta, r_0)$.
\end{lem}


The full $H^m$-norm of solutions $u$ to~\eqref{IBVP} cannot be controlled in this way since normal derivatives destroy
the boundary condition. From  the system~\eqref{IBVP} itself one can read off regularity of normal derivatives 
of the tangential components of $u$ because of the structure of the boundary matrix $\mathcal{A}_3 = \tilde{\mathcal{A}}_3^{\operatorname{co}}$.
The remaining four components will be recovered by means of cancellation properties of the Maxwell equations 
which imply that the `generalized divergence' $\Div(\mathcal{A}_1, \mathcal{A}_2, \mathcal{A}_3)$ of the Maxwell operator 
only contains first order derivatives. 

To define this concept, take $\mathcal{A}_1, \mathcal{A}_2 \in \Fcoeff{0}{cp}$
and $\mathcal{A}_3 = \tilde{\mathcal{A}}_3^{\operatorname{co}}$.
In particular,  there are functions $\mu_{lj} \in \Fuwlk{0}{cp}{1}$ such that
\begin{equation}
\label{EquationFixingmuljForCoefficients}
 \mathcal{A}_j = \sum_{l = 1}^3 \mathcal{A}_l^{\operatorname{co}} \mu_{lj} \quad \text{for \ } j \in \{1, 2\}
 \qquad \text{and}\qquad \mu_{13} = \mu_{23} = 0, \quad  \mu_{33} = 1,
\end{equation}
see \eqref{EquationDefinitionFmcoeff} and \eqref{EquationDefinitionOfLargeAj}. We now set 
\begin{equation}
\label{EquationDefinitionOfTildeMu}
 \mu=(\mu_{lj})_{l,j=1}^3, \qquad
\hat{\mu} = \begin{pmatrix}
               \mu_{11} &\mu_{12} &\mu_{13} \\
               \mu_{21} &\mu_{22} &\mu_{23} \\
               \mu_{31} &\mu_{32} &-\mu_{33} \\
              \end{pmatrix},
\qquad    \tilde{\mu} = \begin{pmatrix}
		  \mu &0 &0 &0\\
		  0   &\mu &0 &0 \\
		  0  &0 &\hat{\mu} &0 \\
		  0   &0 &0  &\hat{\mu}
        \end{pmatrix}, 
 \end{equation}
 and for $h \in \Ltwoh^{12}$ we define
 \begin{equation}
 \label{EquationDefinitionOfGeneralizedDiv}
  \Div(\mathcal{A}_1, \mathcal{A}_2, \mathcal{A}_3) h = \sum_{k = 1}^3 \Big( (\tilde{\mu}^T \nabla h)_{kk}, (\tilde{\mu}^T \nabla h)_{(k+3) k}, 
   (\tilde{\mu}^T \nabla h)_{(k+6)k},  (\tilde{\mu}^T \nabla h)_{(k+9)k} \Big).
 \end{equation}
 In view of the iteration and regularization process below, in the next proposition we treat solutions and data which are 
 a bit less regular than needed in this section and we consider the initial value problem
\begin{equation}
  \label{IVP}
 \left\{\begin{aligned}
   \mathcal{L}(\mathcal{A}_0,\ldots, \mathcal{A}_3, \mathcal{D}) u &= f, \qquad &&x \in \R^3_+, \qquad &&t \in J; \\
   u(0) &= u_0,  &&x \in \R^3_+.
 \end{aligned} \right.
\end{equation}
A \emph{solution} of~\eqref{IVP} is a function $u \in C(\clJ, \Ltwoh)$ with $u(0) = u_0$ in $\Ltwoh$ and $\mathcal{L} u = f$ in 
$H^{-1}(\Omega)$. The following result is the core step in our regularity theory.
\begin{prop}
\label{PropositionCentralEstimateInNormalDirection}
Let $T' > 0$, $\eta > 0$, $\gamma \geq 1$, and $r \geq r_0 >0$. 
Take coeffcients  $\mathcal{A}_0 \in \Fupdwl{0}{\operatorname{cp}}{\eta}$, $\mathcal{A}_1, \mathcal{A}_2 \in \Fcoeff{0}{\operatorname{cp}}$, 
$\mathcal{A}_3 = \tilde{\mathcal{A}}_3^{\operatorname{co}}$,
and $\mathcal{D} \in \Fuwl{0}{\operatorname{cp}}$ with 
\begin{align*}
	&\|\mathcal{A}_i\|_{W^{1,\infty}(\Omega)} \leq r, \quad \|\mathcal{D}\|_{W^{1,\infty}(\Omega)} \leq r, \\
	&\|\mathcal{A}_i(0)\|_{L^\infty(\R^3_+)} \leq r_0,  \quad \|\mathcal{D}(0)\|_{L^\infty(\R^3_+)} \leq r_0
\end{align*}
for all $i \in \{0, 1, 2\}$. 
Choose data $f \in \G{0}$ with $\Div(\mathcal{A}_1, \mathcal{A}_2, \mathcal{A}_3) f \in \Ltwoa$ and $u_0 \in \Hh{1}$.
Let $u$ solve~\eqref{IVP} and assume that $u$ is an element of $C^1(\clJ, L^2(\R^3_+)) \cap C(\clJ, \Hhta{1}) \cap L^\infty(J,\Hh{1})$.
Then $u$ belongs to $\G{1}$ and there are constants $C_{1,0} = C_{1,0}(\eta,r_0) \geq 1$ and $C_1 = C_1(\eta, r, T') \geq 1$ 
such that it satisfies
\begin{align}
 \label{EquationFirstOrderFinal}
    \Gnorm{0}{\nabla u}^2 &\leq e^{C_1 T}  \Big((C_{1,0} + T C_1)\Big(\sum_{j = 0}^2 \Ltwohnt{\partial_j u}^2 
      +   \Ltwohnt{f}^2 + \Hhn{1}{u_0}^2 \Big)\nonumber \\
    & \hspace{4em}  + \frac{C_1}{\gamma} \Ltwoan{\Div(\mathcal{A}_1, \mathcal{A}_2, \mathcal{A}_3) f}^2  \Big).
\end{align}
If $f$ is even contained in $\Ha{1}$, we obtain
\begin{align}
 \label{EquationFirstOrderFinalVariant}
    \Gnorm{0}{\nabla u}^2 &\leq e^{C_1 T}  \Big(\!(C_{1,0} + T C_1)\Big(\sum_{j = 0}^2 \Ltwohnt{\partial_j u}^2  
      +   \Ltwohn{f(0)}^2 + \Hhn{1}{u_0}^2\Big)\nonumber \\
    & \hspace{4em}  + \frac{C_1}{\gamma} \Hangamma{1}{f}^2   \Big).
\end{align}
Finally, if $f$ merely belongs to $\Ltwoa$ with $\Div(\mathcal{A}_1, \mathcal{A}_2, \mathcal{A}_3) f \in \Ltwoa$, we still have
\begin{align}
 \label{EquationFirstOrderL2}
 \Ltwoan{\nabla u}^2 &\leq e^{C_1 T} \Big((C_{1,0} + T C_1)  \Big(  \sum_{j = 0}^2 \Ltwoan{\partial_j u}^2  + \Ltwoan{f}^2 + \Hhn{1}{u_0}^2\Big) \nonumber\\
 &\hspace{4em} + \frac{C_1}{\gamma}  \Ltwoan{\Div(\mathcal{A}_1, \mathcal{A}_2, \mathcal{A}_3) f}^2 \Big).
\end{align}
\end{prop}

\begin{proof}
We have  to show that $\partial_3 u\in C(\clJ, \Ltwoh)$ and that inequalities~\eqref{EquationFirstOrderFinal} to~\eqref{EquationFirstOrderL2} are
true. We employ the matrix $\tilde{\mu}$ from~\eqref{EquationDefinitionOfTildeMu}.
Recall that the coefficients  $\mathcal{A}_l$ are given by \eqref{EquationFixingmuljForCoefficients} 
and $\mathcal{A}_3 = \tilde{\mathcal{A}}_3^{\operatorname{co}}$,
$\mathcal{A}_l^{\operatorname{co}}$  and $\tilde{\mathcal{A}}_3^{\operatorname{co}}$ by  \eqref{EquationDefinitionOfLargeAj}, 
as well as $A_l^{\operatorname{co}}$ and $J_l$ by \eqref{EquationDefinitionOfAj}, for $l\in\{1,2,3\}$
Morever, $J_{l;mn} = - \epsilon_{lmn}$ for all $l, m, n \in \{1,2,3\}$ and the Levi-Civita symbol, i.e.,
\begin{equation*}
	\epsilon_{ijk} =  \left\{\begin{aligned}
 &1  \quad &&\text{if } (i,j,k) \in \{(1,2,3), (2,3,1), (3,1,2)\},\\
 -&1   &&\text{if } (i,j,k) \in \{(3,2,1), (2,1,3), (1,3,2)\},\\
  &0 &&\text{else}.
  \end{aligned}\right.,
\end{equation*}
Since the coefficients are Lipschitz, we can differentiate
\begin{align}
\label{EquationDifferentiatingMTransposedAzeroGradu}
  \partial_t (\tilde{\mu}^T \mathcal{A}_0  \nabla u) &=  \tilde{\mu}^T \partial_t \mathcal{A}_0 \nabla u + \tilde{\mu}^T \mathcal{A}_0 \partial_t \nabla u \notag\\
  &=  \tilde{\mu}^T \partial_t \mathcal{A}_0 \nabla u + \tilde{\mu}^T \mathcal{A}_0 \nabla \Big(\mathcal{A}_0^{-1} \Big(f - \sum_{j = 1}^3 \mathcal{A}_j \partial_j u - \mathcal{D} u \Big)\Big) \nonumber\\
  &=  \tilde{\mu}^T \partial_t \mathcal{A}_0 \nabla u  + \tilde{\mu}^T \mathcal{A}_0 \nabla \mathcal{A}_0^{-1}\Big(f - \sum_{j = 1}^3 \mathcal{A}_j \partial_j u - \mathcal{D} u \Big)  \nonumber \\
  &\quad + \tilde{\mu}^T \nabla f - \tilde{\mu}^T \sum_{j = 1}^2 \nabla \mathcal{A}_j \partial_j u - \tilde{\mu}^T \nabla \mathcal{D} u 
  - \tilde{\mu}^T \mathcal{D} \nabla u - \tilde{\mu}^T \sum_{j = 1}^3 \mathcal{A}_j \nabla \partial_j u\nonumber \\
  &=: \Lambda- \tilde{\mu}^T \sum_{j = 1}^3 \mathcal{A}_j \nabla \partial_j u 
\end{align}
in $L^\infty(J, \Hh{-1})$. Here we use~\eqref{IVP} and write
$ ((\nabla \mathcal{A}_0^{-1}) h)_{jk} := \sum_{l=1}^{12} \partial_k \mathcal{A}_{0;jl}^{-1} h_l $
etc. Note that $\Lambda$ only contains first order spatial derivatives of $u$. 
We next compute
\begin{align}
\label{EquationTraceInvolvingLeviCivita}
 \sum_{k = 1}^3 &\Big(\tilde{\mu}^T \sum_{j = 1}^3 \mathcal{A}_j \nabla \partial_j u \Big)_{kk} 
  = \sum_{j,k = 1}^3 \sum_{l,p = 1}^{12} \tilde{\mu}^T_{kl} \mathcal{A}_{j;lp} \partial_k \partial_j u_p 
  = \sum_{j,k,l = 1}^3 \sum_{p = 1}^{12} \mu_{lk} \mathcal{A}_{j;lp} \partial_k \partial_j u_p \nonumber\\
  &= \sum_{j,k,l,n,p = 1}^3  \mu_{lk} A^{\operatorname{co}}_{n;l(p+3)} \mu_{nj} \partial_k \partial_j u_{p+3}
   = \sum_{j,k,l,n,p = 1}^3  \epsilon_{nlp} \mu_{lk} \mu_{nj} \partial_k \partial_j u_{p+3}\\
 \label{EquationTraceInterchangingIndices}
 &= \sum_{j,k,l,n,p = 1}^3  \epsilon_{lnp} \mu_{nj}  \mu_{lk} \partial_j \partial_k u_{p+3}
 = - \sum_{j,k,l,n,p = 1}^3  \epsilon_{nlp} \mu_{lk} \mu_{nj} \partial_k \partial_j u_{p+3}\,,
\end{align}
exchanging the indices $l$ and $n$ as well as $k$ and $j$ in the penultimate step.
Equations~\eqref{EquationTraceInvolvingLeviCivita} and~\eqref{EquationTraceInterchangingIndices}  yield
\begin{equation}
 \label{EquationTraceSmaller3EqualsZero}
 \sum_{k = 1}^3 \Big(\tilde{\mu}^T \sum_{j = 1}^3 \mathcal{A}_j \nabla \partial_j u \Big)_{kk} = 0.
\end{equation}
Analogously, it follows
\begin{align}
 \label{EquationTraceLarger3EqualsZero}
 &\sum_{k = 1}^3 \Big(\tilde{\mu}^T \sum_{j = 1}^3 \mathcal{A}_j \nabla \partial_j u \Big)_{(k+3)k} = 0.
\end{align}
In the other components we take care of the extra signs in \eqref{EquationDefinitionOfTildeMu} and \eqref{EquationDefinitionOfLargeAj}, calculating
\begin{align}
\label{EquationTraceInvolvingLeviCivita2}
 \sum_{k = 1}^3 \Big(\tilde{\mu}^T &\sum_{j = 1}^3 \mathcal{A}_j \nabla \partial_j u \Big)_{(k+6)k} 
  = \sum_{j,k = 1}^3 \sum_{l,p = 1}^{12} \tilde{\mu}^T_{(k+6)l} \mathcal{A}_{j;lp} \partial_k \partial_j u_p \nonumber \\
  &= \sum_{j,k,l = 1}^3 \sum_{p = 1}^{12} \hat{\mu}_{lk} \mathcal{A}_{j;(l+6)p} \partial_k \partial_j u_p 
   = \sum_{j,k,l,p = 1}^3  \hat{\mu}_{lk} \mathcal{A}_{j;(l+6)(p+9)} \partial_k \partial_j u_{p+9} \nonumber\\
  &= \sum_{j,k,l,n,p = 1}^3  \mu_{lk} (-1)^{\delta_{3l} \delta_{3k}} A^{\operatorname{co}}_{n;l(p+3)} \mu_{nj}(-1)^{\delta_{3j} \delta_{3n}} \partial_k \partial_j u_{p+9} \nonumber\\
   &= \sum_{j,k,l,n,p = 1}^3  \epsilon_{nlp} (-1)^{\delta_{3l} \delta_{3k}} (-1)^{\delta_{3n} \delta_{3j} } \mu_{lk} \mu_{nj} \partial_k \partial_j u_{p+9}\\
 &= \sum_{j,k,l,n,p = 1}^3  \epsilon_{lnp} (-1)^{\delta_{3n} \delta_{3j}} (-1)^{\delta_{3l} \delta_{3k}} \mu_{nj}  \mu_{lk} \partial_j \partial_k u_{p+9} \notag \\
& = - \sum_{j,k,l,n,p = 1}^3  \epsilon_{nlp} (-1)^{\delta_{3l} \delta_{3k}} (-1)^{\delta_{3n} \delta_{3j} } \mu_{lk} \mu_{nj} \partial_k \partial_j u_{p+9}.
 \label{EquationTraceInterchangingIndices2}
\end{align}
Comparing the expressions~\eqref{EquationTraceInvolvingLeviCivita2} and~\eqref{EquationTraceInterchangingIndices2}, we infer
\begin{equation}
 \label{EquationTraceSmaller9EqualsZero}
 \sum_{k = 1}^3 \Big(\tilde{\mu}^T \sum_{j = 1}^3 \mathcal{A}_j \nabla \partial_j u \Big)_{(k+6)k} = 0.
\end{equation}
Proceeding similarly, we derive
\begin{equation}
 \label{EquationTraceLarger9EqualsZero}
 \sum_{k = 1}^3 \Big(\tilde{\mu}^T \sum_{j = 1}^3 \mathcal{A}_j \nabla \partial_j u \Big)_{(k+9)k} = 0.
\end{equation}
Integrating in time, the formulas~\eqref{EquationDifferentiatingMTransposedAzeroGradu}, \eqref{EquationTraceSmaller3EqualsZero} 
\eqref{EquationTraceLarger3EqualsZero}, \eqref{EquationTraceSmaller9EqualsZero} and  \eqref{EquationTraceLarger9EqualsZero}
imply the identities
\begin{equation*}
 \sum_{k = 1}^3 (\tilde{\mu}^T \mathcal{A}_0  \nabla u)_{(k+i)k}(t) = \sum_{k = 1}^3 (\tilde{\mu}^T \mathcal{A}_0  \nabla u)_{(k+i)k}(0) +  \sum_{k = 1}^3 \int_0^t\Lambda_{(k+i)k}(s) ds
\end{equation*}
in $\Hh{-1}$ for all $t \in \clJ$ and $i \in \{0,3,6,9\}$. The function $\Lambda$ is  integrable with values in $\Ltwoh$ so that the equality holds in $\Ltwoh$ 
for all $t \in \clJ$. 
Let $t \in \clJ$.
 We denote the $k$-th row respectively the $k$-th column of a matrix 
$N$ by $N_{k \cdot}$ respectively $N_{\cdot k}$, and we set 
\begin{align*}
 &F_{13+l}(t) = \!\sum_{k = 1}^3 (\tilde{\mu}^T \mathcal{A}_0  \nabla u)_{(k+3l)k}(0) +  \sum_{k = 1}^3 \int_0^t\!\!\Lambda_{(k+3l)k}(s) ds 
  - \!\sum_{k = 1}^2 (\tilde{\mu}^T \mathcal{A}_0)_{(k+3l) \cdot} \partial_k u(t),\notag\\
 & (F_1, \ldots, F_{12})^T = f - \sum_{j = 0}^2 \mathcal{A}_j \partial_j u - \mathcal{D} u
\end{align*}
for $l\in\{0,1,2,3\}$.  The map $F = (F_1, \ldots, F_{16})^T$ belongs to $C(\clJ, \Ltwoh)$ and 
\begin{equation}
\label{EquationFullRankEquationforNormalDerivativeOfu}
 \breve{\mu} \partial_3 u = F, \qquad \text{setting \ }\breve{\mu} = \begin{pmatrix}
            \mathcal{A}_3 \\
            (\tilde{\mu}^T \mathcal{A}_0)_{3 \cdot} \\
            (\tilde{\mu}^T \mathcal{A}_0)_{6 \cdot} \\
            (\tilde{\mu}^T \mathcal{A}_0)_{9 \cdot} \\
            (\tilde{\mu}^T \mathcal{A}_0)_{12 \cdot}
           \end{pmatrix} \in \F{0}^{16 \times 12}.
\end{equation}
Let $\zeta = \tilde{\mu}^T \mathcal{A}_0$ and the matrix $G_1$ be equal to  
\begin{align*}
 \begin{pmatrix}
        1 &0 &0 &0 &0 &0 &0 &0 &0 &0 &0 &0 &0 &0 &0 &0 \\
        0 &-1 &0 &0 &0 &0 &0 &0 &0 &0 &0 &0 &0 &0 &0 &0\\
        0 &0 &1 &0 &0 &0 &0 &0 &0 &0 &0 &0 &0 &0 &0 &0\\
        0 &0 &0 &-1 &0 &0 &0 &0 &0 &0 &0 &0 &0 &0 &0 &0\\
        0 &0 &0 &0 &1 &0 &0 &0 &0 &0 &0 &0 &0 &0 &0 &0\\
        0 &0 &0 &0 &0 &1 &0 &0 &0 &0 &0 &0 &0 &0 &0 &0\\
        0 &0 &0 &0 &0 &0 &-1 &0 &0 &0 &0 &0 &0 &0 &0 &0\\
        0 &0 &0 &0 &0 &0 &0 &1 &0 &0 &0 &0 &0 &0 &0 &0\\
        0 &0 &0 &0 &0 &0 &0 &0 &1 &0 &0 &0 &0 &0 &0 &0\\
        0 &0 &0 &0 &0 &0 &0 &0 &0 &1 &0 &0 &0 &0 &0 &0\\
        0 &0 &0 &0 &0 &0 &0 &0 &0 &0 &-1 &0 &0 &0 &0 &0\\
        0 &0 &0 &0 &0 &0 &0 &0 &0 &0 &0 &1 &0 &0 &0 &0\\
        -\zeta_{3,5} & \zeta_{3,4} &0 &\zeta_{3,2} &-\zeta_{3,1} &0 &\zeta_{3,11} &-\zeta_{3,10} &0 &-\zeta_{3,8} &\zeta_{3,7} &0 &1 &0 &0 &0 \\
        -\zeta_{6,5} & \zeta_{6,4} &0 &\zeta_{6,2} &-\zeta_{6,1} &0 &\zeta_{6,11} &-\zeta_{6,10} &0 &-\zeta_{6,8} &\zeta_{6,7} &0 &0 &1 &0 &0 \\
        \zeta_{9,5} & -\zeta_{9,4} &0 &-\zeta_{9,2} &\zeta_{9,1} &0 &-\zeta_{9,11} &\zeta_{9,10} &0 &\zeta_{9,8} &-\zeta_{9,7} &0 &0 &0 &-1 &0 \\
        \zeta_{12,5} & -\zeta_{12,4} &0 &-\zeta_{12,2} &\zeta_{12,1} &0 &-\zeta_{12,11} &\zeta_{12,10} &0 &\zeta_{12,8} &-\zeta_{12,7} &0 &0 &0 &0 &-1 \\
       \end{pmatrix}. \nonumber
\end{align*}
We derive the crucial identity 
\begin{align*}
 G_1 \breve{\mu}& = \begin{pmatrix}
                    0 &0 &0 &0 &1 &0 &0 &0 &0 &0 &0 &0 \\
                    0 &0 &0 &1 &0 &0 &0 &0 &0 &0 &0 &0\\
                    0 &0 &0 &0 &0 &0 &0 &0 &0 &0 &0 &0\\
                    0 &1 &0 &0 &0 &0 &0 &0 &0 &0 &0 &0\\
                    1 &0 &0 &0 &0 &0 &0 &0 &0 &0 &0 &0\\
                    0 &0 &0 &0 &0 &0 &0 &0 &0 &0 &0 &0\\
                    0 &0 &0 &0 &0 &0 &0 &0 &0 &0 &1 &0 \\
                    0 &0 &0 &0 &0 &0 &0 &0 &0 &1 &0 &0\\
                    0 &0 &0 &0 &0 &0 &0 &0 &0 &0 &0 &0\\
                    0 &0 &0 &0 &0 &0 &0 &1 &0 &0 &0 &0\\
                    0 &0 &0 &0 &0 &0 &1 &0 &0 &0 &0 &0\\
                    0 &0 &0 &0 &0 &0 &0 &0 &0 &0 &0 &0\\
                    0 &0 &\alpha_{3,3} &0 &0 &\alpha_{3,6} &0 &0 &\alpha_{3,9} &0 &0 &\alpha_{3,12} \\
                    0 &0 &\alpha_{6,3} &0 &0 &\alpha_{6,6} &0 &0 &\alpha_{6,9} &0 &0 &\alpha_{6,12} \\
                    0 &0 &\alpha_{9,3} &0 &0 &\alpha_{9,6} &0 &0 &\alpha_{9,9} &0 &0 &\alpha_{9,12} \\
                    0 &0 &\alpha_{12,3} &0 &0 &\alpha_{12,6} &0 &0 &\alpha_{12,9} &0 &0 &\alpha_{12,12} \\
                   \end{pmatrix},\\
           \alpha_{kn} &:= \zeta_{kn} =   \sum_{l = 1}^{12} \tilde{\mu}^T_{kl} \mathcal{A}_{0;ln}  = \mathcal{A}_{0;kn} \qquad \text{for \ } k \in \{3,6\},\\
           \alpha_{kn} &:= -\zeta_{kn} = -\sum_{l = 1}^{12} \tilde{\mu}^T_{kl} \mathcal{A}_{0;ln} = \mathcal{A}_{0;kn}\qquad \text{for \ } k \in \{9,12\},
\end{align*}
where $n \in \{3,6,9,12\}$. Here we use $\tilde{\mu}_{lk} = 1$ for $l = k$ and 
$\tilde{\mu}_{lk} = 0$ for $l \neq k$, if $k \in \{3,6\}$, as well as 
$\tilde{\mu}_{lk} = -1$ for $l = k$  and $\tilde{\mu}_{lk} = 0$ for $l \neq k$, if $k \in \{9,12\}$. Since
\begin{equation*}
 \begin{pmatrix}
        \alpha_{3,3} &\alpha_{3,6} &\alpha_{3,9} &\alpha_{3,12} \\
        \alpha_{6,3} &\alpha_{6,6} &\alpha_{6,9} &\alpha_{6,12} \\
        \alpha_{9,3} &\alpha_{9,6} &\alpha_{9,9} &\alpha_{9,12} \\
        \alpha_{12,3} &\alpha_{12,6} &\alpha_{12,9} &\alpha_{12,12}
       \end{pmatrix}
       = 
       \begin{pmatrix}
        \mathcal{A}_{0;3,3} &\mathcal{A}_{0;3,6} &\mathcal{A}_{0;3,9} &\mathcal{A}_{0;3,12} \\
        \mathcal{A}_{0;6,3} &\mathcal{A}_{0;6,6} &\mathcal{A}_{0;6,9} &\mathcal{A}_{0;6,12} \\
        \mathcal{A}_{0;9,3} &\mathcal{A}_{0;9,6} &\mathcal{A}_{0;9,9} &\mathcal{A}_{0;9,12} \\
        \mathcal{A}_{0;12,3} &\mathcal{A}_{0;12,6} &\mathcal{A}_{0;12,9} &\mathcal{A}_{0;12,12}
       \end{pmatrix}
       \geq \eta,
\end{equation*} 
this matrix has an inverse $\beta$ bounded by  $C(\eta)$.
Setting $ G_2 = \big(\begin{smallmatrix}
        I_{12 \times 12} &0 \\
        0 &\beta
       \end{smallmatrix}\big),$
we compute
\begin{equation}
\label{EquationMassMatrixAfterGaussAlgorithm}
  G_2 G_1 \breve{\mu} = \begin{pmatrix}
			      0 &0 &0 &0 &1 &0 &0 &0 &0 &0 &0 &0 \\
                    0 &0 &0 &1 &0 &0 &0 &0 &0 &0 &0 &0\\
                    0 &0 &0 &0 &0 &0 &0 &0 &0 &0 &0 &0\\
                    0 &1 &0 &0 &0 &0 &0 &0 &0 &0 &0 &0\\
                    1 &0 &0 &0 &0 &0 &0 &0 &0 &0 &0 &0\\
                    0 &0 &0 &0 &0 &0 &0 &0 &0 &0 &0 &0\\
                    0 &0 &0 &0 &0 &0 &0 &0 &0 &0 &1 &0 \\
                    0 &0 &0 &0 &0 &0 &0 &0 &0 &1 &0 &0\\
                    0 &0 &0 &0 &0 &0 &0 &0 &0 &0 &0 &0\\
                    0 &0 &0 &0 &0 &0 &0 &1 &0 &0 &0 &0\\
                    0 &0 &0 &0 &0 &0 &1 &0 &0 &0 &0 &0\\
                    0 &0 &0 &0 &0 &0 &0 &0 &0 &0 &0 &0\\
                    0 &0 &1 &0 &0 &0&0 &0 &0 &0 &0 &0 \\
                    0 &0 &0 &0 &0 &1 &0 &0 &0 &0 &0 &0 \\
                    0 &0 &0 &0 &0 &0 &0 &0 &1 &0 &0 &0 \\
                    0 &0 &0 &0 &0 &0 &0 &0 &0 &0 &0 &1 
                           \end{pmatrix} =: \tilde{M}.
\end{equation}
Equations~\eqref{EquationFullRankEquationforNormalDerivativeOfu} and~\eqref{EquationMassMatrixAfterGaussAlgorithm} yield
\begin{equation}
\label{EquationFullRankMatrixMtildeInFrontOfPartial3u}
 \tilde{M} \partial_3 u = G_2 G_1 F.
\end{equation}
The formulas in \eqref{EquationFixingmuljForCoefficients} imply the inequality
\begin{equation*}
 \|G_2 G_1 \|_{L^\infty(\Omega)} \leq C(\eta) (1 + c_0)^2 \quad\text{with}\quad
  c_0 := \max\{\max_{j = 0, \ldots, 3} \|\mathcal{A}_j \|_{L^\infty(\Omega)}, \|\mathcal{D}\|_{L^\infty(\Omega)}\}.
\end{equation*}
Since the matrix $\tilde{M}$ has rank $12$, equation \eqref{EquationFullRankMatrixMtildeInFrontOfPartial3u} shows that
$\partial_3 u$ is contained in $C(\overline{J}, \Ltwoh)$ and bounded by
\begin{equation} \label{EquationFirstEstimateForNormalDerivativeOfu}
 \Ltwohn{\partial_3 u(t)} \leq C(\eta)(1 + c_0)^2 \Ltwohn{F(t)}.
\end{equation}
This estimate is analogous to (3.29) in the proof of  Proposition~3.3 in \cite{SpitzMaxwellLinear},
where a comparable function $F$ was involved. The remaining arguments are the same as in \cite{SpitzMaxwellLinear}
and therefore omitted. They mainly consist of straightforward estimates and an application of Gronwall's 
inequality.
\end{proof}

We can now combine Lemma~\ref{LemmaEllerResult}, Lemma~\ref{LemmaCentralEstimateInTangentialDirections} and 
Proposition~\ref{PropositionCentralEstimateInNormalDirection} in an  iteration argument to establish
the desired a priori estimates of arbitrary order. This is done as in the proof of Theorem~4.4 in \cite{SpitzMaxwellLinear},
also using  the auxiliary   results from Section~\ref{SectionNotation}. Here the different structure in \eqref{IBVP} arising
from the interface condition does not play a role. So we do not give the proof.

\begin{theorem}
 \label{TheoremAPrioriEstimates}
 Let $T' > 0$, $\eta > 0$, $r \geq  r_0 >0$,  $m \in \N$, and $\tilde{m} = \max\{m,3\}$. Pick $T \in (0, T']$ and set $J = (0,T)$.
 Take coefficients $\mathcal{A}_0 \in \Fupdwl{\tilde{m}}{\operatorname{cp}}{\eta}$, $\mathcal{A}_1, \mathcal{A}_2 \in \Fcoeff{\tilde{m}}{\operatorname{cp}}$, 
 $\mathcal{A}_3 = \tilde{\mathcal{A}}_3^{\operatorname{co}}$,  
 $\mathcal{D} \in \Fuwl{\tilde{m}}{\operatorname{cp}}$, and $B = \mathcal{B}^{\operatorname{co}}$ satisfying
 \begin{align*}
  &\Fnorm{\tilde{m}}{\mathcal{A}_i} \leq r, \quad \Fnorm{\tilde{m}}{\mathcal{D}} \leq r, \\
  &\max \{\Fvarnorm{\tilde{m}-1}{\mathcal{A}_i(0)},\max_{1 \leq j \leq \tilde{m}-1} \Hhn{\tilde{m}-j-1}{\partial_t^j \mathcal{A}_0(0)}\} \leq r_0, \\
  &\max \{\Fvarnorm{\tilde{m}-1}{\mathcal{D}(0)},\max_{1 \leq j \leq \tilde{m}-1} \Hhn{\tilde{m}-j-1}{\partial_t^j \mathcal{D}(0)}\} \leq r_0
 \end{align*}
 for all $i \in \{0, 1, 2\}$.
 Choose data $f \in \Ha{m}$, $g \in \E{m}$, and $u_0 \in \Hh{m}$. Assume that the solution $u$ 
 of~\eqref{IBVP}  belongs to $\G{m}$. Then there is a number
 $\gamma_m = \gamma_m(\eta, r, T') \geq 1$ such that $u$ satisfies
 \begin{align}
  &\Gnorm{m}{u}^2  \leq (C_{m,0} + T C_m) e^{m C_1 T} \Big(  \sum_{j = 0}^{m-1} \Hhn{m-1-j}{\partial_t^j f(0)}^2 + \Enorm{m}{g}^2  \nonumber\\
      &\hspace{18em} + \Hhn{m}{u_0}^2 \Big) + \frac{C_m}{\gamma}  \Hangamma{m}{f}^2    \nonumber
 \end{align}
 for all $\gamma \geq \gamma_m$, where $C_m = C_m(\eta, r,T') \geq 1$, $C_{m,0} = C_{m,0}(\eta,r_0) \geq 1$, and 
 $C_1 = C_1(\eta,r,T')$ is a constant independent of $m$.
\end{theorem}


\section{Regularity of solutions to the linear problem}
\label{SectionRegularity}

In this section we prove that the $G_0(\Omega)$-solution $u$ of~\eqref{IBVP} actually belongs to $\G{m}$ if the data and the 
coefficients are accordingly smooth and compatible. To this aim, different regularizing techniques 
in normal, tangential, and time direction are used. We first show that regularity in time and 
in tangential directions implies regularity in normal direction. This is the crucial step in the regularization
argument, and it heavily relies on the structure of the Maxwell system. As in Proposition~\ref{PropositionCentralEstimateInNormalDirection},
we only look at the linear initial value problem~\eqref{IVP}.

\begin{lem}
 \label{LemmaRegularityInNormalDirection}
Let $\eta > 0$, $m \in \N$, and $\tilde{m} =  \max\{m,3\}$. Take coefficients $\mathcal{A}_0 \in \Fupdwl{\tilde{m}}{cp}{\eta}$, 
$\mathcal{A}_1, \mathcal{A}_2 \in \Fcoeff{\tilde{m}}{\operatorname{cp}}$, $\mathcal{A}_3 = \tilde{\mathcal{A}}_3^{\operatorname{co}}$,
and $\mathcal{D} \in \Fuwl{\tilde{m}}{cp}$. Choose data $f \in \Ha{m}$
and $u_0 \in \Hh{m}$.  Let $u$ be a solution of~\eqref{IVP} for these coefficients and data. 
Assume that $u$ belongs to $\bigcap_{j = 1}^m C^j(\clJ, \Hh{m-j})$.

Take $k \in \{1, \ldots, m\}$ and a multi-index $\alpha \in \N_0^4$ with $|\alpha| = m$, $\alpha_0 = 0$,
and $\alpha_3 = k$. 
Suppose that $\partial^\beta u$ is contained in $\G{0}$ for all $\beta \in \N_0^4$ with $|\beta| = m$ and 
$\beta_3 \leq k-1$.
Then $\partial^\alpha u$ is an element of $\G{0}$.
\end{lem}
\begin{proof} 
 I) We begin with several preparations.  Let $M_\epsilon$, $\epsilon>0$, be a standard mollifier on $\R^3$ 
 with kernel $\rho\ge0$. Let $\delta>0$. We introduce the translation operator
 \begin{align}
 \label{EquationDefinitionTranslation}
  T_\delta v(x) = v(x_1,x_2, x_3 + \delta) \quad
  \text{for } \ v \in L^1_{\operatorname{loc}}(\R^3_+) \text{ and a.e. } x \in \R^2 \times (-\delta, \infty).
 \end{align}
 Notice that $T_\delta$ maps $W^{l,p}(\R^3_+)$ continuously into 
 $W^{l,p}(\R^2 \times (-\delta, \infty))$ and that $\partial^{\tilde{\alpha}} T_\delta v = T_\delta \partial^{\tilde{\alpha}} v$
 for all $v \in W^{l,p}(\R^3_+)$, $\tilde{\alpha} \in \N_0^4$ with $|\tilde{\alpha}| \leq l$, $l \in \N_0$, and $1 \leq p \leq \infty$.
 If $v \in L^1_{\operatorname{loc}}(\R^3)$, we further define $T_{\delta} v$ by formula~\eqref{EquationDefinitionTranslation} 
 for all $\delta \in \R$.
 
 Functions which are only defined on a subset of $\R^3$ will be identified with their zero-extensions. 
  Moreover, restrictions of a map $v$ to a subset are also denoted by~$v$.
 We extend the translations $T_\delta$ to continuous operators on $H^{-1}(\R^3_+)$ by 
 setting
 \begin{align*}
    &\langle T_\delta v, \psi \rangle_{H^{-1}(\R^3_+) \times H^1_0(\R^3_+)} = \langle v, T_{-\delta} \psi \rangle_{H^{-1}(\R^3_+) \times H^1_0(\R^3_+)} 
 \end{align*}
 for all  $\psi \in H^1_0(\R^3_+)$ and $\delta > 0$. It is then straighforward to check that
  $\partial_j T_\delta v  = T_\delta \partial_j v$ for all $v \in \Ltwoh$ and $\delta > 0$.
 
 We want to apply $M_\epsilon$ to functions in 
 $L^1_{\operatorname{loc}}(\R^3_+)$ without obtaining singularities at the boundary in limit processes.
 To that purpose, we take $0 < \epsilon < \delta$ and look at the regularization $M_\epsilon T_\delta v$
 for $v\in L^1_{\operatorname{loc}}(\R^3_+)$. 
 If $v$ and $\partial_j v$ belong to $L^1_{\operatorname{loc}}(\R^3_+)$, then also $M_\epsilon T_\delta v$ 
 has a weak derivative in $\R^3_+$ and
  $\partial_j M_\epsilon T_\delta  v = M_\epsilon T_\delta  \partial_j v$ for all $j \in \{1,2,3\}$.
 
 We set $\tilde{\rho}(x) = \rho(-x)$ for all $x \in \R^3$ and denote 
the corresponding mollifier by $\tilde{M}_\epsilon$.  A straightforward computation shows that
 \begin{align}
    \label{DefinitionConvolutionOnHMinusOneDomain}
 &\langle M_\epsilon T_\delta  v, \psi \rangle_{H^{-1}(\R^3_+) \times H^1_0(\R^3_+)} 
 = \langle  v, T_{-\delta} \tilde{M}_\epsilon  \psi \rangle_{H^{-1}(\R^3_+) \times H^1_0(\R^3_+)}
\end{align}
for all $ v \in \Ltwoh$ and $\psi \in H^1_0(\R^3_+)$.
 As $T_{-\delta} \tilde{M}_\epsilon$ maps $H^1_0(\R^3_+)$ continuously into itself, the mapping
 $M_\epsilon T_\delta$ continuously extends to 
 an operator on $H^{-1}(\R^3_+)$ via formula~\eqref{DefinitionConvolutionOnHMinusOneDomain}.
  We deduce the identity
\begin{align*}
 \partial_j M_\epsilon T_\delta v = M_\epsilon \partial_j T_\delta v =  M_\epsilon T_\delta  \partial_j v
\end{align*}
by duality for all $j \in \{1,2,3\}$ and $v \in \Ltwoh$. Finally, for $A \in W^{1,\infty}(\R^3_+)$ and $v \in \Hh{-1}$ 
we obtain  $(T_\delta A) T_\delta v =  T_\delta (A v)$ in $\Hh{-1}$.

\smallskip

II) Let $0 < \epsilon < \delta$. We abbreviate the differential operators 
$\mathcal{L}(T_\delta \mathcal{A}_j, T_\delta \mathcal{D})$ by $\mathcal{L}_\delta$ and 
$\Div(T_\delta \mathcal{A}_1, T_\delta \mathcal{A}_2, T_\delta \mathcal{A}_3)$ by 
$\Div_\delta$. (Recall~\eqref{EquationDefinitionOfGeneralizedDiv}.) Let $\alpha \in \N_0^4$ with $|\alpha| = m$, 
$\alpha_0 = 0$, and $\alpha_3 = k$. We set $\alpha' = \alpha - e_3 \in \N_0^4$. The derivative
$\partial^{\alpha'} u$ belongs to $\G{0}$ by assumption. 
Because of the mollifier, the map $M_\epsilon T_\delta \partial^{\alpha'} u$ is contained in 
$C^1(\clJ, \Hh{2}) \hookrightarrow \G{1}$, $M_\epsilon T_\delta \partial^{\alpha'} u_0$ in $\Hh{1}$,  
$\mathcal{L}_\delta M_\epsilon T_\delta \partial^{\alpha'} u$  in $\G{0}$, 
and $\Div_\delta \mathcal{L}_\delta M_\epsilon T_\delta \partial^{\alpha'} u$ in $\Ltwoa$. 
To show convergence of $\partial_3 M_\epsilon T_\delta \partial^{\alpha'} u$ as $\epsilon \rightarrow 0$, we
want to apply the a priori estimate~\eqref{EquationFirstOrderFinal}.
Therefore, we have to study the convergence properties of the  functions $\mathcal{L}_\delta M_\epsilon T_\delta \partial^{\alpha'} u$ 
and $\Div_\delta \mathcal{L}_\delta M_\epsilon T_\delta \partial^{\alpha'} u$ as $\epsilon \rightarrow 0$. 
We focus on the latter as this is the more difficult one.

We use  the maps $\mu_{kl}$, $\hat \mu$, and $\tilde\mu$ from  \eqref{EquationDefinitionOfTildeMu}.
Exploiting step~I), we compute
 \begin{align}
  \label{EquationPreparationToComputeDivOfLMepsu}
  &(T_\delta \tilde{\mu})^T \nabla \mathcal{L}_\delta M_\epsilon T_\delta \partial^{\alpha'} u \\
  &= \sum_{j = 0}^2 (T_\delta \tilde{\mu})^T (T_\delta \nabla \mathcal{A}_j) \partial_j M_{\epsilon} T_\delta \partial^{\alpha'} u + (T_\delta \tilde{\mu})^T (T_\delta \nabla \mathcal{D}) M_\epsilon T_\delta \partial^{\alpha'} u   \nonumber\\
  &\quad \! + T_\delta (\tilde{\mu}^T \mathcal{A}_0) \nabla M_\epsilon T_\delta \partial_t \partial^{\alpha'} u + T_\delta (\tilde{\mu}^T \mathcal{D}) \nabla M_\epsilon T_\delta \partial^{\alpha'} u  + \! \sum_{j = 1}^3 T_\delta(\tilde{\mu}^T \mathcal{A}_j) \nabla \partial_j M_{\epsilon} T_\delta \partial^{\alpha'} u \nonumber \\
  &=: \Lambda^{\delta,\epsilon} + \sum_{j = 1}^3 T_\delta(\tilde{\mu}^T \mathcal{A}_j) \nabla \partial_j M_{\epsilon} T_\delta \partial^{\alpha'} u. \nonumber
 \end{align}
 The cancellation properties of $\mathcal{L}_\delta$ established in formulas~\eqref{EquationTraceSmaller3EqualsZero}, \eqref{EquationTraceLarger3EqualsZero}, 
 \eqref{EquationTraceSmaller9EqualsZero} and \eqref{EquationTraceLarger9EqualsZero} show that
 \begin{align*}
 	&\sum_{k = 1}^3 \sum_{j = 1}^3 (T_\delta(\tilde{\mu}^T \mathcal{A}_j) \nabla \partial_j M_{\epsilon} T_\delta \partial^{\alpha'} u)_{(k+3l)k} = 0
 \end{align*}
 for all $l \in \{0,1,2,3\}$.
 Equation~\eqref{EquationPreparationToComputeDivOfLMepsu} thus leads to
 \begin{equation}
 	\label{EquationEstimateForDivFirstResult}
 	\Div_\delta  \mathcal{L}_\delta M_\epsilon T_\delta \partial^{\alpha'} u = \sum_{k = 1}^3 \Big( \Lambda_{kk}^{\delta,\epsilon}, \Lambda_{(k+3)k}^{\delta,\epsilon}, \Lambda_{(k+6)k}^{\delta,\epsilon}, \Lambda_{(k+9)k}^{\delta,\epsilon} \Big).
 \end{equation}
 We rewrite $\Lambda^{\delta, \epsilon}$ in the form
 \begin{align*}
 	\Lambda^{\delta, \epsilon} &= \sum_{j = 0}^2 [T_\delta (\tilde{\mu}^T \nabla \mathcal{A}_j), M_\epsilon] \partial_j T_\delta \partial^{\alpha'} u  
 	  + [T_\delta(\tilde{\mu}^T \nabla \mathcal{D}), M_\epsilon] T_\delta \partial^{\alpha'} u \\
 	&\quad  + [T_\delta (\tilde{\mu}^T \mathcal{A}_0), M_\epsilon] \nabla  T_\delta \partial_t \partial^{\alpha'} u  
 	   + [T_\delta (\tilde{\mu}^T \mathcal{D}), M_\epsilon] \nabla T_\delta \partial^{\alpha'} u\nonumber \\
 	&\quad + M_\epsilon T_\delta \Big( \sum_{j=0}^2 \tilde{\mu}^T \nabla \mathcal{A}_j \partial_j \partial^{\alpha'} u 
 	  + \tilde{\mu}^T \nabla \mathcal{D} \partial^{\alpha'} u + \tilde{\mu}^T \mathcal{A}_0 \nabla \partial_t \partial^{\alpha'} u 
 	   + \tilde{\mu}^T \mathcal{D} \nabla \partial^{\alpha'} u \Big).
 \end{align*}
 In view  of the terms with $m$ space derivatives in the last line, we introduce the map
 \begin{align*}
  \tilde{f}_{\alpha'} &= \sum_{0 < \beta \leq \alpha'} \binom{\alpha'}{\beta} \partial^{\beta}  (\tilde{\mu}^T \mathcal{A}_0) \nabla \partial^{\alpha' - \beta} \partial_t u
			  + \sum_{0 < \beta \leq \alpha'} \binom{\alpha'}{\beta} \partial^{\beta}  (\tilde{\mu}^T \mathcal{D}) \nabla \partial^{\alpha' - \beta} u \\
		      &\quad + \sum_{j = 0}^2 \sum_{0 < \beta \leq \alpha'} \binom{\alpha'}{\beta} \partial^\beta (\tilde{\mu}^T \nabla \mathcal{A}_j) \partial^{\alpha' - \beta} \partial_j u  + \sum_{0 < \beta \leq \alpha'} \binom{\alpha'}{\beta} \partial^\beta (\tilde{\mu}^T \nabla \mathcal{D}) \partial^{\alpha' - \beta}  u.
 \end{align*}
 As $u$ and $\partial_t u$ are contained in $C(\clJ, \Hh{m-1})$, Lemma~\ref{LemmaRegularityForA0} implies that 
$\tilde{f}_{\alpha'}$ is an element of $\Ltwoa$.  It follows 
 	\begin{align*}
 	\Lambda^{\delta, \epsilon}
 	&= \sum_{j = 0}^2 [T_\delta (\tilde{\mu}^T \nabla \mathcal{A}_j), M_\epsilon] \partial_j T_\delta \partial^{\alpha'} u  + [T_\delta(\tilde{\mu}^T \nabla \mathcal{D}), M_\epsilon] T_\delta \partial^{\alpha'} u \\
 	&\quad + [T_\delta (\tilde{\mu}^T \mathcal{A}_0), M_\epsilon] \nabla  T_\delta \partial_t \partial^{\alpha'} u + [T_\delta (\tilde{\mu}^T \mathcal{D}), M_\epsilon] \nabla T_\delta \partial^{\alpha'} u + \partial^{\alpha'} M_\epsilon T_\delta (\tilde{\mu}^T \nabla f)  \nonumber \\
 	&\quad  - M_\epsilon T_\delta \tilde{f}_{\alpha'} -\sum_{j = 1}^3 \partial^{\alpha'} M_\epsilon T_\delta (\tilde{\mu}^T \mathcal{A}_j \nabla \partial_j  u) \\
 	&=: \tilde{\Lambda}^{\delta, \epsilon}  -\sum_{j = 1}^3 \partial^{\alpha'} M_\epsilon T_\delta (\tilde{\mu}^T \mathcal{A}_j \nabla \partial_j  u).
 \end{align*}
 Equations~\eqref{EquationTraceSmaller3EqualsZero}, \eqref{EquationTraceLarger3EqualsZero}, 
 \eqref{EquationTraceSmaller9EqualsZero} and \eqref{EquationTraceLarger9EqualsZero}
 also yield that
 \begin{align*}
 	 \sum_{k = 1}^3  \Big(\Lambda_{kk}^{\delta, \epsilon}, \Lambda_{(k+3)k}^{\delta, \epsilon}, \Lambda_{(k+6)k}^{\delta, \epsilon}, \Lambda_{(k+9)k}^{\delta, \epsilon} \Big) 
 	= \sum_{k = 1}^3 \Big(\tilde{\Lambda}_{kk}^{\delta, \epsilon}, \tilde{\Lambda}_{(k+3)k}^{\delta, \epsilon}, \tilde{\Lambda}_{(k+6)k}^{\delta, \epsilon}, \tilde{\Lambda}_{(k+9)k}^{\delta, \epsilon} \Big).
 \end{align*}
 By means of~\eqref{EquationEstimateForDivFirstResult}, we arrive at the core identity
 \begin{align}
  \label{EquationDivOfLMepsEdeltau}
  &\Div_\delta \mathcal{L}_\delta M_\epsilon T_\delta \partial^{\alpha'} u 
   = \sum_{k = 1}^3 \Big(\tilde{\Lambda}_{kk}^{\delta, \epsilon}, \tilde{\Lambda}_{(k+3)k}^{\delta, \epsilon}, 
      \tilde{\Lambda}_{(k+6)k}^{\delta, \epsilon}, \tilde{\Lambda}_{(k+9)k}^{\delta, \epsilon} \Big).
 \end{align}
Starting from its counterpart (4.7) in  \cite{SpitzMaxwellLinear}, the rest of the reasoning is now the same as in the proof
of Lemma~4.1 in this paper. One uses that $M_\epsilon T_\delta \partial^{\alpha'} u$
 solves the initial value problem~\eqref{IVP} with differential operator $\mathcal{L}_\delta$,
 inhomogeneity $\mathcal{L}_\delta M_\epsilon T_\delta \partial^{\alpha'} u$ and initial value $M_\epsilon T_\delta u_0$.
 In these data and in \eqref{EquationDivOfLMepsEdeltau}, one 
  can pass to the limit in $L^2$ as $\epsilon \rightarrow 0$ employing estimates for the commutators 
  of the mollifier and the coefficients. The estimate \eqref{EquationFirstOrderFinal} from 
  Proposition~\ref{PropositionCentralEstimateInNormalDirection} then allows to bound $\nabla T_\delta \partial^{\alpha'} u$ in $G_0(\Omega)$,
  uniformly in $\delta>0$, see (4.15) in  \cite{SpitzMaxwellLinear}. One can then let $\delta\to0$ obtaining the result. We omit the details.
\end{proof}


Replacing estimate~\eqref{EquationFirstOrderFinal} from Proposition~\ref{PropositionCentralEstimateInNormalDirection} 
 by inequality~\eqref{EquationFirstOrderL2} in the above proof, one derives the following variant of Lemma~\ref{LemmaRegularityInNormalDirection}, 
 cf.\ Corollary 4.2 in \cite{SpitzMaxwellLinear}.

\begin{cor}
 \label{CorollaryRegularityInNormalDirectionL2InTime}
Let $\eta > 0$, $m \in \N$, and $\tilde{m} =  \max\{m,3\}$. Take coefficients $\mathcal{A}_0 \in \Fupdwl{\tilde{m}}{cp}{\eta}$, 
$\mathcal{A}_1, \mathcal{A}_2 \in \Fcoeff{\tilde{m}}{\operatorname{cp}}$, $\mathcal{A}_3 = \tilde{\mathcal{A}}_3^{\operatorname{co}}$,
and $\mathcal{D} \in \Fuwl{\tilde{m}}{cp}$. Choose data $f \in \Ha{m}$ and $u_0 \in \Hh{m}$. 
Let $u$ be a solution of the initial value problem~\eqref{IVP} with these coefficients and data.
Assume that $u$ belongs to $\bigcap_{j = 1}^m C^j(\clJ, \Hh{m-j})$.

Take $k \in \{1, \ldots, m\}$ and a multi-index $\alpha \in \N_0^4$ with $|\alpha| = m$, $\alpha_0 = 0$,
and $\alpha_3 = k$. 
Suppose that $\partial^\beta u$ is contained in $\Ltwoa$ for all $\beta \in \N_0^4$ with $|\beta| = m$ and 
$\beta_3 \leq k-1$.
Then $\partial^\alpha u$ is an element of $\Ltwoa$.
\end{cor}


Based on Lemma~\ref{LemmaRegularityInNormalDirection} and Corollary~\ref{CorollaryRegularityInNormalDirectionL2InTime},
the regularization arguments in tangential and time direction are analogous to the proofs of Lemma~4.4 and 4.5
in \cite{SpitzMaxwellLinear}. One first studies the solution $u$ mollified in $(x_1,x_2)$. The regularized solution $u_\epsilon$ satisfies
the Maxwell system with modified data (as in (4.20) of \cite{SpitzMaxwellLinear}). It then crucially enters into the bound 
of $u$ in a family of weighted tangential Sobolev norms, taken from Section~1.7 and Section~2.4 in~\cite{Hoermander}.
The a priori estimate from Lemma~\ref{LemmaEllerResult} allows us to control $u_\epsilon$ in $G_0$.
It is then possible to take the limit $\epsilon \to 0$. The results from \cite{Hoermander} require smooth coefficients
so that temporarily  we have to assume this extra regularity.

In the time direction one looks at the problem solved by the time derivative $v$ 
of $u$, cf.\ (4.32) in \cite{SpitzMaxwellLinear}. Integration with respect to time yields a function which 
coincides with $u$, implying the required time regularity.
Here the compatibility conditions are needed.
In these arguments the new features of the problem \eqref{IBVP} do not play a role and one can follow the lines of the
proofs of \cite{SpitzMaxwellLinear}. We thus only state the results.

\begin{lem}
 \label{LemmaRegularityInSpaceTangential}
 Let $\eta > 0$, $m \in \N$, and $\tilde{m} = \max\{m,3\}$.
 Take coefficients $\mathcal{A}_0 \in \Fupdwl{\tilde{m}}{cp}{\eta}$, $\mathcal{A}_1, \mathcal{A}_2 \in \Fcoeff{\tilde{m}}{cp}$, 
 $\mathcal{A}_3 = \tilde{\mathcal{A}}_3^{\operatorname{co}}$,
 $\mathcal{D} \in \Fuwl{\tilde{m}}{cp}$ and $B= \mathcal{B}^{\operatorname{co}}$. We further assume that these 
 coefficients belong  to $C^\infty(\overline{\Omega})$.  Let $u$ be the weak solution of~\eqref{IBVP} with data 
 $f \in \Hata{m}$,  $g \in \E{m}$,  and $u_0 \in \Hhta{m}$. 
 Suppose that $u$ belongs to $\bigcap_{j = 1}^m C^j(\clJ, \Hh{m-j})$. Pick a multi-index $\alpha \in \N_0^4$ with $|\alpha| = m$ 
 and $\alpha_0 = \alpha_3 = 0$. Then $\partial^\alpha u$ is an element of $C(\clJ, \Ltwoh)$.
\end{lem}

\begin{lem}
 \label{LemmaBasicRegularityInTime}
 Let $\eta > 0$. Take coefficients $\mathcal{A}_0 \in \Fupdwl{3}{cp}{\eta}$, $\mathcal{A}_1, \mathcal{A}_2 \in \Fcoeff{3}{cp}$, 
 $\mathcal{A}_3 = \tilde{\mathcal{A}}_3^{\operatorname{co}}$,
 $\mathcal{D} \in \Fuwl{3}{cp}$, and $B = \mathcal{B}^{\operatorname{co}}$. 
 Choose data $u_{0} \in \Hh{1}$, $g \in  \E{1}$,
 and $f \in \Ha{1}$. Assume that the tuple $(0, \mathcal{A}_0,\ldots, \mathcal{A}_3, \mathcal{D}, B, f, g,  u_0)$ 
 fulfills the compatibility conditions~\eqref{EquationCompatibilityConditionPrecised} on $G=\R^3_+$ of order $1$. 
 Let $u \in C(\clJ, \Ltwoh)$ be the weak solution of~\eqref{IBVP} with data
 $f$, $g$, and $u_{0}$. Assume that $u \in C^1(\overline{J'}, \Ltwoh)$ implies 
 $u \in G_1(J' \times \R^3_+)$ for every open interval $J' \subseteq J$. 
 Then $u$ belongs to $\G{1}$.
\end{lem}


To iterate the previous result, we need a relation between the operators $S_{m,p}$ of different order
stated in the next lemma. It follows from  a straightforward computation  based on definition
\eqref{EquationDefinitionSmp} of $S_{m,p}$ as in Lemma~4.8 of~\cite{SpitzDissertation}.
\begin{lem}
 \label{LemmaHigherOrderCompatibilityConditions}
 Let  $\eta > 0$, $m \in \N$ and $\tilde{m} = \max\{m,3\}$. Take coefficients $\mathcal{A}_0 \in \Fupdwl{\max\{m+1,3\}}{cp}{\eta}$ with 
 $\partial_t \mathcal{A}_0 \in \Fuwl{\tilde{m}}{cp}$, $\mathcal{A}_1$, $\mathcal{A}_2 \in \Fcoeff{\max\{m+1,3\}}{cp}$, 
 $\mathcal{A}_3 = \tilde{\mathcal{A}}_3^{\operatorname{co}}$, $\mathcal{D} \in \Fuwl{\max\{m+1,3\}}{cp}$,
 and $B = \mathcal{B}^{\operatorname{co}}$. Choose data $t_0 \in \clJ$, 
 $u_0 \in \Hh{m+1}$, $g \in \E{m+1}$, and $f \in \Ha{m+1}$. Assume that $u \in \G{m}$ solves~\eqref{IBVP}
 with initial time $t_0$.
 Set $u_1 = S_{m+1,1}(t_0, \mathcal{A}_0, \ldots,\mathcal{A}_3, \mathcal{D}, f, u_0)$ and $f_1 = \partial_t f - \partial_t \mathcal{D} u$.
 Let $p \in \{0, \ldots, m-1\}$.
 We then obtain
 \begin{align*}
  S_{m, p}(t_0,\mathcal{A}_0, \ldots, \mathcal{A}_3, \partial_t \mathcal{A}_0 + \mathcal{D}, f_1, u_1)
   = S_{m+1, p+1}(t_0,\mathcal{A}_0, \ldots, \mathcal{A}_3, \mathcal{D}, f,u_0).
 \end{align*}
\end{lem}


Combining the above results with an iteration argument, we derive the desired 
regularity of the solution $u$ provided the coefficients are smooth.
\begin{prop}
 \label{PropositionRegularityForApproximatingProblem}
 Let $\eta > 0$, $m \in \N$, and $\tilde{m} = \max\{m,3\}$. 
 Take coefficients 
 $\mathcal{A}_0 \in \Fupdwl{\tilde{m}}{cp}{\eta}$ with $\partial_t \mathcal{A}_0 \in \Fuwl{\max\{m-1,3\}}{cp}$ , $\mathcal{A}_1, \mathcal{A}_2 \in \Fcoeff{\tilde{m}}{cp}$, 
 $\mathcal{A}_3 = \tilde{\mathcal{A}}_3^{\operatorname{co}}$,
  $\mathcal{D} \in \Fuwl{\tilde{m}}{cp}$, and $B = \mathcal{B}^{\operatorname{co}}$. 
  Assume that these coefficients are  contained in $C^\infty(\overline{\Omega})$.
 Choose data $f \in \Ha{m}$, $g \in \E{m}$, and $u_0 \in \Hh{m}$ such that the tuple 
 $(0, \mathcal{A}_0, \ldots, \mathcal{A}_3, \mathcal{D}, B, f, g, u_0)$ 
 satisfies the compatibility conditions~\eqref{EquationCompatibilityConditionPrecised} on $G=\R^3_+$
 of order $m$.  Let $u$ be the weak solution of~\eqref{IBVP}  Then $u$ belongs to $\G{m}$.
\end{prop}
\begin{proof}
Lemma~\ref{LemmaBasicRegularityInTime}, Lemma~\ref{LemmaRegularityInSpaceTangential}, 
 and Lemma~\ref{LemmaRegularityInNormalDirection} show the assertion  for $m = 1$.
Let the claim be true for some $m \in \N$ and let the assumptions be fulfilled for $m+1$.  
The weak solution $u$ of~\eqref{IBVP} hence belongs to $\G{m}$, and $\partial_t u$ satisfies
 \begin{equation*}
 \left\{\begin{aligned}
    \mathcal{L}_{\partial_t} v   &= \partial_{t} f - \partial_t \mathcal{D} u=:f_1, \quad &&x \in \R^3_+, \quad &&t \in J; \\
    B v &= \partial_t g, \quad &&x \in \partial \R^3_+, &&t \in J; \\
    v(0) &= S_{m+1,1}(0, \mathcal{A}_0, \ldots, \mathcal{A}_3, \mathcal{D}, f, u_0)=:u_1,  &&x \in \R^3_+,
  \end{aligned} \right.
\end{equation*}
 where we write $\mathcal{L}_{\partial_t}$ for $\mathcal{L}(\mathcal{A}_0, \ldots, \mathcal{A}_3, \partial_t \mathcal{A}_0 + \mathcal{D})$.
 The initial field $u_1$ belongs to $\Hh{m}$ by Lemma~\ref{LemmaEstimatesForHigherOrderInitialValues},
 the inhomogeneity $f_1$ to $\Ha{m}$ by  Lemma~\ref{LemmaRegularityForA0}, and 
 $\partial_t g$ to $\E{m}$. The coefficients satisfy  the conditions of Lemma~\ref{LemmaHigherOrderCompatibilityConditions} 
 and $\partial_t \mathcal{A}_0 + \mathcal{D}$ is an element of $\Fuwl{\tilde{m}}{cp} \cap C^\infty(\overline{\Omega})$.
Lemma~\ref{LemmaHigherOrderCompatibilityConditions} thus shows 
  the compatibility conditions~\eqref{EquationCompatibilityConditionPrecised} of order $m$ for the tuple 
 $(0, \mathcal{A}_0, \ldots, \mathcal{A}_3, \partial_t \mathcal{A}_0 + \mathcal{D}, f_1, \partial_t g, u_1)$.
 By the induction hypothesis, the function  $\partial_t u$ is contained in $\G{m}$, so that $u$ belongs to
 $\bigcap_{j=1}^{m+1} C^j(\clJ, \Hh{m+1-j})$. 
 Lemma~\ref{LemmaRegularityInSpaceTangential} and Lemma~\ref{LemmaRegularityInNormalDirection} then imply that
 the solution $u$ is an element of $\G{m+1}$.
\end{proof}

It remains to remove the extra regularity assumptions. Lemma~\ref{LemmaApproximationOfCoefficients}
provides suitable approximations of the given coefficients. However, after this procedure the compatibility 
conditions can be violated. To overcome this difficulty, we modify the initial fields appropriately in 
Lemma~\ref{LemmaExistenceOfApproximatingSequence}. The proof of this result is based on the next  fact which 
again relies on the algebraic structure of the coefficient matrices.

\begin{lem}
 \label{LemmaConstructionForCC} 
 Let $\eta > 0$, $p\in\N_0$, and  $m,k \in \N$ with $m \geq 3$ and $k \leq m-1$. Take $\mathcal{A}_0 \in \Fpdk{m}{\eta}{12}$ and 
 $\mathcal{A}_3 = \tilde{\mathcal{A}}_3^{\operatorname{co}}$.
 Choose $r > 0$ such that $\Fvarnorm{m-1}{\mathcal{A}_0(0)}\leq r$.
 Take an approximating family $\{\mathcal{A}_{0,\epsilon}\}_{\epsilon> 0}$  provided 
 by Lemma~\ref{LemmaApproximationOfCoefficients}.
 Let $v_{0,\epsilon}$  be maps in $\Hh{k}^{12}$ for $\epsilon > 0$. Then there exists a number 
 $\epsilon_0 >0$, a constant $C = C(\eta,r)$, and a family of functions $\{v_{p,\epsilon}\}_{0 < \epsilon < \epsilon_0}$ in $\Hh{k}^{12}$ 
 such that
  \begin{align*}
  \mathcal{A}_3 (\mathcal{A}_{0,\epsilon}(0)^{-1} \mathcal{A}_{3})^p v_{p,\epsilon} = \mathcal{A}_3 v_{0,\epsilon}\qquad \text{and}\qquad
    \Hhn{k}{v_{p,\epsilon}} \leq C \Hhn{k}{v_{0,\epsilon}}
 \end{align*}
 for all $\epsilon \in (0, \epsilon_0)$.
\end{lem}
\begin{proof}
 I) By Lemma~\ref{LemmaApproximationOfCoefficients} there is a number $\epsilon_0 > 0$ such that
  \begin{align}
  \label{EquationBoundsForApproximatingSequenceCoefficients}
   \Fvarnorm{m-1}{\mathcal{A}_{0,\epsilon}(0)} \leq 2 r
  \end{align}
  for all $\epsilon \in (0, \epsilon_0)$. Let $\epsilon \in (0,\epsilon_0)$. We introduce the invertible matrices
  \begin{equation*}
   Q = \begin{pmatrix}
          0 &0 &0 &0 &1 &0 \\
          0 &0 &0 &-1 &0 &0 \\
          0 &0 &1 &0 &0  &0 \\
          0 &-1 &0 &0 &0  &0 \\
          1 &0 &0 &0 &0  &0 \\
          0 &0 &0 &0 &0  &1
         \end{pmatrix}
         \qquad \text{and} \qquad
    \mathcal{Q} = \begin{pmatrix}
                   Q &0 \\
                   0 &-Q
                  \end{pmatrix}
    \end{equation*}
    and note that
    \begin{equation*}
    \mathcal{A}_3 \mathcal{Q} = \tilde{\mathcal{A}}_3^{\operatorname{co}}\mathcal{Q} = \begin{pmatrix}
                                 J_{\operatorname{bl}} &0 &0 &0 \\
                                 0 &J_{\operatorname{bl}} &0 &0 \\
                                 0 &0 &J_{\operatorname{bl}} &0 \\
                                 0 &0 &0 &J_{\operatorname{bl}}
                                \end{pmatrix}, \quad \text{where} \quad
    J_{\operatorname{bl}} = \begin{pmatrix}
			      1 &0 &0  \\
			      0 &1 &0  \\
			      0 &0 &0  \\
                            \end{pmatrix}.
  \end{equation*}
  Since $\mathcal{A}_{0,\epsilon}\ge \eta$, also the matrix
  \begin{align*}
   \Theta_\epsilon = \begin{pmatrix}
       \mathcal{A}_{0,\epsilon; 3,3}  &\mathcal{A}_{0,\epsilon;3,6} &\mathcal{A}_{0,\epsilon; 3,9} &\mathcal{A}_{0,\epsilon; 3,12} \\
       \mathcal{A}_{0,\epsilon;6,3}   &\mathcal{A}_{0,\epsilon;6,6} &\mathcal{A}_{0,\epsilon; 6,9} &\mathcal{A}_{0,\epsilon; 6,12} \\
       \mathcal{A}_{0,\epsilon; 9,3}  &\mathcal{A}_{0,\epsilon; 9,6} &\mathcal{A}_{0,\epsilon; 9,9} &\mathcal{A}_{0,\epsilon; 9,12} \\
       \mathcal{A}_{0,\epsilon; 12,3} &\mathcal{A}_{0,\epsilon; 12,6} &\mathcal{A}_{0,\epsilon; 12,9} &\mathcal{A}_{0,\epsilon; 12,12}
       \end{pmatrix},
  \end{align*}
 satisfies $\Theta_\epsilon \geq \eta$ on $\Omega$.
  In particular, $\Theta_\epsilon$ has an inverse with
  \begin{equation}
   \label{EquationBoundForThetaInverse}
   \Fvarnorm{m-1}{\Theta_\epsilon^{-1}(0)} \leq C(\eta, r)\qquad \text{for all \ } \epsilon \in (0, \epsilon_0).
  \end{equation}

  II) Let $w_0 \in \Hh{k}^{12}$. 
  We can define scalar functions $h_{1,\epsilon}, \ldots, h_{4,\epsilon}$ by
  \begin{equation*}
  	(h_{1,\epsilon}, \ldots, h_{4,\epsilon}) = -\Theta_\epsilon^{-1}(0) (\mathcal{A}_{0,\epsilon}(0) w_0)_{(3,6,9,12)},
  \end{equation*} 
  where we write $\zeta_{(3,6,9,12)}=(\zeta_3, \zeta_6,\zeta_9,\zeta_{12})$ for any vector $\zeta\in \R^{12}$.
  Lemma~\ref{LemmaRegularityForA0} and the inequalities  \eqref{EquationBoundsForApproximatingSequenceCoefficients}
  and~\eqref{EquationBoundForThetaInverse} imply that
  \begin{equation}
  	\label{EquationBoundForh1epsh2eps}
  	\Hhn{k}{(h_{1,\epsilon}, \ldots, h_{4,\epsilon})} \leq C(\eta,r) \Hhn{k}{w_0}.
  \end{equation}
We next set
  \begin{align}
  	\label{EquationDefinitionOfw1eps}
  	\hat{w}_{\epsilon} = \mathcal{Q} \tilde{w}_{\epsilon}, \qquad
  	\tilde{w}_{\epsilon} = -\mathcal{A}_{0,\epsilon}(0)  \Big( w_0 + h_{1,\epsilon} e_3 + h_{2,\epsilon} e_6 + h_{3,\epsilon} e_9 + h_{4,\epsilon} e_{12} \Big).
  \end{align}
  Lemma~\ref{LemmaRegularityForA0}, \eqref{EquationBoundsForApproximatingSequenceCoefficients}, 
   and~\eqref{EquationBoundForh1epsh2eps} again provide a constant $C(\eta,r)$ such that
  \begin{equation}\label{EquationEstimateForConstructedFunctionInSummary}
  	\Hhn{k}{\hat{w}_{\epsilon}} \leq C(\eta,r) \Hhn{k}{w_0}.
  \end{equation}
  Observe that
   \begin{equation*}
     	(\tilde{w}_{\epsilon})_{(3,6,9,12)} = (-\mathcal{A}_{0,\epsilon}(0)   w_0)_{(3,6,9,12)} 
   		- \Theta_\epsilon(0)(h_{1,\epsilon},\ldots,h_{4,\epsilon}) = 0,
   \end{equation*}
   and hence $\mathcal{A}_3 \mathcal{Q} \tilde{w}_{\epsilon} = \tilde{w}_{\epsilon}$. We thus compute
	\begin{align}\label{EquationConstructionOfFunctionSummary}
		\mathcal{A}_3(-\mathcal{A}_{0,\epsilon}(0)^{-1} \mathcal{A}_3) \hat{w}_{\epsilon} 
		  &=  \mathcal{A}_3 (-\mathcal{A}_{0,\epsilon}(0)^{-1}) \tilde{w}_{\epsilon} = \mathcal{A}_3 w_0
	\end{align}
	using \eqref{EquationDefinitionOfw1eps} and $\operatorname{ker}\mathcal{A}_3= \operatorname{span}\{e_3, e_6, e_9,e_{12}\}$.

	III) To show the assertion of the lemma, we proceed inductively. We claim that for all $p \in \N_0$, $\epsilon \in (0, \epsilon_0)$, 
	and $w \in \Hh{k}^{12}$ there is a function $w_{p,\epsilon}(w)$ in $\Hh{k}^{12}$ 
	and a constant $C_p = C_p(\eta,r)$
	such that
	\begin{align}
		\label{EquationInductionClaimForConstructionvdeltaepsilon}
		&\mathcal{A}_{3} (-\mathcal{A}_{0,\epsilon}(0)^{-1} \mathcal{A}_3)^{p} w_{p,\epsilon}(w) = \mathcal{A}_{3} w, \\
		&\Hhn{k}{w_{p,\epsilon}(w)} \leq C_p \Hhn{k}{w}.\label{EquationInductionClaimForEstimatevdeltaepsilon}
	\end{align}
	We can simply set $w_{0,\epsilon}(w) =w$. Let the claim be true for a number $p \in \N_0$. 
	Fix $\epsilon \in (0, \epsilon_0)$ and $w \in \Hh{k}^{12}$. Step II) applied with $w_0 = w$ yields a function 
	$\tilde{w}_{p,  \epsilon} \in \Hh{k}^{12}$ satisfying
	\begin{align}
		 \label{EquationConstructionFirstStep}
		 \mathcal{A}_{3} (-\mathcal{A}_{0,\epsilon}(0)^{-1} \mathcal{A}_3) \tilde{w}_{p, \epsilon} = \mathcal{A}_{3} w
		\hspace{0.5em} \text{ and } \hspace{0.5em} \Hhn{k}{ \tilde{w}_{p,  \epsilon}} \leq C(\eta,r) \Hhn{k}{w}.
	\end{align}
	We now define $w_{p+1, \epsilon}(w) = w_{p,\epsilon}(\tilde{w}_{p,  \epsilon})$ for each 
	$\epsilon \in (0, \epsilon_0)$. 
	The map $w_{p+1,  \epsilon}(w)$ then is contained in $\Hh{k}^{12}$, and we compute
	\begin{align*}
		\mathcal{A}_{3} (-\mathcal{A}_{0,\epsilon}(0)^{-1} \mathcal{A}_3)^{p+1} w_{p+1,  \epsilon}(w) &= \mathcal{A}_{3}(-\mathcal{A}_{0,\epsilon}(0)^{-1}) \mathcal{A}_3 (-\mathcal{A}_{0,\epsilon}(0)^{-1} \mathcal{A}_3)^{p} w_{p, \epsilon}(\tilde{w}_{p,  \epsilon}) \\
		&= \mathcal{A}_{3}(-\mathcal{A}_{0,\epsilon}(0)^{-1}) \mathcal{A}_3 \tilde{w}_{p,  \epsilon} 
		= \mathcal{A}_{3} w,
	\end{align*}
	where we employed the induction hypothesis~\eqref{EquationInductionClaimForConstructionvdeltaepsilon} and~\eqref{EquationConstructionFirstStep}.
	Combining~\eqref{EquationInductionClaimForEstimatevdeltaepsilon} with~\eqref{EquationEstimateForConstructedFunctionInSummary}, we further obtain
	\begin{align*}
		\Hhn{k}{w_{p+1,\epsilon}(w)} = \Hhn{k}{w_{p,\epsilon}(\tilde{w}_{p,  \epsilon})} 
		\leq C_p \Hhn{k}{\tilde{w}_{p,  \epsilon}} \leq C \Hhn{k}{w},
	\end{align*}
	where $C = C(\eta,r)$. The claim now follows by induction.
	
	We obtain the assertion of the lemma by setting $v_{p,\epsilon} = w_{p, \epsilon}(v_{0, \epsilon})$.
\end{proof}

\begin{lem}
 \label{LemmaExistenceOfApproximatingSequence}
 Let $\eta > 0$, $m \in \N$, and $\tilde{m} = \max\{m,3\}$. Take coefficients $\mathcal{A}_0 \in \Fupdwl{\tilde{m}}{cp}{\eta}$,
 $\mathcal{A}_1, \mathcal{A}_2 \in \Fcoeff{\tilde{m}}{cp}$, $\mathcal{A}_3 = \tilde{\mathcal{A}}_3^{\operatorname{co}}$, 
 $\mathcal{D} \in \Fuwl{\tilde{m}}{cp}$, and $B = \mathcal{B}^{\operatorname{co}}$. Choose  data $f \in \Ha{m}$, $g \in \E{m}$,
 and $u_0 \in \Hh{m}$ which fulfill the
 compatibility conditions~\eqref{EquationCompatibilityConditionPrecised} on $G=\R^3_+$ of order $m$ in $t_0 \in \clJ$. 
 Let $\{\mathcal{A}_{i,\epsilon}\}_{\epsilon > 0}$ and $\{\mathcal{D}_{\epsilon}\}_{\epsilon > 0}$ be the families of functions 
 provided by Lemma~\ref{LemmaApproximationOfCoefficients} for $\mathcal{A}_i$ and $\mathcal{D}$ respectively for $i \in \{0,1,2\}$. 
 Then there exists a number $\epsilon_0 > 0$ and a family   $\{u_{0,\epsilon}\}_{0 < \epsilon <\epsilon_0}$ in $\Hh{m}$ such that the 
 compatibility conditions for the tuple $(t_0, \mathcal{A}_{0,\epsilon}, \mathcal{A}_{1,\epsilon}, \mathcal{A}_{2,\epsilon}, \mathcal{A}_3, \mathcal{D}_\epsilon, B, f,g, u_{0,\epsilon})$ of order $m$ are satisfied
  and $u_{0,\epsilon}$ tends to $u_0$ in $\Hh{m}$ as $\epsilon \rightarrow 0$.
\end{lem}
\begin{proof}
 Without loss of generality we assume $t_0 = 0$. 
 We set $u_{0,\epsilon} = u_0 + h_\epsilon$ and look for functions $h_\epsilon \in \Hh{m}$ with $h_\epsilon \rightarrow 0$ in 
 $\Hh{m}$  such that the compatibility conditions are fulfilled. Since 
 $B = \mathcal{M} \mathcal{A}_3$ for a constant matrix $\mathcal{M} = \mathcal{M}^{\operatorname{co}}$ by \eqref{EquationDecompositionOfA3co}, it  
 suffices  to find $h_\epsilon$ with
 \begin{align*}
  \mathcal{A}_3 S_{m,p}(0, \mathcal{A}_{0,\epsilon}, \mathcal{A}_{1,\epsilon}, \mathcal{A}_{2,\epsilon}, \mathcal{A}_3, \mathcal{D}_\epsilon, f, u_0 + h_\epsilon)  = \mathcal{A}_3 S_{m,p}(0, \mathcal{A}_0,\ldots, \mathcal{A}_3, \mathcal{D}, f, u_0)
 \end{align*}
 for all $0 \leq p \leq m-1$ on $\partial \R^3_+$. Using Lemma~\ref{LemmaConstructionForCC} one can now repeat 
 steps~I) and II) of the proof of Lemma~4.8 of \cite{SpitzMaxwellLinear} in which the structure arising from the 
 interface problem does not play a role. We thus omit the details.
 \end{proof}

 We can now deduce the differentiability theorem by 
applying Proposition~\ref{PropositionRegularityForApproximatingProblem} to the solutions of 
the approximating initial boundary value problems with coefficients and data from 
Lemma~\ref{LemmaExistenceOfApproximatingSequence}. Compared to  \cite{SpitzMaxwellLinear},
again the specific structure of our problem  does not enter the reasoning,
and thus we do not give a proof  and refer to Theorem~4.10 of \cite{SpitzMaxwellLinear} for the details.

\begin{theorem}
  \label{TheoremRegularityOfSolution}  \label{TheoremAPrioriEstimatesAndRegularityOfSolutionForGeneralCoefficients}
 Let $\eta > 0$, $m \in \N$, and  $\tilde{m} = \max\{m,3\}$. Take coefficients
 $\mathcal{A}_0 \in \Fupdwl{\tilde{m}}{cp}{\eta}$, $\mathcal{A}_1, \mathcal{A}_2 \in \Fcoeff{\tilde{m}}{cp}$, $\mathcal{A}_3 = \tilde{\mathcal{A}}_3^{\operatorname{co}}$,
  $\mathcal{D} \in \Fuwl{\tilde{m}}{cp}$, and $B = \mathcal{B}^{\operatorname{co}}$. Choose data
 $f \in \Ha{m}$, $g \in \E{m}$, and $u_0 \in \Hh{m}$ such that the tuple $(0, \mathcal{A}_0, \ldots, \mathcal{A}_3, \mathcal{D}, B, f,g, u_0)$ 
 satisfies the compatibility conditions~\eqref{EquationCompatibilityConditionPrecised} on $G=\R_+^3$ of order~$m$.
 Then the weak solution $u$ of~\eqref{IBVP} belongs to $\G{m}$.
\end{theorem}

 \begin{rem}
  Recall that Theorem~\ref{TheoremExistenceAndUniquenessOnDomain} is valid for coefficients
  $A_0$ and $D$ which have merely a limit as $|(t,x)| \rightarrow \infty$.
  Also all intermediate results extend to such coefficients. In particular, Proposition~\ref{PropositionCentralEstimateInNormalDirection}, 
  Theorem~\ref{TheoremAPrioriEstimates}, and Theorem~\ref{TheoremRegularityOfSolution} are still
  true if $\mathcal{A}_0$ and $\mathcal{D}$ only have a limit as $|(t,x)| \rightarrow \infty$, cf.\ the proof of Theorem~4.13 in \cite{SpitzDissertation}.
 \end{rem}
 
 \section{Local existence and uniqueness of the nonlinear system}
 \label{SectionLocalExistence}
 
 In this section we prove existence and uniqueness of a solution of~\eqref{EquationNonlinearIBVP} by 
 a fixed point argument based on the a priori estimates and the regularity theory from 
 Sections~\ref{SectionAPrioriEstimates} and \ref{SectionRegularity} for the corresponding linear problem. 
 We define a solution of~\eqref{EquationNonlinearIBVP} to be a function $u$ belonging to 
 $\bigcap_{j = 0}^m C^j(I, \mathcal{H}^{m-j}(G))$ with $\overline{\image u_{\pm}} \subseteq \mathcal{U}_{\pm}$ 
 for all $t \in I$ and satisfying~\eqref{EquationNonlinearIBVP}. Here $I$ is an interval with $t_0 \in I$.
 We further allow more general functions $\sigma$ than arising from the model~\eqref{EquationOhmsLaw}.
 The specific structure of the interface conditions does not enter very much in the proofs from now on.
 For this reason we can be more brief in this part of the paper and often refer the reader to the article
 \cite{SpitzQuasilinearMaxwell}, where the initial boundary value problem was treated in detail.
 We first introduce the spaces 
\begin{align}\label{def:ml}
 &\ml{m,n}{G}{\mathcal{U}_\pm} \\
 &= \{\theta \colon (G_+ \times \mathcal{U}_+) \cup (G_{-} \times \mathcal{U}_{-}) \rightarrow \R^{n \times n} \text{ with } 
	 \theta_\pm \in C^m(G_\pm \times \mathcal{U}_\pm, \R^{n \times n}) \text{ and } \notag \\
 &\hspace{3em} \sup_{(x,y) \in G_\pm \times \mathcal{U}_{\pm,1}} |\partial^\alpha \theta(x,y)| < \infty 
 \text{ for all } \alpha \in \N_0^9 \text{ with }
 |\alpha| \leq m \text{ and }  \mathcal{U}_{\pm,1} \Subset \mathcal{U}_\pm \}, \notag \\
 &\mlpd{m,n}{G}{\mathcal{U}_\pm} = \{\theta \in \ml{m,n}{G}{\mathcal{U}_\pm} \colon \text{There exists } \eta > 0 \text{ with } 
 \theta  =  \theta^T  \geq  \eta \notag \\
 &\hspace{18em} \text{on }  G_\pm  \times  \mathcal{U}_\pm\}\notag 
\end{align}
for our nonlinearities.
Here $\theta_+$ and $\theta_{-}$ denote the restrictions of $\theta$ to $G_+ \times \mathcal{U}_+$ 
respectively $G_{-} \times \mathcal{U}_{-}$. Moreover, by writing $G_\pm \times \mathcal{U}_{\pm}$ we 
address the two sets $G_+ \times \mathcal{U}_+$ and $G_{-} \times \mathcal{U}_{-}$.
Actually, we only need the dimensions $n = 1$ or $n = 6$.

We often have to control compositions $\theta(v)$ in higher regularity in terms of $v$. In Lemma~2.1
and Corollary~2.2  of  
\cite{SpitzQuasilinearMaxwell} the necessary formulas and estimates have been provided 
for functions defined on a single domain. Our interface case can then be treated by applying these facts 
to the subsets $G_\pm$ separately. Since the proofs below are only sketched, we do not repeat the modified
versions of these rather lengthy  auxiliary  results. 

As in the linear case discussed in Section~\ref{SectionNotation}, regular solutions of \eqref{EquationNonlinearIBVP}
have to satisfy compatibility conditions. To express them, we first introduce  the operators that give 
the initial values of the time differentiated version of \eqref{EquationNonlinearIBVP}, cf.~\eqref{EquationDefinitionSmp}.

\begin{definition}
 \label{DefinitionNonlinearSmpOperators}
 Let $J \subseteq \R$ be an open interval, $m \in \N$, 
 $\chi \in \mlpdpm{m,6}{G}{\mathcal{U}_\pm}$, and $\sigma\in \ml{m,6}{G}{\mathcal{U}_\pm}$. 
 We inductively define the operators
\begin{align*}
 S_{\chi, \sigma, G, m,p} \colon \clJ \times \Hpmadom{\max\{m,3\}} \times \mathcal{H}^{\max\{m,2\}}(G,\mathcal{U})  \rightarrow \Hpmhdom{m-p}
\end{align*}
by $S_{\chi, \sigma,G, m, 0,\pm}(t_0, f_\pm,u_{0,\pm}) = u_{0,\pm}$ and
\begin{align}
\label{EquationDefinitionHigherOrderInitialValuesNonlinear}
 &S_{\chi, \sigma,G, m,p,\pm }(t_0, f_\pm,u_{0,\pm})\\
 & \hspace{0.75cm}= \chi_{\pm}(u_{0,\pm})^{-1} \Big(\partial_t^{p-1}f_\pm(t_0) - \sum_{j=1}^3 {A}_j^{\operatorname{co}} \partial_j S_{\chi, \sigma,G,m,p-1,\pm}(t_0, f_\pm,u_{0,\pm}) \nonumber\\
    &\hspace{8em} - \sum_{l=1}^{p-1} \binom{p-1}{l} M_{1,\pm}^l(t_0, f_{\pm}, u_{0,\pm}) S_{\chi, \sigma,G,m,p-l,\pm}(t_0,f_\pm,u_{0,\pm}) \notag \\
   &\hspace{8em} - \sum_{l=0}^{p-1} \binom{p-1}{l} M_{2,\pm}^l(t_0, f_\pm, u_{0,\pm}) S_{\chi, \sigma,G,m,p-1-l,\pm}(t_0, f_\pm,u_{0,\pm})\Big), \nonumber\\
 &M_{k,\pm}^p = \sum_{1 \leq j \leq p} \sum_{\substack{\gamma_1,\ldots,\gamma_j \in \N_0^4 \setminus\{0\} \\ \sum \gamma_i = (p,0,0,0)}} \sum_{l_1, \ldots, l_j = 1}^6
	  C((p,0,0,0), \gamma_1, \ldots, \gamma_j)  \nonumber\\
	  &\hspace{8em} \cdot  (\partial_{y_{l_j}} \cdots \partial_{y_{l_1}} \theta_{k,\pm})(u_{0,\pm})
	      \prod_{i=1}^j S_{\chi,\sigma,G,m,|\gamma_i|,\pm}(t_0,f_\pm,u_{0,\pm})_{l_i } \label{EquationDefinitionMkp}
\end{align}
for $1 \leq p \leq m$, $k \in \{1,2\}$, where $\theta_1 = \chi$, $\theta_2 = \sigma$,  $M_{2,\pm}^0 = \sigma_\pm(u_{0,\pm})$,
and $C$ is a combinatorical constant, cf.\ Lemma~2.1 and (2.8) of \cite{SpitzQuasilinearMaxwell}. 
By $\mathcal{H}^{\max\{m,2\}}(G,\mathcal{U})$ we mean 
those functions $u_0 \in \Hpmhdom{\max\{m,2\}}$ with $\overline{\image u_{0,\pm}} \subseteq \mathcal{U}_{\pm}$.
\end{definition}
Lemma~2.4 of \cite{SpitzQuasilinearMaxwell} shows that the operators $S_{\chi, \sigma, G, m, p}$ indeed map into $\Hpmhdom{m-p}$ 
and it provides corresponding  estimates. (One applies it to the subsets $G_\pm$ separately.) Using
Lemma~2.1 of \cite{SpitzQuasilinearMaxwell}, we can differentiate \eqref{EquationNonlinearIBVP} $p$-times and obtain
\begin{equation}
	\label{EquationTimeDerivativesOfSolutionEqualSChiSigmaG}
	\partial_t^p u(t_0) = S_{\chi, \sigma, G, m,p}(t_0, f, u_0) \quad \text{for all } p \in \{0, \ldots, m\}
\end{equation}
if $u \in \Gpmdom{m}$ is a solution of~\eqref{EquationNonlinearIBVP} with data $f \in \Hpmadom{m}$,
$u_0 \in \Hpmhdom{m}$, and $g \in \Edom{m}$. Proceeding similarly with the interface and boundary condition,
equation \eqref{EquationTimeDerivativesOfSolutionEqualSChiSigmaG} leads to the identities
\begin{align}
\label{EquationNonlinearCompatibilityConditions}
	B_\Sigma S_{\chi, \sigma, G, m,p}(t_0, f, u_0) &= \partial_t^p g(t_0) \quad \text{on } \Sigma, \\
	B_{\partial G} S_{\chi, \sigma, G, m,p}(t_0, f, u_0) &= 0 \quad \text{on } \partial G
	  \qquad \text{for all \ } p \in \{0, \ldots, m-1\},\nonumber
\end{align}
which are necessary for the existence of a $\Gpmdom{m}$-solution of~\eqref{EquationNonlinearIBVP}. 
We say that the data tuple $(\chi, \sigma, t_0, B_\Sigma, B_{\partial G}, f, g, u_0)$ fulfills the \emph{compatibility conditions}
of order $m$ if $\overline{\image u_{0,\pm}} \subseteq \mathcal{U}_\pm$ and the equations~\eqref{EquationNonlinearCompatibilityConditions} are true.

\begin{rem}
	\label{RemarkChiuInFm}
	Analogously to Remark~1.2 in~\cite{SpitzMaxwellLinear}, the linear theory allows for coefficients 
	in $\mathcal{W}^{1,\infty}(J \times G)$ 
	whose derivatives up to order $m$ on $G_\pm$ are contained in $L^\infty(J, \Lpmtwohdom) + L^\infty(J \times G_\pm)$. In view of 
	Lemma~2.1 in \cite{SpitzQuasilinearMaxwell}, we can thus apply the linear theory with coefficients $\chi(\hat{u})$ and 
	$\sigma(\hat{u})$ and $\hat{u} \in \Gpmdomvar{\tilde{m}}$.
	However, the part of the derivatives in $L^\infty(J \times G)$ is easier to treat so that we concentrated on coefficients 
	from $\Fpmdom{m}$ in Sections~\ref{SectionAPrioriEstimates} and~\ref{SectionRegularity}. The same is true for the nonlinear problem. 
	In the proofs we will thus assume without loss of generality that $\chi$ and $\sigma$ from $\mlpm{m,6}{G}{\mathcal{U}_\pm}$ have 
	decaying space derivatives as $|x| \rightarrow \infty$.  More precisely, for 
	all multiindices $\alpha \in \N_0^9$ with $\alpha_4 = \ldots = \alpha_9 = 0$ and  $1 \leq |\alpha| \leq m$, $R > 0$, $\mathcal{U}_{1,\pm} \Subset \mathcal{U}_\pm$,
	and $v \in L^\infty(J, \Ltwohdom)$ with $\image v_\pm \subseteq \mathcal{U}_{1,\pm}$ and 
	$\|v\|_{L^\infty(J, \Ltwohdom)} \leq R$ we require
	\begin{align}
	\label{EquationPropertyForL2}
		&(\partial^\alpha \chi_\pm)(v_\pm), (\partial^\alpha \sigma_\pm)(v_\pm) \in L^\infty(J, L^2(G_\pm)), \nonumber\\
		&\|(\partial^\alpha \chi_\pm)(v_\pm)\|_{L^\infty(J, L^2(G_\pm))} + \|(\partial^\alpha \sigma_\pm)(v_\pm)\|_{L^\infty(J, L^2(G_\pm))} \leq C,
	\end{align}
	where $C = C(\chi,\sigma,m,R,\mathcal{U}_{1,\pm})$.
	With this assumption we obtain from Lemma~2.1 in \cite{SpitzQuasilinearMaxwell} that $\chi(\hat{u})$ and $\sigma(\hat{u})$ 
	belong to $\Fpmdom{m}$.
	
	Finally, we note that for unbounded $G$ the above considerations are unnecessary since then $\Lpmtwohdom + L^\infty(G_\pm) = \Lpmtwohdom$.
\end{rem}
The next lemma relates the maps $S_{\chi,\sigma,G,m,p}$ to their linear counterparts in~\eqref{EquationDefinitionSmp}. 

\begin{lem}
 \label{LemmaCorrespondenceLinearNonlinearInZero}
 Let $J \subseteq \R$ be an open interval, $t_0 \in \clJ$, and $m \in \N$ with $m \geq 3$.
 Take  $\chi \in \mlpdpm{m,6}{G}{\mathcal{U}_\pm}$ and $\sigma \in \mlpm{m,6}{G}{\mathcal{U}_\pm}$.
 Choose data $f \in \Hpmadom{m}$ and $u_0 \in \Hpmhdom{m}$ with $\overline{\image u_{0,\pm}}\subseteq\mathcal{U}_\pm$. 
 Let $r  > 0$. Assume that $f$ and $u_0$ satisfy
 \begin{align*}
  \begin{aligned}
    &\Hpmhndom{m}{u_0} \leq r, \quad &&\max_{0 \leq j \leq m-1} \Hpmhndom{m-j-1}{\partial_t^j f(t_0)} \leq r, \\
    &\Gpmdomnormwg{m-1}{f} \leq r, &&\Hpmandom{m}{f} \leq r.
  \end{aligned}
 \end{align*}
 
(1)	Let $\hat{u} \in \Gpmdomvar{m}$ with $\partial_t^p \hat{u}(t_0) = S_{\chi,\sigma,G,m,p}(t_0,f,u_0)$ 
	for $0 \leq p \leq m - 1$. Then $\hat{u}$ fulfills the equations
	\begin{align}
	    \label{EquationConnectionLinearAndNonlinearCompatibilityConditions}
	    S_{G,m,p}(t_0,\chi(\hat{u}),A_1^{\operatorname{co}}, A_2^{\operatorname{co}},A_3^{\operatorname{co}}, \sigma(\hat{u}), f, u_0) = S_{\chi,\sigma,G, m, p}(t_0, f, u_0)
	\end{align}
	for all $p \in \{0,\ldots,m\}$.

	(2) There  is a constant $C(\chi,\sigma,m,r,\mathcal{U}_{1,\pm}) > 0$ 
	and a function $u$ in $\Gpmdom{m}$ realizing the initial conditions 
	\begin{align*}
	  \partial_t^p u(t_0) = S_{\chi, \sigma, G, m, p}(t_0,f,u_0)
	\end{align*}
	for all $p \in \{0,\ldots,m\}$ and it is bounded by
	\begin{align*}
	    \Gpmdomnormwg{m}{u} \leq C(\chi,\sigma,m,r,\mathcal{U}_{1,\pm}) \Big(\sum_{j = 0}^{m-1} \Hpmhndom{m-j-1}{\partial_t ^j f(t_0)} + \Hpmhndom{m}{u_0} \Big).
	\end{align*}   
	Here $\mathcal{U}_{1,\pm}$ denote compact subsets of $\mathcal{U}_\pm$ with $\image u_{0,\pm} \subseteq \mathcal{U}_{1,\pm}$.
\end{lem}

\begin{proof}
Assertion~(1) can be shown  by induction using the definitions of the 
 operators $S_{G,m,p}$ in~\eqref{EquationDefinitionSmp} and  of 
 $S_{\chi,\sigma,G,m,p}$ in~\eqref{EquationDefinitionHigherOrderInitialValuesNonlinear},
 as well as Lemma~2.1 in \cite{SpitzQuasilinearMaxwell}.

 Since $S_{\chi, \sigma, G, m, p}(t_0,f,u_0)$ belongs to $\Hpmhdom{m-p}$ for all $p \in \{0, \ldots, m\}$, 
 an extension theorem (see e.g.\ Lemma~2.34 in~\cite{SpitzDissertation} applied on $G_+$ and $G_{-}$ separately)
 yields the existence of a function $u$ in $\Gpmdom{m}$ with   $\partial_t^p u(t_0) = S_{\chi, \sigma, G, m, p}(t_0,f,u_0)$ and
\begin{equation*}
 \Gpmdomnormwg{m}{u} \leq C \sum_{p = 0}^m \Hpmhndom{m-p}{S_{\chi, \sigma, G, m, p}(t_0,f,u_0)}
\end{equation*} 
for all $p \in \{0, \ldots, m\}$.  Lemma~2.4 of \cite{SpitzQuasilinearMaxwell} then implies assertion (2).
\end{proof}

 We introduce slightly strengthened assumptions on our material laws $\chi$ and $\sigma$ to guarantee that 
 $\chi(\hat{u})$ and $\sigma(\hat{u})$ converge at infinity, as required in 
 Theorem~\ref{TheoremExistenceAndUniquenessOnDomain}. 
 \begin{align*}
  &\mlwlpm{m,n}{G}{\mathcal{U}_\pm}{cv} = \{\theta \in \mlpm{m,n}{G}{\mathcal{U}_\pm} \colon \exists A \in \R^{n \times n} \text{ such that for all } \\
      &\hspace{8.5em} (x_k, y_k)_k \in (G \times \mathcal{U})^{\N} \text{ with } |x_k| \rightarrow \infty \text{ and } y_k \rightarrow 0: \\
      &\hspace{8.5em} \theta(x_k, y_k) \rightarrow A \text{ as } k \rightarrow \infty\}, \\
   &\mlpdwlpm{m,n}{G}{\mathcal{U}_\pm}{cv} = \mlpdpm{m,n}{G}{\mathcal{U}_\pm} \cap \mlwlpm{m,n}{G}{\mathcal{U}_\pm}{cv}.
 \end{align*}
  The space $\mlwlpm{m,n}{G}{\mathcal{U}_\pm}{cv}$ coincides with $\mlpm{m,n}{G}{\mathcal{U}_\pm}$ in \eqref{def:ml}
 if $G$ is bounded. 
 
 The next result provides the uniqueness of solutions of~\eqref{EquationNonlinearIBVP}. 
 Its proof is an obvious modification of Lemma~7.1 in \cite{SpitzQuasilinearMaxwell}
 and therefore omitted.
 \begin{lem}
 Let $t_0 \in \R, T > 0$, $J = (t_0, t_0 + T)$, and $m \in \N$ 
 with $m \geq 3$. Take material laws $\chi \in \mlpdwlpm{m,6}{G}{\mathcal{U}_\pm}{cv}$ 
 and $\sigma \in \mlwlpm{m,6}{G}{\mathcal{U}_\pm}{cv}$.
 Choose data $f \in \Hpmadom{m}$, $g \in \Edom{m}$, and $u_0 \in \Hpmhdom{m}$. 
 Let $u_1$ and $u_2$ be two solutions in $\Gpmdom{m}$ of~\eqref{EquationNonlinearIBVP} with 
 initial time $t_0$. Then $u_1 = u_2$.
\end{lem}

 We now show the basic local existence theorem for \eqref{EquationNonlinearIBVP} by
 a contraction argument. To close the argument, one has to take great care
 of the constants.  In particular, the 
 structure of the a priori estimate in Theorem~\ref{TheoremExistenceAndUniquenessOnDomain} is crucial here.
 \begin{theorem}
 \label{TheoremLocalExistenceNonlinear}
 Let $t_0 \in \R$, $T > 0$,  $J = (t_0,t_0 + T)$, and $m \in \N$ with $m \geq 3$. 
 Take $\chi \in \mlpdwlpm{m,6}{G}{\mathcal{U}_\pm}{cv}$ and $\sigma \in \mlwlpm{m,6}{G}{\mathcal{U}_\pm}{cv}$.
 Let $B_\Sigma$ and $B_{\partial G}$ be given by~\eqref{EquationDefinitionB}. Choose data $f \in \Hpmadom{m}$,  $g \in \Edom{m}$, and 
 $u_0 \in \Hpmhdom{m}$ with $\overline{\image u_{0,\pm}} \subseteq \mathcal{U}_\pm$
 such that the tuple $(\chi, \sigma, t_0, B_\Sigma, B_{\partial G}, f, g, u_0)$ fulfills 
 the nonlinear compatibility conditions~\eqref{EquationNonlinearCompatibilityConditions} of order $m$.
 Pick a radius $r > 0$ satisfying
 \begin{align}
 \label{EquationDataSmallerRadiusInLocalExistenceTheorem}
  &\sum_{j = 0}^{m-1} \Hpmhndom{m-1-j}{\partial_t^j f(t_0)}^2 + \Enormwgdom{m}{g}^2 + \Hpmhndom{m}{u_0}^2 + \Hpmandom{m}{f}^2 \leq r^2.
 \end{align}
 Take a number $\kappa > 0$ with 
 \begin{equation*}
  \dist(\{u_{0,\pm}(x) \colon x \in G_\pm\}, \partial \mathcal{U}_\pm) > \kappa.
 \end{equation*}
 Then there exists a time
 $\tau = \tau(\chi,\sigma,m,T,r,\kappa) > 0$ such that the nonlinear initial boundary value 
 problem~\eqref{EquationNonlinearIBVP} with data $f$, $g$, and $u_0$ 
 has a unique 
 solution $u$ on $[t_0, t_0 + \tau]$ which belongs to $\mathcal{G}_m(J_\tau \times G)$, 
 where $J_\tau = (t_0, t_0 + \tau)$.
\end{theorem}
\begin{proof}
Without loss of generality we assume $t_0 = 0$ and that~\eqref{EquationPropertyForL2} holds true 
for $\chi$ and $\sigma$, cf. Remark~\ref{RemarkChiuInFm}. Let $\tau \in (0,T]$. We set $J_\tau = (0, \tau)$ and 
\begin{equation}
\label{EquationDefinitionUkappa}
 \mathcal{U}_{\kappa,\pm} = \{y \in \mathcal{U}_\pm \colon \dist(y, \partial \mathcal{U}_\pm) \geq \kappa \} \cap \overline{B}_{2C_{\operatorname{Sob}} r}(0),
\end{equation}
where $C_{\operatorname{Sob}}$ is the norm of the Sobolev embedding $\Hpmhdom{2}\hookrightarrow L^\infty(G)$. 
The sets $\mathcal{U}_{\kappa,\pm}$ are compact and contain $\overline{\image u_{0,\pm}}$.

 Let $R > 0$. As in step~I of the proof of Theorem~3.3 in~\cite{SpitzQuasilinearMaxwell} one 
 checks that
 \begin{align*}
  B_R(J_\tau) := \{v \in \GpmdomvarP{m} \colon &\GpmdomnormwgP{m}{v} \leq R, \, \|v - u_0\|_{L^\infty(J_\tau \times G)} \leq \kappa/2, \\
  &\partial_t^j v(0) = S_{\chi,\sigma,G,m,j}(0,f,u_0) \text{ for } 0 \leq j \leq m-1\}
 \end{align*}
 is a complete metric space when endowed with $d(v_1,v_2) = \GpmdomnormwgP{m-1}{v_1 - v_2}$. It is non-empty
 thanks to Lemma~\ref{LemmaCorrespondenceLinearNonlinearInZero} and the choice of $R$ and $\tau$ below.
 
Let $\hat{u} \in B_R(J_\tau)$. We have $\chi \geq \eta$ for some $\eta > 0$.
The map $\chi(\hat{u})$ is contained in $\mathcal{F}_{m,\eta}^{\operatorname{cv}}(J_\tau \times G)$ 
and $\sigma(\hat{u})$ in $\mathcal{F}_{m}^{\operatorname{cv}}(J_\tau \times G)$ 
by Lemma~2.1 in \cite{SpitzQuasilinearMaxwell}, Remark~\ref{RemarkChiuInFm},
and Sobolev's embedding. Lemma~\ref{LemmaCorrespondenceLinearNonlinearInZero}
and the assumptions imply that the tuple 
$(t_0,\chi(\hat{u}), A_1^{\operatorname{co}}, A_2^{\operatorname{co}}, A_3^{\operatorname{co}},\sigma(\hat{u}),B_\Sigma, B_{\partial G},f,g,u_0)$ 
fulfills the linear compatibility conditions~\eqref{EquationCompatibilityConditionPrecised}.
Theorem~\ref{TheoremExistenceAndUniquenessOnDomain} then yields a solution $u \in \GpmdomP{m}$ of 
the linear sytem~\eqref{EquationIBVPIntroduction} with differential operator $L(\chi(\hat{u}),A_1^{\operatorname{co}}, A_2^{\operatorname{co}}, A_3^{\operatorname{co}}, \sigma(\hat{u}))$
and data $f$, $g$, and $u_0$. In this way one defines a mapping $\Phi \colon \hat{u} \mapsto u$ from $B_R(J_\tau)$ to $\GpmdomP{m}$. 
 We are now  looking for a radius $R>0$ and a (small) time $\tau>0$ such that 
 $\Phi$ leaves invariant  $B_R(J_\tau)$.

For this purpose take numbers $\tau \in (0,T]$ and $R > C_{\ref{LemmaCorrespondenceLinearNonlinearInZero}}( \chi, \sigma, m, r,\mathcal{U}_{\kappa,\pm}) (m+1) r$ 
which will be fixed below. Let $\hat{u} \in B_R(J_\tau)$. Lemma~2.4 in \cite{SpitzQuasilinearMaxwell} and 
\eqref{EquationDataSmallerRadiusInLocalExistenceTheorem} imply that
\begin{align}
\label{EquationBoundForLinearHigherOrderInitialValue}
 \Hpmhndom{m-p}{S_{\chi,\sigma,G,m,p}(0,f,u_0)} \leq C_{2.4,\text{\cite{SpitzQuasilinearMaxwell}}}(\chi,\sigma,m,r,\mathcal{U}_{\kappa,\pm})
\end{align}
for all $p \in \{0,\ldots,m\}$ and a constant $C_{2.4,\text{\cite{SpitzQuasilinearMaxwell}}}$. From Lemma~2.1 of  \cite{SpitzQuasilinearMaxwell} we infer
\begin{align*}
 \Fpmvarnormdom{m-1}{\chi(\hat{u})(0)}, \Fpmvarnormdom{m-1}{\sigma(\hat{u})(0)} \leq C_{2.1,\text{\cite{SpitzQuasilinearMaxwell}}}(\chi,\sigma,m,r,\mathcal{U}_{\kappa,\pm}),
\end{align*}
using \eqref{EquationDataSmallerRadiusInLocalExistenceTheorem} and  $\chi(\hat{u})(0) = \chi(u_0)$, for instance. 
Note that $\overline{\image \hat{u}_\pm}$ is contained in the compact set
\begin{equation*}
 \tilde{\mathcal{U}}_{\kappa,\pm} = \mathcal{U}_{\kappa,\pm} + \overline{B}(0,\kappa/2) \subseteq \mathcal{U}_\pm
\end{equation*}
as $\hat{u} \in B_R(J_\tau)$.
Lemma~2.1 in \cite{SpitzQuasilinearMaxwell} and estimate~\eqref{EquationBoundForLinearHigherOrderInitialValue} lead to the bounds
\begin{align*}
 \Hpmhndom{m-l-1}{\partial_t^l \chi(\hat{u})(0)}&\leq C_{2.1,\text{\cite{SpitzQuasilinearMaxwell}}}(\chi, m, \mathcal{U}_{\kappa,\pm}) (1 + \max_{0 \leq k \leq l} \Hpmhndom{m-k-1}{\partial_t^k \hat{u}(0)})^{m-1} \\
 &\hspace*{-1.5cm}= C_{2.1,\text{\cite{SpitzQuasilinearMaxwell}}}(\chi, m, \mathcal{U}_{\kappa,\pm}) (1 + \max_{0 \leq k \leq l} \Hpmhndom{m-k-1}{S_{\chi,\sigma,G,m,k}(0,f,u_0)})^{m-1} \\
 &\hspace*{-1.5cm}\leq C_{2.1,\text{\cite{SpitzQuasilinearMaxwell}}}(\chi, m, \mathcal{U}_{\kappa,\pm}) (1 + C_{2.4,\text{\cite{SpitzQuasilinearMaxwell}}}(\chi,\sigma,m,r,\mathcal{U}_{\kappa,\pm}))^{m-1}, \\
 \Hpmhndom{m-l-1}{\partial_t^l \sigma(\hat{u})(0)} &\leq C_{2.1,\text{\cite{SpitzQuasilinearMaxwell}}}(\sigma, m,  \mathcal{U}_{\kappa,\pm}) (1 + C_{2.4,\text{\cite{SpitzQuasilinearMaxwell}}}(\chi,\sigma,m,r,\mathcal{U}_{\kappa,\pm}))^{m-1}
\end{align*}
for all $l \in \{1,\ldots,m-1\}$. We thus find a radius $r_0 = r_0(\chi,\sigma,m,r,\kappa)$ such that
\begin{align*}
 &\max\{\Fpmvarnormdom{m-1}{\chi(\hat{u})(0)}, \max_{1 \leq l \leq m-1} \Hpmhndom{m-l-1}{\partial_t^l \chi(\hat{u})(0)}\} \leq r_0, \nonumber \\
 &\max\{\Fpmvarnormdom{m-1}{\sigma(\hat{u})(0)}, \max_{1 \leq l \leq m-1} \Hpmhndom{m-l-1}{\partial_t^l \sigma(\hat{u})(0)}\} \leq r_0.
\end{align*}
Since $\hat{u}$ belongs to $B_R(J_\tau)$, Lemma~2.1 in \cite{SpitzQuasilinearMaxwell} yields the inequality
\begin{align*}
 \Fpmnormdom{m}{\chi(\hat{u})}, \Fpmnormdom{m}{\sigma(\hat{u})} &\leq C_{2.1,\text{\cite{SpitzQuasilinearMaxwell}}}(\chi,\sigma,m,\tilde{\mathcal{U}}_{\kappa,\pm})(1+R)^m.
\end{align*}
Hence, there is a radius $R_1 = R_1(\chi, \sigma, m, R, \kappa)$ with 
\begin{align*}
 \Fpmnormdom{m}{\chi(\hat{u})} \leq R_1 \qquad \text{and} \qquad \Fpmnormdom{m}{\sigma( \hat{u})} \leq R_1.
\end{align*}

We next define the constant $C_{m,0} = C_{m,0}(\chi,\sigma,r,\kappa)$ by
\begin{align*}
 C_{m,0}(\chi,\sigma,r,\kappa) &= C_{\ref{TheoremExistenceAndUniquenessOnDomain},m,0}(\eta(\chi),r_0(\chi,\sigma,m,r,\kappa)),
\end{align*}
where $C_{\ref{TheoremExistenceAndUniquenessOnDomain},m,0}$ denotes the constant $C_{m,0}$ from Theorem~\ref{TheoremExistenceAndUniquenessOnDomain}.
The radius $R = R(\chi,\sigma,m,r,\kappa)$ for $B_R(J_\tau)$ is now fixed as
\begin{align}
 \label{EquationRadiusForFixedPointSpace}
 R= \max\Big\{\sqrt{6 \, C_{m,0}(\chi,\sigma,r,\kappa)}\, r, \, C_{\ref{LemmaCorrespondenceLinearNonlinearInZero}}( \chi, \sigma, m, r,\mathcal{U}_{\kappa,\pm})(m+1) r + 1\Big\}.
\end{align}
We further introduce the constants
\begin{align*}
 \gamma_m = \gamma_m(\chi, \sigma, T,r,\kappa) &:= \gamma_{\ref{TheoremExistenceAndUniquenessOnDomain},m}(\eta(\chi),R_1(\chi,\sigma,m,R(\chi,\sigma,m,r,\kappa),\kappa),T), \\ 
 C_m = C_m(\chi,\sigma,T,r) &:= C_{\ref{TheoremExistenceAndUniquenessOnDomain},m}(\eta(\chi), R_1(\chi,\sigma,m,R(\chi,\sigma,m,r,\kappa),\kappa),T), 
\end{align*}
where $\gamma_{\ref{TheoremExistenceAndUniquenessOnDomain},m}$ and 
$C_{\ref{TheoremExistenceAndUniquenessOnDomain},m}$ are the corresponding constants from Theorem~\ref{TheoremExistenceAndUniquenessOnDomain}.
Let $ C_{2.2,\text{\cite{SpitzQuasilinearMaxwell}}}(\theta,m, R, \tilde{\mathcal{U}}_{\kappa,\pm})$ be the constant that arises when applying  
Corollary~2.2 of \cite{SpitzQuasilinearMaxwell} to the components of $\theta \in \mlpm{m,6}{G}{\mathcal{U}_\pm}$.
We now define the parameter $\gamma = \gamma(\chi, \sigma, m, T, r,\kappa)$ and the time step 
$\tau = \tau(\chi, \sigma, m, T, r,\kappa)$  by
\begin{align}
 \gamma &= \max\Big\{\gamma_m, \,C_{m,0}^{-1} C_m \Big\}, \notag \\ 
  \tau &= \min\Big\{T, \, (2 \gamma + m C_{\ref{TheoremExistenceAndUniquenessOnDomain},1})^{-1} \log{2}, \, C_m^{-1} C_{m,0}, (2 C_{\operatorname{Sob}} R)^{-1} \kappa,
      \notag \\ 
      &\hspace{4em} [32  R^2 C_{m,0} C_{P}^2(C_{2.2,\text{\cite{SpitzQuasilinearMaxwell}}}^2(\chi,m, R, \tilde{\mathcal{U}}_\kappa) 
	+ C_{2.2,\text{\cite{SpitzQuasilinearMaxwell}}}^2(\sigma,m,R,\tilde{\mathcal{U}}_\kappa))]^{-1}\Big\}, \label{EquationDefinitionTauForFixedPointArgument}
\end{align}
where $C_{P}$ denotes the constant from Lemma~\ref{LemmaRegularityForA0}.

>From now on the reasoning follows the lines of steps~III)--V) of the proof of Theorem~3.3 in \cite{SpitzQuasilinearMaxwell}.
The above choice of constants and the linear results of our paper imply that $\Phi$ is a strict contraction on $B_R(J_\tau)$
which yields the assertion.
\end{proof}


\begin{rem}
 \label{RemarkNegativeTimes}
 Using time reversion and adapting coefficients and data accordingly, we can transfer
  the result of Theorem~\ref{TheoremLocalExistenceNonlinear} to the negative time direction, 
  cf.\ Remark~7.12 in~\cite{SpitzDissertation}.
 \end{rem}

We assume that the conditions of Theorem~\ref{TheoremLocalExistenceNonlinear} are valid and that the functions
$f$ and $g$ belong to the spaces
$\mathcal{H}^m((-T,T) \times G)$ respectively $E_m((-T,T) \times \Sigma)$,
 for  all $T > 0$.  We now define the \emph{maximal 
existence times} by 
\begin{equation}\label{def:max-times}\begin{split}
  T_+(m, t_0, f,g, u_0) &= \sup \{\tau \geq t_0 \colon \exists \, \mathcal{G}_m  \text{-solution of } \eqref{EquationNonlinearIBVP} 
      \text{ on } [t_0,  \tau]\}, \\
  T_{-}(m, t_0, f,g, u_0) &= \inf \{\tau \leq t_0 \colon \exists \, \mathcal{G}_m  \text{-solution of } \eqref{EquationNonlinearIBVP} 
      \text{ on } [\tau, t_0]\}.
 \end{split}\end{equation}
 The interval $(T_{-}(m,t_0,f,g,u_0) , T_+(m,t_0,f,g,u_0)) =: I_{max}(m,t_0,f,g,u_0)$ is called 
 the \emph{maximal interval of existence}. These notions are modified in a straightforward way if the inhomogeneities are given on an 
 open interval $J\subseteq \R$ with $t_0\in \clJ$.  By standard methods we can extend the solution 
 given by Theorem~\ref{TheoremLocalExistenceNonlinear} and Remark~\ref{RemarkNegativeTimes} to a 
 \emph{maximal solution} $u \in \bigcap_{j=0}^m C^j(I_{max}, \Hpmhdom{m-j})$ of~\eqref{EquationNonlinearIBVP}
 on $I_{max}$ which cannot be extended beyond this interval. More precisely, we obtain the following 
 basic blow-up criterion, cf.\ Lemma~4.1 of \cite{SpitzQuasilinearMaxwell}.
 
 \begin{prop}
   \label{LemmaBlowUpCriterion}
    Let $t_0 \in \R$ and $m \in \N$ with $m \geq 3$. 
 Take $\chi \in \mlpdwlpm{m,6}{G}{\mathcal{U}_\pm}{cv}$ and $\sigma \in \mlwlpm{m,6}{G}{\mathcal{U}_\pm}{cv}$.
 Choose data $f \in \mathcal{H}^m((-T,T) \times G)$, $g \in E_m((-T,T) \times \Sigma)$, and $u_0 \in \Hpmhdom{m}$ 
 for  all $T > 0$ and define $B_\Sigma$ and $B_{\partial G}$ as in~\eqref{EquationDefinitionB}. 
 Assume that the tuple $(\chi,\sigma, t_0, B_\Sigma, B_{\partial G}, f,g, u_0)$ fulfills the compatibility 
 conditions~\eqref{EquationNonlinearCompatibilityConditions} of order $m$.   Let $u$ be the maximal solution of~\eqref{EquationNonlinearIBVP} on $I_{max}$ 
   introduced above. If 
   $T_+ = T_+(m, t_0, f,g, u_0) < \infty$, then one of the following blow-up properties
   \begin{enumerate}
    \item
    $\liminf_{t \nearrow T_+} \dist(\{u_+(t,x) \colon x \in G_+\}, \partial \mathcal{U}_+) = 0$ \ or correspondingly for $u_{-}$,
    \item 
    $\lim_{t \nearrow T_+} \Hpmhndom{m}{u(t)} = \infty$
   \end{enumerate}
   occurs.
   The analogous result is true for $T_{-}(m, t_0, f, g, u_0)$.
  \end{prop}
 
  \section{Local well-posedness}
 \label{SectionMain}
   The blow-up criterion in Proposition~\ref{LemmaBlowUpCriterion} can be improved.
  By Theorem~\ref{TheoremLocalWellposednessNonlinear}, if $T_+ < \infty$ (and 
  the solution does not come arbitrarily close to $\partial \mathcal{U}_+$ or $\partial \mathcal{U}_{-}$), then the spatial Lipschitz norm 
  of the solution has to blow up as $t\to T_+$, see Theorem~\ref{TheoremLocalWellposednessNonlinear} below. 
Similar blow-up criteria have been established for several quasilinear hyperbolic systems both 
on the full space and on domains, see e.g.~\cite{BahouriCheminDanchin, BenzoniGavage, Klainerman,  Majda}.
For this improvement over the $\Hpmhdom{m}$-norm, one has to exploit that 
  a solution $u$ of the nonlinear problem~\eqref{EquationNonlinearIBVP} solves the linear problem~\eqref{EquationIBVPIntroduction} 
  with coefficients $\chi(u)$ and $\sigma(u)$, and then use  
   Moser-type estimates. Lemma~4.2 from~\cite{SpitzQuasilinearMaxwell} provides a version of these estimates suited to our 
   setting in which we admit space dependent nonlinearities. We can apply this lemma to the subdomains $G_\pm$ separately.

  The next proposition is the main step towards  the improved blow-up condtion.   In its proof one 
  differentiates~\eqref{EquationNonlinearIBVP} and applies the basic $L^2$-estimate \eqref{EquationEllerEstimateInL2} 
  to the derivative of $u$. For the tangential and time derivatives,  the Moser-type estimates allow us to treat the arising 
  inhomogeneities in such a way that the Gronwall lemma yields the desired estimate. 
  In order to bound the normal derivatives of $u$, we have to combine the above approach with 
  Proposition~\ref{PropositionCentralEstimateInNormalDirection}.  Once more the reasoning is parallel to that in 
  \cite{SpitzQuasilinearMaxwell}, making use of the linear results of the present paper. For details we thus refer to 
 the proof of Proposition~4.4 in \cite{SpitzQuasilinearMaxwell}.
  
  \begin{prop}
	\label{PropositionEstimateForNonlinearSolutionOnDomain}
	Let $m \in \N$ with $m \geq 3$ and $t_0 \in \R$. Take nonlinearities
	$\chi \in \mlpdwlpm{m,6}{G}{\mathcal{U}_\pm}{cv}$ and $\sigma \in \mlwlpm{m,6}{G}{\mathcal{U}_\pm}{cv}$. Let
   $B_{\Sigma}$ and $B_{\partial G}$ be defined as in~\eqref{EquationDefinitionB}.
    Choose data $u_0 \in \Hpmhdom{m}$, $g \in E_m((-T,T) \times \Sigma)$, 
   and $f \in \mathcal{H}^m((-T,T) \times G)$ 
   for all $T > 0$ such that the tuple $(\chi,\sigma,t_0, B_\Sigma, B_{\partial G}, f,g,u_0)$ fulfills the compatibility 
   conditions~\eqref{EquationNonlinearCompatibilityConditions}
   of order $m$.
   Let $u$ denote the maximal solution of~\eqref{EquationNonlinearIBVP} on $(T_{-},T_+)$. We introduce the quantity 
    \begin{align*}
     \omega(T) = \sup_{t \in (t_0,T)} \| u(t) \|_{\mathcal{W}^{1,\infty}(G)}
    \end{align*}
    for every $T \in (t_0,T_+)$. We further take $r > 0$ with 
    \begin{align*}
     &\sum_{j=0}^{m-1} \Hpmhndom{m-j-1}{\partial_t^j f(t_0)} \! + \!  \|g\|_{E_m((t_0, T_+) \times \Sigma)} \! + \! \Hpmhndom{m}{u_0} \! + \! \|f\|_{\mathcal{H}^m((t_0, T_+) \times G)} \leq r.
    \end{align*}
    We set $T^* = T_+$ if $T_+ < \infty$ and take any $T^* > t_0$ if $T_+ = \infty$. Let $\omega_0 > 0$ and let 
    $\mathcal{U}_{1,\pm}$ be compact subsets of $\mathcal{U}_\pm$.

    Then there exists a constant $C = C(\chi,\sigma,m,r,\omega_0,\mathcal{U}_{1,\pm},T^* - t_0)$ such that
    \begin{align*}
     \|u\|_{\mathcal{G}_m((t_0,T)\times G)}^2 &\leq C \Big( \sum_{j=0}^{m-1} \Hpmhndom{m-1-j}{\partial_t^j f(t_0)}^2  + \Hpmhndom{m}{u_0}^2 + \|g\|_{E_m((t_0, T) \times \Sigma)}^2 \\
     &\hspace{4em} + \|f\|_{\mathcal{H}^m((t_0, T) \times G)}^2 \Big)
    \end{align*}
    for all times $T \in (t_0,T^*)$ which have the property that  $\omega(T) \leq \omega_0$ and $\image u_\pm(t) \subseteq \mathcal{U}_{1,\pm}$ for all $t \in [t_0,T]$. 
    The analogous result is true on $(T_{-},t_0)$.
\end{prop}

The main missing part of the final local wellposedness theorem is the continuous dependence on initial data.
Here a loss of derivatives occurs since the difference of 
 two solutions satisfies an equation with a less regular right-hand side. 
The next lemma shows the core fact in this context. It improves the convergence of solutions $u_n$ by one level of regularity,
assuming uniform bounds of $u_n$ and  convergence of the data in the higher norm.
In the proof one uses that derivatives of the solutions satisfy a system with modified forcing terms.
These problems are then splitted in one  with fixed inhomogeneities (arising from the limit data) 
and  one with right-hand sides tending to 0 (up to to an error term treated in a Gronwall argument).
Such techniques were developed for the full space (see e.g.~\cite{BahouriCheminDanchin}). We combine this approach with our linear 
results to prevent a loss of normal regularity at the characteristic boundary. Here again the structure of Maxwell's equations 
is crucially used. The proof is  a combination of that of Lemma~5.2 in \cite{SpitzQuasilinearMaxwell} 
with the theorems of the previous sections. It is thus omitted.

  \begin{lem}
  \label{LemmaConvergenceOfNonlinearSolutionInGmProvidedConvergenceInGmminus1}
   Let $J' \subseteq \R$ be an open and bounded interval, $t_0 \in \overline{J'}$, and $m \in \N$ with $m \geq 3$.
   Take functions $\chi \in \mlpdwlpm{m,6}{G}{\mathcal{U}_\pm}{cv}$ and $\sigma \in \mlwl{m,6}{G}{\mathcal{U}_\pm}{cv}$. 
Let    $B_{\Sigma}$ and $B_{\partial G}$ be defined by~\eqref{EquationDefinitionB}.
   Choose data $f_n, f \in \mathcal{H}^m(J' \times G)$, $g_n, g \in E_m(J' \times \Sigma)$, 
   and $u_{0,n}, u_0 \in \Hpmhdom{m}$ 
   for all $n \in \N$ with 
   \begin{align*}
    \Hpmhndom{m}{u_{0,n} - u_0} \longrightarrow 0, \quad \|g_n - g\|_{E_m(J' \times \Sigma)} \longrightarrow 0, \quad \|f_n - f\|_{\mathcal{H}^m(J' \times G)} \longrightarrow 0,
   \end{align*}
   as $n \rightarrow \infty$.    We further assume that the system~\eqref{EquationNonlinearIBVP} 
   with data $(t_0, f_n, g_n, u_{0,n})$ and $(t_0, f, g, u_0)$ has $\mathcal{G}_m(J' \times G)$-solutions
   $u_n$ and $u$ for all $n \in \N$, that there are compact subsets $\tilde{\mathcal{U}}_{1,\pm}$ of $\mathcal{U}_\pm$ 
   with $\image u_\pm(t) \subseteq \tilde{\mathcal{U}}_{1,\pm}$ for all $t \in J'$, that $(u_n)_n$ is bounded in $\mathcal{G}_m(J' \times G)$, 
   and that $(u_n)_n$ converges to $u$ in $\mathcal{G}_{m-1}(J' \times G)$.
   Then the functions $u_n$ tend to $u$ in $\mathcal{G}_m(J' \times G)$.
  \end{lem}

  Finally, we can prove the full local wellposedness theorem. 
  In the following we will write $B_{M}(x,r)$ for the ball of radius $r$ around a point $x$ from a
  metric space~$M$. For times $t_0 < T$ we further define the data space
  \begin{align*}
     &M_{\chi,\sigma,m}(t_0,T) = \{(\tilde{f},\tilde{g}, \tilde{u}_0) \in \mathcal{H}^m((t_0,T) \times G) \times E_m((t_0,T) \times \Sigma) \times \Hpmhdom{m} \colon \\
	&\hspace{12.9em} (\chi, \sigma, t_0,B_\Sigma, B_{\partial G}, \tilde{f}, \tilde{g}, \tilde{u}_0) \text{ is compatible of order } m\}, 
	\end{align*}
and endow it with the metric
\begin{align*}
d((\tilde{f}_1,\, &\tilde{g}_1, \tilde{u}_{0,1}), (\tilde{f}_2, \tilde{g}_2, \tilde{u}_{0,2})) \\
      & = \max \{\|\tilde{f}_1 - \tilde{f}_2 \|_{\mathcal{H}^m((t_0,T) \times G)}, \|\tilde{g}_1 - \tilde{g}_2\|_{E_m((t_0, T) \times \Sigma)}, \Hpmhndom{m}{\tilde{u}_{0,1} - \tilde{u}_{0,2}}\}.
    \end{align*}

  \begin{theorem}
   \label{TheoremLocalWellposednessNonlinear}
   Let $m \in \N$ with $m \geq 3$ and fix $t_0 \in \R$. Take functions $\chi \in \mlpdwlpm{m,6}{G}{\mathcal{U}_\pm}{cv}$ and 
   $\sigma \in \mlwlpm{m,6}{G}{\mathcal{U}_\pm}{cv}$. Let
   $B_{\Sigma}$ and $B_{\partial G}$ be defined by~\eqref{EquationDefinitionB}.
   Choose data $u_0 \in \Hpmhdom{m}$, $g \in E_m((-T,T) \times \Sigma)$, 
   and $f \in \mathcal{H}^m((-T,T) \times G)$ 
   for all $T > 0$ such that $\overline{\image u_{0,\pm}} \subseteq \mathcal{U}_\pm$ and 
   the tuple $(\chi,\sigma,t_0,B_\Sigma, B_{\partial G},f,g,u_0)$ fulfills the compatibility 
   conditions~\eqref{EquationNonlinearCompatibilityConditions}   of order $m$.
   
   Then the maximal existence times $T_\pm = T_\pm(m, t_0, f,g, u_0)$ from \eqref{def:max-times} 
    do not depend on $k \in \{3, \ldots, m\}$. Moreover, the following assertions are true.
   \begin{enumerate}[leftmargin = 2em]
    \item  There exists a unique maximal solution $u$ of~\eqref{EquationNonlinearIBVP}
     which belongs to the function space $\bigcap_{j = 0}^m C^{j}((T_{-}, T_+), \Hpmhdom{m-j})$.
    \item  If $T_+ < \infty$, then
	\begin{enumerate}
	 \item the restriction $u_+$ leaves every compact subset of $\mathcal{U}_+$ or $u_{-}$ leaves 
	 every compact subset of $\mathcal{U}_{-}$, or
	 \item  $\limsup_{t \nearrow T_+} \max\{\|\nabla u_+(t)\|_{L^{\infty}(G_+)}, \|\nabla u_{-}(t)\|_{L^{\infty}(G_{-})}\} = \infty$.
	\end{enumerate}
	The analogous result holds for $T_{-}$.
    \item 
    Fix $T \in (t_0, T_+)$ and take $T'\in (T,T_+)$. Then there is a number 
    $\delta > 0$ such that for all data $(\tilde{f}, \tilde{g},  \tilde{u}_0) \in B_{M_{\chi,\sigma,m}(t_0, T')}((f,g,u_0), \delta)$ 
    the maximal existence time satisfies $T_+(m, t_0, \tilde{f},\tilde{g}, \tilde{u}_0) > T$. 
    We denote by $u(\cdot; \tilde{f}, \tilde{g}, \tilde{u}_0)$ the corresponding maximal 
    solution of~\eqref{EquationNonlinearIBVP}. The flow map
    \begin{align*}
     \Psi\colon B_{M_{\chi,\sigma,m}(t_0, T')}((f,g,u_0), \delta)  &\rightarrow \mathcal{G}_m((t_0,T) \times G) , \quad
     (\tilde{f},\tilde{g}, \tilde{u}_0) \longmapsto u(\cdot; \tilde{f}, \tilde{g}, \tilde{u}_0), 
    \end{align*}
    is continuous, and  there is a constant $C = C(\chi, \sigma,m,r,T_+ - t_0, \kappa_0)$ such that
    \begin{align*}
     &\|\Psi(\tilde{f}_1, \tilde{g}_1, \tilde{u}_{0,1}) - \Psi(\tilde{f}_2, \tilde{g}_2, \tilde{u}_{0,2})\|_{\mathcal{G}_{m-1}((t_0,T) \times G)} \nonumber \\
     &\leq C \sum_{j = 0}^{m-1} \Hpmhndom{m-j-1}{\partial_t^j \tilde{f}_1(t_0) - \partial_t^j \tilde{f}_2(t_0)} + C \|\tilde{g}_1 - \tilde{g}_2\|_{E_{m-1}((t_0,T) \times \Sigma)} \nonumber\\
     &\qquad + C \Hpmhndom{m}{\tilde{u}_{0,1} - \tilde{u}_{0,2}} +C \| \tilde{f}_1 - \tilde{f}_2 \|_{\mathcal{H}^{m-1}((t_0,T) \times G)}
    \end{align*}
    for all $(\tilde{f}_1, \tilde{g}_1, \tilde{u}_{0,1}), (\tilde{f}_2, \tilde{g}_2, \tilde{u}_{0,2}) \in B_{M_{\chi,\sigma,m}(t_0, T')}((f,g,u_0), \delta)$, 
    where $\kappa_0 = \dist(\image u_{0,\pm}, \partial \mathcal{U}_\pm)$.
    The analogous result is true for $T_{-}$.
   \end{enumerate}
  \end{theorem}
  \begin{proof}[Sketch of the proof]
  We note that in part (3) one may extend $\tilde f$ and $\tilde g$ to the time interval $\R$ to be in the framework of the previous parts
  of the theorem. Except for part (3), the assertions easily follow from  Propositions~\ref{LemmaBlowUpCriterion}  and 
  \ref{PropositionEstimateForNonlinearSolutionOnDomain}. In the context of part (3) we set  $\tilde u= u(\cdot; \tilde{f}, \tilde{g}, \tilde{u}_0)$.
 If this solution exists on an interval $[t_0,t']$ with $\mathcal{G}_m$--norm less than $R'$,
 Theorem~\ref{TheoremExistenceAndUniquenessOnDomain} and the results of Section 2 in 
 \cite{SpitzQuasilinearMaxwell} allow us to bound $u-\tilde u$ in $\mathcal{G}_{m-1,\gamma}((t_0,t') \times G)$ by analogous norms of the differences
 of the data, if $\gamma(R')$ is large enough. We next use a time step $\tau$ as in \eqref{EquationDefinitionTauForFixedPointArgument} and a radius
 $R$ as in \eqref{EquationRadiusForFixedPointSpace} in the proof of Theorem~\ref{TheoremLocalExistenceNonlinear}, where we have fixed a sufficiently large
 radius $r>0$ for the data. If $\delta>0$ is small enough, this theorem then yields a solution $\tilde u$ of~\eqref{EquationNonlinearIBVP}
 in $\mathcal{G}_m((t_0,t+\tau)\times G)$ with norm less or equal $R$, for  data $(\tilde{f}, \tilde{g}, \tilde{u}_0)$.
 Using  the bound in $\mathcal{G}_{m-1,\gamma}((t_0,t') \times G)$ just mentioned and 
 Lemma~\ref{LemmaConvergenceOfNonlinearSolutionInGmProvidedConvergenceInGmminus1}, we obtain the continuity of the flow map on 
 $\mathcal{G}_m((t_0,t+\tau)\times G)$. Decreasing $\delta>0$ if necessary, one can then deduce assertion (3) iteratively.
 The details are analogous to the proof of Theorem~5.3 in  \cite{SpitzQuasilinearMaxwell} which only uses different linear results.
  \end{proof}

 \section{Appendix}
 In this appendix we show that the interface conditions for  $\bs{D}$ and $\bs{B}$ are preserved.
 
 \begin{lem}
  \label{LemmaConservationOfInterfaceConditions}
  Let $t_0, T \in \R$ with $t_0 < T$ and set $J = (t_0, T)$. 
Let  $(\bs{E}, \bs{H}, \bs{D}, \bs{B})$ in $C(\clJ, \mathcal{H}^1(G)) \cap C^1(\clJ, L^2(G))$ be a solution of the 
  Maxwell system~\eqref{EquationMaxwellSystem} with $\bs{J} \in L^2(J, \mathcal{H}(\div,G))$ and $\bs{J}_\Sigma \in L^2(J, H(\div,\Sigma))$ 
  satisfying $[\bs{E} \times \nu] = 0$ and 
  $[\bs{H} \times \nu] = \bs{J}_{\Sigma}$ on $J \times \Sigma$. Set $\rho_\Sigma(t) = \rho_{\Sigma,0} - 
  \int_{t_0}^t (\div_{\Sigma} \bs{J}_{\Sigma} - [\bs{J} \cdot \nu])(s) ds$ for all $t \in \clJ$.
  \begin{enumerate}
   \item \label{ItemConservationInterfaceConditionB} If $[\bs{B} \cdot \nu](t_0) = 0$ on $\Sigma$, then $[\bs{B} \cdot \nu] = 0$ on $J \times \Sigma$.
   \item \label{ItemConservationInterfaceConditionD} If $[\bs{D} \cdot \nu](t_0) = -\rho_{\Sigma,0}$, then $[\bs{D} \cdot \nu] = -\rho_\Sigma$ on $J \times \Sigma$.
  \end{enumerate}
 \end{lem}
 \begin{proof}
  \ref{ItemConservationInterfaceConditionB} Since $\partial_t \bs{B}_{\pm}$ belongs to $H(\div, G_{\pm})$, 
  these fields have a normal trace in $H^{-1/2}(\Sigma)$ for each $t \in \clJ$. Employing that also 
  $\curl \bs{E}_{\pm}\in H(\div, G_{\pm})$, we compute
  \begin{align*}
\langle \partial_t [&\bs{B} \cdot \nu](t), \phi \rangle_{H^{-1/2}(\Sigma) \times H^{1/2}(\Sigma)} 
   =  \langle  [\partial_t \bs{B} \cdot \nu](t), \phi \rangle_{H^{-1/2}(\Sigma) \times H^{1/2}(\Sigma)}  \\
   &= \langle  [- \curl \bs{E} \cdot \nu](t), \phi \rangle_{H^{-1/2}(\Sigma) \times H^{1/2}(\Sigma)} \\
   &= \langle  - \curl \bs{E}_+(t) \cdot \nu, \phi \rangle_{H^{-1/2}(\Sigma) \times H^{1/2}(\Sigma)} 
     + \langle  \curl \bs{E}_{-} (t) \cdot \nu, \phi \rangle_{H^{-1/2}(\Sigma) \times H^{1/2}(\Sigma)} \\
   &= - \int_{G_+} \div \curl \bs{E}_+(t) \phi \, dx - \int_{G_+} \curl \bs{E}_+(t) \cdot \nabla \phi \, dx - \int_{G_{-}} \div \curl \bs{E}_{-}(t) \phi \, dx \\
   &\quad - \int_{G_{-}} \curl \bs{E}_{-}(t) \cdot \nabla \phi \, dx\\
   &= - \int_{G_+} \bs{E}_+(t) \cdot \curl \nabla \phi \, dx + \langle  \bs{E}_+(t) \times \nu, \nabla \phi \rangle_{H^{-1/2}(\Sigma) \times H^{1/2}(\Sigma)} \\
   &\quad - \int_{G_{-}} \bs{E}_{-}(t) \cdot \curl \nabla \phi \, dx + \langle  \bs{E}_{-}(t) \times (-\nu), \nabla \phi \rangle_{H^{-1/2}(\Sigma) \times H^{1/2}(\Sigma)} \\
   &= \langle  [\bs{E} \times \nu](t), \nabla \phi \rangle_{H^{-1/2}(\Sigma) \times H^{1/2}(\Sigma)} = 0
  \end{align*}
  for all $t \in J$ and $\phi \in C_c^\infty(G)$. Since 
  $\tr_{\Sigma} H^1_0(G) = H^{1/2}(\Sigma)$, we infer that $\partial_t [\bs{B} \cdot \nu] = 0$ on 
  $\clJ \times G$. As $[\bs{B} \cdot \nu](t_0) = 0$ on $\Sigma$, we arrive at $[\bs{B} \cdot \nu] = 0$ 
  on $J \times \Sigma$.
  
  \ref{ItemConservationInterfaceConditionD}  We proceed as in part (1).
  Using the assumptions on $\bs{J}$, we  compute
  \begin{align*}
   \langle \partial_t [&\bs{D} \cdot \nu](t), \phi \rangle_{H^{-1/2}(\Sigma) \times H^{1/2}(\Sigma)} 
   =  \langle  [\partial_t \bs{D} \cdot \nu](t), \phi \rangle_{H^{-1/2}(\Sigma) \times H^{1/2}(\Sigma)}  \\
   &= \langle  [(\curl \bs{H} - \bs{J}) \cdot \nu](t), \phi \rangle_{H^{-1/2}(\Sigma) \times H^{1/2}(\Sigma)} \\
   &= - \langle  [\bs{J} \cdot \nu](t), \phi \rangle_{H^{-1/2}(\Sigma) \times H^{1/2}(\Sigma)} 
      - \langle  [\bs{H} \times \nu](t), \nabla \phi \rangle_{H^{-1/2}(\Sigma) \times H^{1/2}(\Sigma)} \\
   &= - \langle  [\bs{J} \cdot \nu](t), \phi \rangle_{H^{-1/2}(\Sigma) \times H^{1/2}(\Sigma)} 
      - \langle  \bs{J}_\Sigma(t), \nabla \phi \rangle_{H^{-1/2}(\Sigma) \times H^{1/2}(\Sigma)} 
  \end{align*}
  for all $\phi \in C_c^\infty(G)$ and almost all $t \in J$. Since $J_{\Sigma} = [\bs{H} \times \nu]$, 
  the boundary current density $\bs{J}_\Sigma$ is tangent to $\Sigma$, i.e., $\bs{J}_\Sigma = \pi_\Sigma \bs{J}_\Sigma$, 
  where $\pi_\Sigma = \pi_{\Sigma,x}$ denotes the orthogonal projection on the tangent space at $x \in \Sigma$. We infer that
  \begin{align*}
   &\langle  \bs{J}_\Sigma(t), \nabla \phi \rangle_{H^{-1/2}(\Sigma) \times H^{1/2}(\Sigma)} 
   = \langle  \pi_\Sigma \bs{J}_\Sigma(t), \pi_\Sigma \nabla \phi \rangle_{H^{-1/2}(\Sigma) \times H^{1/2}(\Sigma)} \\
   &= \langle  \bs{J}_\Sigma(t), \nabla_\Sigma \phi \rangle_{H^{-1/2}(\Sigma) \times H^{1/2}(\Sigma)} = - \langle \div_{\Sigma} \bs{J}_\Sigma(t), \phi \rangle_{H^{-1/2}(\Sigma) \times H^{1/2}(\Sigma)},
  \end{align*}
  where we refer to  Definition~2.2  of  $\nabla_\Sigma$  and $\div_\Sigma$ in~\cite{Cessenat}.
  We conclude that
  \begin{align*}
   \langle \partial_t [\bs{D} \cdot \nu](t), \phi \rangle_{H^{-1/2}(\Sigma) \times H^{1/2}(\Sigma)} 
   = \langle \div_\Sigma \bs{J}_\Sigma(t) -  [\bs{J} \cdot \nu](t), \phi \rangle_{H^{-1/2}(\Sigma) \times H^{1/2}(\Sigma)}
  \end{align*}
 for all $\phi \in C_c^\infty(G)$ and almost all $t \in J$. Arguing as in~\ref{ItemConservationInterfaceConditionB}, 
 we derive claim (2).
 \end{proof}

 \bibliographystyle{plain}
\bibliography{maxwell-inter}

\end{document}